\documentclass[10pt,reqno]{amsart}
\usepackage[utf8]{inputenc}
\usepackage{lmodern}
\usepackage[T1]{fontenc}
\usepackage{amsmath,amssymb}
\usepackage{hyperref}
\usepackage{mathrsfs}
\usepackage{amsfonts}
\usepackage{mathtools}
\usepackage{tikz}
\usepackage{lmodern}
\usepackage{bigints}
\usepackage{relsize}
\usepackage{bbm}
\usepackage{amsthm}
\usepackage{mathtools}
\mathtoolsset{showonlyrefs}
\usepackage{geometry}
\usepackage{scalerel,stackengine}
\stackMath
\newcommand\reallywidehat[1]{%
\savestack{\tmpbox}{\stretchto{%
  \scaleto{%
    \scalerel*[\widthof{\ensuremath{#1}}]{\kern-.6pt\bigwedge\kern-.6pt}%
    {\rule[-\textheight/2]{1ex}{\textheight}}%WIDTH-LIMITED BIG WEDGE
  }{\textheight}% 
}{0.5ex}}%
\stackon[1pt]{#1}{\tmpbox}%
}
\numberwithin{equation}{subsection}
\newcommand{\norm}[1]{\left\lVert#1\right\rVert}
\newtheorem{theorem}{Theorem}[section]
\newtheorem{corollary}[theorem]{Corollary}
\newtheorem{lemma}[theorem]{Lemma}
\newtheorem{proposition}[theorem]{Proposition}
\newtheorem{definition}[theorem]{Definition}

\newtheorem{remark}[theorem]{Remark}
\theoremstyle{definition}

\DeclareMathOperator{\I}{Im}

\DeclareMathOperator{\Ra}{Range}

\newcommand{\cgc}[1]{\color{red}  {\tt [GC: #1]} \color{black}  }

\title{Scattering Theory and dispersive estimates for general $1$d Charge Transfer Models}
\author[G. Chen]{Gong Chen$\textsuperscript{\textdagger}$}
\author[A. Moutinho]{Abdon Moutinho$^{\ast}$}

\email{gc@math.gatech.edu}
\email{aneto8@gatech.edu}

\address{School of Mathematics, Georgia Institute of Technology, Atlanta, GA 30332, USA}
\thanks{$^\ast$Corresponding author; e-mail: aneto8@gatech.edu.}
\thanks{\textsuperscript{\textdagger}GC  was partially supported by NSF grant DMS-2350301 and by Simons foundation MP-TSM00002258.}
\date{\today}

\begin{document}
\begin{abstract}
We continue our study of scattering theory and dispersive properties for one-dimensional charge transfer models, namely linear Schr\"odinger equations with multiple moving potentials. By the discovery of a refined structure of  the construction of distorted Fourier transforms adapted to the multi-potential framework, we remove the large-velocity separation assumption imposed in \cite{dispa}. This work thus completes the full scattering theory and dispersive analysis for general one-dimensional charge transfer models. These dispersive estimates provide the foundation for analyzing asymptotic stability and collision phenomena for multi-solitons in a general setting.
\end{abstract}
\maketitle
\tableofcontents

\section{Introduction}
Motivated by the study of dynamics of multi-solitons for the one–dimensional nonlinear Schr\"odinger equation (NLS), fixed $m\in\mathbb{N}$, we consider  the following linear matrix Schr\"odinger equation with a time-dependent charge transfer Hamiltonian
\begin{equation}\label{p}\tag{CTM}
    \begin{aligned}
& i \partial_t \vec{\psi}+\left(\begin{array}{cc}
\partial_x^2 & 0 \\
0 & {-}\partial_x^2
\end{array}\right) \vec{\psi}+\sum_{j=1}^m V_j\left(t \right) \vec{\psi}=0 \text{,  $(t,x)\in\mathbb{R}\times\mathbb{R}$} \\
& \left.\vec{\psi}\right|_{t=0}=\vec{\psi}_0\in L^{2}_{x}(\mathbb{R},\mathbb{C}^{2}),
\end{aligned}
\end{equation}
where  the functions $V_j$ are matrix potentials of the form
$$
V_j(t)=\left(\begin{array}{cc}
U_j(x-v_jt-y_j) & -e^{i \theta_j(t, x)} W_j(x-v_jt-y_j) \\
e^{-i \theta_j(t, x)} W_j(x-v_jt-y_j) & -U_j(x-v_jt-y_j)
\end{array}\right),
$$
with 
$$\omega_j, \gamma_j \in \mathbb{R} \,\,\text{with }\,\omega_j > 0$$
\begin{equation}\label{eq:thetaj}
    \theta_j(t, x)=\left(\left|v_j\right|^2+\omega_j\right) t+2 x v_j+\gamma_j 
\end{equation}
and $v_j$ are distinct velocities listed as
\begin{equation}\label{eq:velocity}
    v_{1}>v_{2}>...>v_{m}.
\end{equation}
We assume that the real-valued functions $U_j(x)$ and $W_j(x)$ exhibit rapid decay, as is guaranteed by the decay of the underlying solitons in the nonlinear setting. Indeed, the system \eqref{p} arises as the  linearization of NLS around a multi–soliton background, with each center $j$ corresponding to one traveling soliton.

To the decay estimates of solutions of \eqref{p} is crucial to the study of the asymptotic stability of multi-solitons for Schr\"odinger models. This approach was taken in the papers \cite{Dispesti} and \cite{nlsstates}  in $d=3$ and in \cite{Perelman4} to prove the asymptotic stability of multi-solitons for the nonlinear Schr\"odinger equation $d\geq 3$.
For $d=1$, in \cite{perelmanasym}, Perelman proved the asymptotic stability of two fast solitons for a large set of one-dimensional nonlinear Schr\"odinger equations. Also see \cite{CJlinear,CJnonlinear,GCwave,GCwavecmp} in the settings of Klein-Gordon equations and wave equations respectively. 
We refer to, for instance,  \cite{Cai,dispa,Phaseanal,Dispesti,Geocharge,Yajimachannels,Zielinski1} for more details on the background and historical references on charge transfer models including their physical applications.

Focusing on the one-dimensional problems, in Perelman's work \cite{perelmanasym}, under the large-velocity separation assumption, for two  generic potentials with only $0$ as their discrete spectrum, the asymptotic completeness and dispersive estimates were obtained. Note that in \cite{perelmanasym}, many constructions are specific for $m=2$. In particular, one of the most important points  is to identify the scattering part of solutions via so-called the dispersive map.  The construction and analysis of the dispersive map in the two-potential setting  is much more simplified due to the symmetry. In \cite{dispa}, we revisited Perelman's work and we systematically developed the scattering theory and established dispersive estimates under the assumption that the potentials move at significantly different velocities, even \emph{in the presence of unstable modes}. In particular, we constructed the dispersive map, see Definition \ref{s0def}, in the general settings and obtained refined mapping properties of it. After detailed analysis of the dispersive map, we proved the existence of wave operators, asymptotic completeness, and pointwise decay of solutions, \emph{without requiring the absence
of threshold resonances}.

In this paper, we continue our study of one–dimensional charge–transfer models, extending our earlier work~\cite{dispa}. By refining the construction of distorted Fourier transforms adapted to the multi–potential setting, we uncover structural identities that yield the scattering theory and dispersive estimates \emph{without} the large–velocity separation assumed in~\cite{dispa}, thereby completing the scattering and dispersive analysis for general one–dimensional charge–transfer models. A key outcome is the invertibility of the dispersive map in several function spaces without the large speed–gap assumption. This is achieved by iterating suitable linear maps and analyzing the resulting structure carefully. Building on these properties of the dispersive map, we obtain canonical decompositions of solutions into a dispersive component and moving discrete modes. We emphasize that, in the absence of large–speed separation, exponential growth associated with distinct traveling unstable modes will cause complicated interaction among each other; nevertheless, we establish a version of asymptotic completeness and the existence of wave operators in the stable space, together with dispersive estimates in the scattering space. These bounds serve as cornerstones for the asymptotic stability of multi–solitons and for global–in–time analyses of low–speed soliton–soliton and soliton–potential collisions; see, for instance,~\cite{holmer2012phase,holmer2007soliton,holmer2007slow,holmer2008soliton,Collisionls,WSsimulation} and references therein.  We conclude our introduction by pointing out that removing the large-separation conditions on speeds not only completes  scattering  and dispersive theory for general charge transfer models, it also opens the door to study dispersive properties for multiple moving potential problems  and asymptotic stability of multi-solitons in relativistic models.

\subsection{Organization and notations}
Before moving on to the main results,  we briefly outline the structure of the article and introduce some of the notation used throughout.
\subsubsection{ Organization}
In the later part of this section, we introduce some basic notations and then state our main results including their extensions to the scalar models. In Section \ref{sec:linearmaps},  we study the precise structures of linear maps of two different forms, which will be crucial for us to study dispersive maps. The invertibility of the dispersive map and its dispersive properties are analyzed in Section \ref{sec:dispersivemap}. With these preparations, in Section \ref{sec:decomp}, we show the decompositions of functions in terms of dispersive maps and moving discrete modes. Then in Section \ref{sec:stabscat}, we show the decomposition of solutions in the stable space  and dispersive properties of solutions in the scattering space. Finally in Section \ref{prooftunst2}, special solutions are constructed for given scattering behaviors and moving discrete modes. In some sense, this is the existence of wave operators in the center-stable space.
\subsubsection{Notations}
Throughout this article, in various places, we use $\Diamond$ to denote dummy variables.

As usual, “$A := B$” or “$B =: A$” is the definition of $A$ by means of the expression
$B$.  

We use $\langle \Diamond \rangle:=\sqrt{1+\Diamond^2},$ $p=\begin{bmatrix}
    1 & 0\\
    0 & 0
\end{bmatrix},$ and $q=\begin{bmatrix}
    0 & 0\\
    0 & 1
\end{bmatrix}.$

$\chi_A$ for some set $A$ is always denoted as a smooth indicator function adapted to the set $A$. 

Throughout, we use $u_t=\partial_t u:=\frac{\partial}{\partial_t}u$ and $u_x=\partial_x u:=\frac{\partial}{\partial_x} u$.

For non-negative $X$, $Y$, we write $X\lesssim Y$ if $X \leq C$ ,
and we use the notation $X\ll Y$ to indicate that the implicit constant should be regarded as small.
Furthermore, for nonnegative $X$ and arbitrary $Y$, we use the shorthand notation $Y = \mathcal{O}(X)$ if
$|Y|\leq C X$. Moreover, for any Banach space $B,$ we say that a function $f\in L^{2}_{x}(\mathbb{R},B)$ satisfies $f=O_{L^{2}}(M)$ for a parameter $M>0,$ if there exists a positive constant $C$ such that $\norm{f}_{L^{2}_{x}(\mathbb{R},B)}\leq C M.$   

\noindent {\it Inner products.}
In terms of the $L^2$ inner product of complex-valued functions, we use
\begin{equation}\label{eq:L2inner}
    \langle f, g\rangle = \int_\mathbb{R} f\overline{g}\, dx.
\end{equation} 
Given two pairs of complex-valued vector functions $\vec{f} = (f_1, f_2)$ and $\vec{g} = (g_1, g_2)$, their  inner product is given by
\begin{equation}\label{eq:L2L2inner}
    \langle \vec{f}, \vec{g}\rangle:=  \int_\mathbb{R} \bigl( f_1 \overline{g_1} + f_2 \overline{g_2} \bigr) \, d x.
\end{equation}
For $n\in\mathbb{N}$, we define the space $\mathcal{F}L^1$ as
\begin{equation*}
    \mathcal{F}L^1\coloneqq \left\{f:\mathbb{R}\to\mathbb{C}\vert\, \norm{f}_{\mathcal{F}L^1}=\norm{\hat{f}(x)}_{L^{1}_{x}(\mathbb{R})}<{+}\infty\right\}.
\end{equation*}
 For any $k\in\mathbb{R},$ the integer part of $k$ is denoted by
 \begin{equation*}
     \lfloor k\rfloor=\max_{h\in\mathbb{Z},h\leq k\,}h.
 \end{equation*}
 Given any function $f$ and a real number $y$, we define the shift of it as
 \begin{equation}
     \tau_y f (\Diamond):= f(\Diamond-y).
 \end{equation}
\subsection{Main results}
In this subsection, we introduce assumptions on potentials, then state the main results in this paper after introducing basic notations from the scattering theory.

\subsubsection{Assumptions on potentials}\label{subsub:assumption}

Fix a $m\in\mathbb{N}$,  we denote
\begin{equation}
     [m]:=\{1,\,2\,,...\,,m\}.
\end{equation}
Motivated by nonlinear applications, we will consider the following potentials.  First, for each $\ell\in\{1,2,\,...,\,m\},$ we consider $U_{\ell},\,W_{\ell}:\mathbb{R}\to\mathbb{R}$ be fast decaying functions and, for $\omega_\ell>0,$ the following operators
\begin{align}\label{eq:H0V}
\mathcal{H}_{0,\omega_\ell}:=\left(\begin{array}{cc}
-\partial_{x x}+\omega_\ell & 0 \\
0 & \partial_{x x}-\omega_\ell
\end{array}\right), V_{\ell}(x)=\left(\begin{array}{cc}
U_{\ell} & -W_{\ell} \\
W_{\ell} & -U_{\ell}
\end{array}\right), \mathcal{H}_{\ell}:=\mathcal{H}_{0,\omega_{\ell}}+V_{\ell}, 
\end{align}
such that  $V_{\ell}$  well as all its derivatives are exponentially decay\footnote{This decay assumption could be weakened. We use the current decay assumption since we can directly refer to the existence of Jost functions in \cite{Busper1,KriegerSchlag}. Actually, the symmetric assumption is imposed for the same reason. Overall, so long as one can show the existence of Jost solutions which asymptotically behave like free waves at $\pm\infty$ depending on normalizations, our argument can always work.}: for some $0<\gamma<1$
\begin{equation}\label{decV}
\left\|V^{(k)}_{\ell}(x)\right\| \leq C_{k,\ell} e^{-\gamma|x|} \quad \forall k \geq 0,    
\end{equation}
and $V_{\ell}(x)=V_{\ell}({-}x)$ for all $x\in\mathbb{R},$ and $\ell\in\{1,2,\,...,\,N\}.$ 
Using
\begin{equation}\label{eq:sigmas}
\sigma_{3}=
    \begin{bmatrix}
     1     & 0\\
     0 & {-}1
    \end{bmatrix},\, \sigma_{2}=
    \begin{bmatrix}
     0     & 1\\
     {-}1 & 0
    \end{bmatrix},\,
\sigma_{1}= \begin{bmatrix}
     0     & 1\\
     1 & 0
    \end{bmatrix},
\end{equation}    
 it is not difficult to verify the following identities
$$
\sigma_3 \mathcal{H}^{*}_{\ell} \sigma_3=\mathcal{H}_{\ell}, \sigma_1 \mathcal{H}_{\ell} \sigma_1=-\mathcal{H}_{\ell}.
$$
\par Next, we consider the following hypotheses for our main results on the asymptotic completeness.
\begin{itemize}
    \item [(H1)] There is no embedded eigenvalue in the essential spectrum of each operator $\mathcal{H}_{\ell}.$ Recall that the essential spectrum of $\mathcal{H}_{\ell}$ is given by $\sigma_{e}\mathcal{H}_{\ell}= ({-}\infty,-\omega_\ell]\cup [\omega_\ell,{+}\infty).$
    \item [(H2)] Each operator $\mathcal{H}_{\ell}$ has $2N_\ell+2M_\ell$ non-zero simple eigenvalues for some $N_{\ell},\,M_{\ell}\in\mathbb{Z}_{\geq 0}.$  
    There are $2N_{\ell}$  non-zero simple eigenvalues on the real line with absolute value less than $\omega_\ell$, and $2M_{\ell}$  non-zero simple eigenvalues on the imaginary line.\footnote{The simplicity of eigenvalues are motivated by nonlinear applications. Our proof does not depend on this simplicity assumption and actually can be easily adapted to the setting with finite dimensional eigenspaces.}
    \item [(H3)] Finally,  we also assume that $0$ is  an eigenvalue of $\mathcal{H}_\ell$, $0\in\sigma_{d}\mathcal{H}_{\ell}$ and $\ker \mathcal{H}_{\ell}^{n}=\ker\mathcal{H}_{\ell}^{2}$ for all  $\ell,$  and every natural number  $n\geq 2.$ 
   
\end{itemize}
The three conditions above will be  standing assumptions throughout this paper, 
and we will not mention them further.

\subsubsection{Notations from scattering theory}
For notional convenience, we list the \emph{non-zero} eigenvalues of $\mathcal{H}_\ell$ as $\{\lambda_{j,\ell}\}_{j=1}^{2N_\ell+2M_\ell}$, and the discrete eigenvalues are given by
\begin{equation}
    \sigma_d\mathcal{H}_\ell=\{0\}\bigcup\{\lambda_{j,\ell}\}_{j=1}^{2N_\ell+2M_\ell}.
\end{equation}
For each $\lambda_{j,\ell}$, by the simplicity assumption, $\dim (\ker (\mathcal{H}_\ell-\lambda_{j,\ell}\mathrm{Id}))=1$ and there is a normalized eigenfunction $\vec{Z}_{j,\ell}$ such that
\begin{equation}\label{eq:nonzerodiscretemodes}
    \mathcal{H}_\ell \vec{Z}_{j,\ell}= \lambda_{j,\ell} \vec{Z}_{j,\ell}.
\end{equation}
For the generalized kernel of $\mathcal{H}_\ell$, we assume that $\dim (\ker \mathcal{H}_\ell)=K_{\ell,1}$ and  $\dim (\ker \mathcal{H}^2_\ell)=K_{\ell,2}$ with $K_{\ell,2}\geq K_{\ell,1}\geq 1$. Moreover, we can find the following sets as the bases for  $\ker \mathcal{H}_\ell$ and $\ker \mathcal{H}^2_\ell$
\begin{equation}\label{eq:generalizedkernel}
 \ker \mathcal{H}_\ell=\mathrm{span}\{\vec{Z}^0_{j,\ell}\}_{j=1}^{K_{\ell,1}},\quad\ker \mathcal{H}^2_\ell=\mathrm{span}\{\vec{Z}^1_{j,\ell}\}_{j=1}^{K_{\ell,2}}.     
\end{equation}
We also set that for $1\leq j\leq K_{\ell,1}$,
\begin{equation}\label{eq:generalizedkernel1}
    \vec{Z}^0_{j,\ell}=\vec{Z}^1_{j,\ell}.
\end{equation}
We also record the notations:
\begin{equation}\label{Pd}\tag{Discrete spectrum space}
P_{d,\ell}=%\Raa
\text{ projection onto the discrete spectrum of $\mathcal{H}_{\ell},$}
\end{equation}
\begin{equation}\label{Pe}\tag{Essential spectrum space}
P_{e,\ell}=%\Raa
\text{ projection onto the essential spectrum of $\mathcal{H}_{\ell}.$}
\end{equation}
For the essential spectrum part, the condition $(\mathrm{H}1)$ imply the existence of the following generalized eigenfunctions $\mathcal{F}_{\ell}(x,k),\,\mathcal{G}_{\ell}(x,k)$ of $\mathcal{H}_{\ell}$ in $L^{\infty}_{x}(\mathbb{R})$ satisfying the following asymptotic behavior for some $\gamma>0.$
\begin{align}\label{asy1}
    \mathcal{G}_\ell(x,{-}k)=&\overline{s_\ell(k)}\left[e^{ikx}
    \begin{bmatrix}
        1\\
        0
    \end{bmatrix}+O\left(\frac{e^{\gamma x}}{(1+\vert k \vert)}\right)\right] \text{, as $x\to {-}\infty,$}\\ \label{asy2}
    \mathcal{G}_\ell(x,{-}k)=&e^{ikx}
    \begin{bmatrix}
        1\\
        0
    \end{bmatrix}+\overline{r_\ell(k)}e^{{-}ikx}
    \begin{bmatrix}
        1\\
        0
    \end{bmatrix}+O\left(\frac{e^{{-}\gamma x}}{(1+\vert k \vert)}\right)
    \text{, as $x\to {+}\infty,$}\\ \label{asy3}
    \mathcal{F}_\ell(x,k)=& s_\ell(k)\left[e^{ikx}
    \begin{bmatrix}
        1\\
        0
    \end{bmatrix}+O\left(\frac{e^{{-}\gamma x}}{(1+\vert k \vert)}\right)\right] \text{, as $x\to {+}\infty,$}\\ \label{asy4}
    \mathcal{F}_\ell(x,k)=&e^{ikx}
    \begin{bmatrix}
        1\\
        0
    \end{bmatrix}+r_\ell(k)e^{{-}ikx}
    \begin{bmatrix}
        1\\
        0
    \end{bmatrix}+O\left(\frac{e^{\gamma x}}{(1+\vert k \vert)}\right)
    \text{, as $x\to {-}\infty.$}
\end{align}
Based on Lemmas $4.13$ and $5.13$ of \cite{dispa}, we record the following important estimates.
\begin{lemma}\label{lem:reftranscoeff}For $r_{\ell}(k),s_{\ell}(k):\mathbb{R}\to \mathbb{C}$, one has for any $n\in\{0,1\}$
\begin{equation}\label{decayrs1}
  \left\vert \frac{d^{n}r_{\ell}(k)}{dk^{n}}\right\vert +\left\vert \frac{d^{n}}{dk^{n}}[1-s_{\ell}(k)] \right\vert=O\left(\frac{1}{(1+\vert k\vert)^{n+1}}\right).
\end{equation}
\end{lemma}
Using the generalized eigenfunctions above, we define the following linear maps 
\begin{align}\label{Ghat}
\hat{G}_{\ell}\left(\begin{bmatrix}
    f_{1}(k)\\
       f_{2}(k)
    \end{bmatrix}\right)(x)\coloneqq \int_{\mathbb{R}}\mathcal{G}_{\ell}(x,k)\frac{f_{1}(k)}{s_{\ell}(k)}\,dk+\int_{\mathbb{R}}\sigma_{3}\mathcal{G}_{\ell}(x,k)\frac{f_{2}(k)}{s_{\ell}(k)}\,dk\in L^{2}_{x}(\mathbb{R},\mathbb{C}^{2})\\ \label{Fhat}
\hat{F}_{\ell}\left(\begin{bmatrix}
g_{1}(k)\\
        g_{2}(k)
    \end{bmatrix}\right)(x)\coloneqq \int_{\mathbb{R}}\mathcal{F}_{\ell}(x,k)\frac{g_{1}(k)}{s_{\ell}(k)}\,dk+\int_{\mathbb{R}}\sigma_{3}\mathcal{F}_{\ell}(x,k)\frac{g_{2}(k)}{s_{\ell}(k)}\,dk \in L^{2}_{x}(\mathbb{R},\mathbb{C}^{2}),
\end{align}
where  $(f_{1},f_{2}),\, (g_{1},g_{2})\in L^{2}_{k}(\mathbb{R},\mathbb{C}^{2}).$ We recall the basic properties of these maps which are proved in \cite{dispa}.
\begin{lemma}
 The domains of  $\hat{G}_{\ell}$ and $\hat{F}_{\ell}$ which are dense in $L^{2}_{k}(\mathbb{R},\mathbb{C}^{2})$. We also have
 \begin{equation}
P_{e,\ell}\mathcal{H}_{\ell}=\Ra \hat{G}_{\ell}=\Ra \hat{F}_{\ell}.
 \end{equation}
\end{lemma}
Next, we define the distorted Fourier operators for each operator $\mathcal{H}_{\ell}.$

%It was already checked in \cite{dispa} that $$ in the paper \cite {dispa}. Moreover, the existence of domains forwas also checked in \cite{dispa}.   
\begin{definition}
Let $\sigma_{1}$ be the matrix
\begin{equation*}
    \sigma_{1}=\begin{bmatrix}
        0 & 1\\
        1 & 0
    \end{bmatrix}.
\end{equation*}
The operators $G^{*}_{\ell},\,F^{*}_{\ell}:L^{2}_{x}(\mathbb{R},\mathbb{C}^{2})\to L^{2}_{k}(\mathbb{R},\mathbb{C}^{2})$ are defined for each $\ell\in [m]$ by
\begin{align*}
   F^{*}_{\ell}(\vec{u})(k)\coloneqq & \frac{1}{\sqrt{2\pi}}\int_{\mathbb{R}} \begin{bmatrix}
        \mathcal{F}^{t}_{\ell}(x,{-}k)\\
        \left(\sigma_{1}\mathcal{F}_{\ell}(x,{-}k)\right)^{t}
    \end{bmatrix}\overrightarrow{u}(x)\,dx,\\
    G^{*}_{\ell}(\vec{u})(k)\coloneqq & \frac{1}{\sqrt{2\pi}}\int_{\mathbb{R}} \begin{bmatrix}
        \mathcal{G}^{t}_{\ell}(x,{-}k)\\
        \left(\sigma_{1}\mathcal{G}_{\ell}(x,{-}k)\right)^{t}
    \end{bmatrix}\overrightarrow{u}(x)\,dx,
\end{align*}
for any function $\vec{u}\in L^{2}_{x}(\mathbb{R},\mathbb{C}^{2}).$ The function
\begin{equation*}
    \begin{bmatrix}
        \mathcal{F}^{t}_{\ell}(x,{-}k)\\
        \left(\sigma_{1}\mathcal{F}_{\ell}(x,{-}k)\right)^{t}
    \end{bmatrix}\overrightarrow{u}(x)
\end{equation*}
is a matrix product of a element in $\mathbb{C}^{2\times 2}$ with the vector $\overrightarrow{u}(x) \in \mathbb{C}^{2}.$ 
\end{definition}
Moreover, the operators $F^{*}_{\ell}$ and $G^{*}_{\ell}$ satisfy the following proposition which is an analogue version of the inverse Fourier transform theorem.
\begin{lemma}\label{leper}
We have the following identities in the frequency side
\begin{align*}
\sigma_{3}F^{*}_{\omega}\sigma_{3}\hat{G}_{\omega}=&\mathrm{Id},\\
\sigma_{3}G^{*}_{\omega}\sigma_{3}\hat{F}_{\omega}=&\mathrm{Id}.
\end{align*}
Moreover, if $\lambda\in\sigma_{d}\mathcal{H}_{\omega}$ and $\vec{v}(x)\in L^{2}_{x}(\mathbb{R},\mathbb{C}^{2})\cap \ker \left[\mathcal{H}_{\omega}-\lambda \mathrm{Id}\right]^{2},$ then
\begin{equation}\label{Fnull}
F^{*}_{\omega}\sigma_{3}\vec{v}=G^{*}_{\omega}\sigma_{3}\vec{v}=0.
\end{equation}
%\cgc{I do not understand this}
%where $\omega>0$ satisfies $[{-}\infty,{-}\omega]\cup [\omega,{+}\infty]\subset \sigma_{c}(\mathcal{H}_{\omega}).$
\end{lemma}
\begin{proof}
    The proof of Lemma \ref{leper} is explained in \cite{dispa}.
\end{proof}
\par With notations above,  now we introduce  the \emph{dispersive map} $\mathcal{S}$ defined in \cite{dispa}.
\begin{definition}\label{s0def}
Given $v_{1}>v_{2}>...>v_{m},\,\delta y_{\ell}=y_{\ell-1}-y_{\ell}\gg 1,$ for  any given
\begin{equation*}
\vec{\phi}(k)= 
    \begin{bmatrix}
     \phi_{1}(k)\\
     \phi_{2}(k)
    \end{bmatrix}\in L^{2} (\mathbb{R},\mathbb{C}^{2}),
\end{equation*}
we define the following formula
\begin{align}\label{eq:Sphi}
   \mathcal{S}(\vec{\phi})(t,x):= & \sum_{\ell=1}^{m}e^{i\left(\frac{v_{\ell}x}{2}-\frac{v_{\ell}^{2}t}{4}+\omega_{\ell}t+\gamma_{\ell}\right)\sigma_{3}}\hat{G}_{\omega_{\ell}}\left(
   e^{{-}it(k^{2}+\omega_{\ell})\sigma_{3}}e^{{-}i\gamma_{\ell}\sigma_{3}} \begin{bmatrix}
e^{iy_{\ell}k}\phi_{1,\ell}\left(k+\frac{v_{\ell}}{2}\right)\\
       e^{iy_{\ell}k}\phi_{2,\ell}\left(k-\frac{v_{\ell}}{2}\right)
    \end{bmatrix}\right)(x-y_{\ell}-v_{\ell}t)\\
    & {-}\frac{1}{\sqrt{2 \pi}}\int_{\mathbb{R}}e^{{-}it k^{2}\sigma_{3}}
    \begin{bmatrix}
       \varphi_{1}(k)\\
       \varphi_{2}(k)
    \end{bmatrix} e^{ikx}\,dk,\nonumber
\end{align}
where the sequence $\{\overrightarrow{\phi_{\ell}}\}_{\ell=1}^m$ and $\vec{\varphi}$ are constructed recursively from $\vec{\phi}(k)$ via the following conditions
\begin{enumerate}
    \item [a)] $\begin{bmatrix}
        \phi_{1,1}(k)\\
        \phi_{2,1}(k)
    \end{bmatrix}=\vec{\phi}(k); $
    \item [b)] for each $\ell\geq 1,$
    \begin{equation*}
      e^{{-}i\gamma_{\ell+1}\sigma_{3}} \begin{bmatrix}
           \phi_{1,\ell+1}(k)\\
           \phi_{2,\ell+1}(k)
       \end{bmatrix}
        =e^{{-}i\gamma_{\ell}\sigma_{3}}
        \begin{bmatrix}
        \frac{\phi_{1,\ell}\left(k\right)-r_{\ell}\left(k-\frac{v_{\ell}}{2}\right)e^{{-}i2y_{\ell}(k-\frac{v_{\ell+1}}{2})+iy_{\ell}(v_{\ell}-v_{1+\ell})}\phi_{1,\ell}\left({-}k+v_{\ell}\right)}{s_{\ell}\left(k-\frac{v_{\ell}}{2}\right)}\\
        \frac{\phi_{2,\ell}\left(k\right)-r_{\ell}\left(k+\frac{v_{\ell}}{2}\right)e^{{-}2iy_{\ell}(k+\frac{v_{\ell+1}}{2})+iy_{\ell}(v_{\ell+1}-v_{\ell})}\phi_{2,\ell}\left({-}k-v_{\ell}\right)}{s_{\ell}\left(k+\frac{v_{\ell}}{2}\right)}
    \end{bmatrix};
    \end{equation*}
    \item [c)] and
    \begin{equation*}
        \begin{bmatrix}
          \varphi_{1}(k)\\
          \varphi_{2}(k)
        \end{bmatrix}=\sum_{\ell=1}^{m-1}\begin{bmatrix}
           \phi_{1,\ell}(k)\\
           \phi_{2,\ell}(k)
       \end{bmatrix}.
    \end{equation*}

\end{enumerate} 
 For the convenience of notations, we use $S(t)$ to denote
\begin{equation}\label{eq:St}
    \mathcal{S}(t)\vec{\phi}:= \mathcal{S}(\vec{\phi})(t,x).
\end{equation}
\end{definition}

\subsubsection{Main results}
\iffalse
We first introduce the indispensable tool in the study of the charge transfer model, Galilei transformations. 
\begin{definition}[Galilei Transformation]
Considering $g:\mathbb{R}\to\mathbb{C}^{2},$ the function $\mathfrak{g}_{\omega_{k},v_{k},y_{k}}(g)(t,x)$ is defined by
\begin{equation}\label{eq:Gali}
\mathfrak{g}_{\omega_{k},v_{k},y_{k},\gamma_{k}}(g)(t,x)=e^{i\sigma_{3}(\frac{v_{k}x}{2}-\frac{tv^{2}}{4}+\gamma_{k}+\omega_{k}t)}g\left(x-v_{k}t-y_{k}\right).
\end{equation}
\end{definition} 
\fi
Our first main result is a decomposition in terms of $\mathcal{S}$ and discrete modes introduced in \eqref{eq:generalizedkernel} and \eqref{eq:nonzerodiscretemodes}. 
\begin{proposition}%[General Asymptotic Completeness]
\label{princ}
%Assume that all hypotheses $(H1),\,(H2),\,(H3),\,(H4)$ are true and that all potentials $V_{\ell}$ satisfy \eqref{decV}. 
There exists $L>0$ depending on the potentials $V_{\ell}$ such that if$$\min_{\ell}y_{\ell}-y_{\ell+1}>L,$$ then $\Ra \mathcal{S}(\Diamond)(0,x)$ is a closed subspace of $L^{2}_{x}(\mathbb{R},\mathbb{C}^{2}),$ and every function $\vec{f}(x)\in L^{2}_{x}(\mathbb{R},\mathbb{C}^{2})$ has a unique representation of the form
\begin{align}\label{princ11}
   f(x)&=\mathcal{S}
(\vec{\phi})(t,x)+\sum_{\ell=1}^{m}\sum_{j=1}^{2N_\ell+2M_\ell}a_{j,\ell}e^{i\sigma_{3}\left(\frac{v_{\ell}x}{2}-\frac{v_{\ell}^{2}t}{4}+\omega_{\ell}t+\gamma_{\ell}\right)}\vec{Z}_{j,\ell}(x-v_{\ell}t-y_{\ell})\\
&
%+\sum_{\ell=1}^{m}\sum_{j=1}^{K_{\ell,1}}e^{i\sigma_{3}\left(\frac{v_{\ell}x}{2}-\frac{v_{\ell}^{2}t}{4}+\omega_{\ell}t+\gamma_{\ell}\right)}\vec{Z}^0_{j,\ell}(x-v_{\ell}t-y_{\ell})
+\sum_{\ell=1}^{m}\sum_{j=1}^{K_{\ell,2}}a_{j,\ell}^{1}e^{i\sigma_{3}\left(\frac{v_{\ell}x}{2}-\frac{v_{\ell}^{2}t}{4}+\omega_{\ell}t+\gamma_{\ell}\right)}\vec{Z}^1_{j,\ell}(x-v_{\ell}t-y_{\ell})
\end{align}
for some $\vec{\phi}$ in the domain of $\mathcal{S}$, and some complex numbers $a_{j,\ell}$ and $a_{j,\ell}^1$.  %where, for each $\ell\in [m],$ the set $\{\vec{Z}_{j,\lambda_{\ell}}\}_{j\in [\dim (\ker-\mathrm{Id}\lambda_{\ell})^{2}]}$ is an orthonormal basis on $L^{2}_{x}(\mathbb{R},\mathbb{C}^{2})$ of
%\begin{equation*}
 %   \ker(\mathcal{H}_{\ell}-\lambda_{\ell}Id)^{2}. 
%\end{equation*}
\end{proposition}
\begin{remark}
The decomposition above  can be thought as  the  extension of  Theorem $1.13$ from \cite{dispa} for all cases having $\min_{\ell}v_{\ell}-v_{\ell+1}>0,$ including the ones satisfying $\vert \min_{\ell}v_{\ell}-v_{\ell+1} \vert\ll 1.$ In contrast to the problem studied in \cite{dispa}, in the current setting, we can not use the large separation conditions for velocities to defeat the exponential instability of unstable modes, so we could not obtain a decomposition like Theorem $1.13$ from \cite{dispa} in terms of solutions to \eqref{p}.
\end{remark}

\noindent {\bf Stable models:}
Without the large-speed separation condition, the exponential instability from different unstable modes will mess up with each other, whence, to study asymptotic completeness,  it is natural to consider stable models. In the absence of unstable modes, we still have a similar result to  Theorem $1.13$ from \cite{dispa}.

\begin{theorem}[Asymptotic completeness for stable models]\label{stablemodel}
In addition to hypotheses $(\mathrm{H}1),\,(\mathrm{H}2)$ and $(\mathrm{H}3),$  assume  that
\begin{equation}\label{eq:stablecond}
    M_\ell=0\,\,\text{in}\,\,(\mathrm{H}2),
\end{equation}
%$$ 
then there exists $L>0$  depending on the potentials $V_{\ell}$ such that
%\begin{equation*}
 %   \sigma_{d}\mathcal{H}_{\ell}\cap \{z\in\mathbb{C}\vert\,\I z\neq 0\}=\emptyset,
%\end{equation*}
$$\min_{\ell}y_{\ell}-y_{\ell+1}>L>1,$$  every solution $\vec{\psi}$ of \eqref{p} has a unique representation of the form
\begin{align}\label{princ11stablemodel}
   \vec{\psi}(t,x)&=\mathcal{T}
\left(\overrightarrow{\phi_{0}}\right)(t,x) +\sum_{\ell=1}^{m}\sum_{j=1}^{2N_\ell+2M_\ell}a_{j,\ell}\mathfrak{G}_{\ell}(\vec{Z}_{j,\ell})(t,x)
+\sum_{\ell=1}^{m}\sum_{j=1}^{K_{\ell,2}}a_{j,\ell}^{1}\mathfrak{G}_{\ell}(\vec{Z}^1_{j,\ell})(t,x)
\end{align}
for some $\vec{\phi_0}$ in the domain of $\mathcal{S}$, and some complex numbers $a_{j,\ell}$ and $a_{j,\ell}^1$. In the decomposition above,
\begin{itemize}
    \item $\mathcal{T}
\left(\overrightarrow{\phi_{0}}\right)(t,x)$ is the unique solution of \eqref{p} satisfying
\begin{equation}\label{tphidecayy}
\lim_{t\to{+}\infty}\norm{\mathcal{T}
\left(\overrightarrow{\phi_{0}}\right)(t,x)-\mathcal{S}
\left(\overrightarrow{\phi_{0}}\right)(t,x)}_{L^{2}_{x}(\mathbb{R})}=0;
\end{equation}
\item The function $\mathfrak{G}_{\ell}(\vec{Z}_{j,\ell})(t,x)$ is the unique solution of \eqref{p} satisfying
\begin{equation}\label{tg1}
    \lim_{t\to{+}\infty}\norm{\mathfrak{G}_{\ell}(\vec{Z}_{j,\ell})(t,x)-e^{{-}i(\lambda_{j,\ell}-\omega_{\ell})t}e^{i\theta_{j}(t,x)\sigma_{3}}\vec{Z}_{j,\ell}(x-v_{\ell}t-y_{\ell})}_{L^{2}_{x}(\mathbb{R})}=0
\end{equation}
where $\theta_j$ is given by \eqref{eq:thetaj};
\item The function $\mathfrak{G}_{\ell}(\vec{Z}^1_{j,\ell})(t,x)$ is the unique solution of \eqref{p} satisfying
\begin{equation}\label{tg2}
\lim_{t\to{+}\infty}\norm{\mathfrak{G}_{\ell}(\vec{Z}^{1}_{j,\lambda_{\ell}})(t,x)-e^{i\theta_{j}(t,x)\sigma_{3}}\vec{Z}^1_{j,\ell}(x-v_{\ell}t-y_{\ell})-te^{i\theta_{j}(t,x)\sigma_{3}}\vec{Y}_{j,\ell}(x-v_{\ell}t-y_{\ell})}_{L^{2}_{x}(\mathbb{R})}=0,
\end{equation}
where $\mathcal{H}_{\ell}Z^1_{j,\ell}=Y_{j,\ell}.$
\end{itemize}
Moreover, for all $t\geq 0,$ there exists a projection $$P_{c}(t):L^{2}_{x}(\mathbb{R},\mathbb{C}^{2})\to \Ra \mathcal{T}(t)\subset L^{2}_{x}(\mathbb{R},\mathbb{C}^{2}) $$
satisfying
\begin{equation*}
  \langle \vec{\psi}(t)-P_{c}(t)\vec{\psi}(t),\sigma_{3}\mathfrak{G}_{\ell}(\vec{Z}_{j,\ell})(t,x) \rangle=  \langle \vec{\psi}(t)-P_{c}(t)\vec{\psi}(t),\sigma_{3}\mathfrak{G}_{\ell}(\vec{Z}^1_{j,\ell})(t,x) \rangle=0,
\end{equation*}
%for any $\mathfrak{v}_{\omega_{\ell},\lambda_{\ell,n}}\in \ker \left( \mathcal{H}_{\ell}-\lambda_{\ell}Id\right)^{2},$
and
$
    P_{c}(t)\vec{\psi}(t,x)=\mathcal{T}(\vec{\phi}_{0})(t,x),
$
for all $t\geq 0.$
\end{theorem}
\begin{remark}
 Theorem \ref{stablemodel} is a corollary of Theorem \ref{stablecase}. Moreover, Theorem \ref{stablecase} covers the result of Theorem \ref{stablemodel} and deals with the cases where  \eqref{eq:stablecond}  does not  hold.   
\end{remark}
\noindent {\bf Unstable models:} As indicated above, when unstable modes appear, exponential instabilities will mess up dynamics. If we stay away from unstable modes, even in the unstable models, we can construct special solutions converging to scattering solutions, moving stable modes, and zero modes. In some sense, the following proposition gives the existence of wave operators away from unstable modes.
\begin{proposition}\label{remark4}
 There is an small constant $\epsilon\in (0,1),$ and for any $\vec{\phi}\in$ Domain of $\mathcal{S}$
there exist a solution $\mathcal{T}(\vec{\phi})(t)$ for any $\vec{\phi}\in$ Domain of $\mathcal{S},$ and solutions $\mathfrak{G}_{\omega_{\ell},v_{\ell},y_{\ell}}(\mathfrak{v}_{\omega_{\ell},\lambda_{\ell,n}})(t,x)$ and for any $\lambda_{\ell,n}$ with $\I\lambda_{\ell,n}<0$ or $\lambda_{\ell,n}=0$ satisfying \eqref{tphidecayy}, \eqref{tg1}, \eqref{tg2} and
\begin{align}\label{eq:convTS}
  \norm{\mathcal{T}(\vec{\phi})(t)-\mathcal{S}(\vec{\phi})(t)}_{L^{2}_{x}(\mathbb{R})}\lesssim & e^{{-}\epsilon t-\epsilon \min_{\ell}y_{\ell}-y_{\ell+1}}\norm{\mathcal{S}(\vec{\phi})(0)}_{L^{2}_{x}(\mathbb{R})},
\end{align}
\begin{align}
  \label{G1solll}
\norm{\mathfrak{G}_{\ell}(\vec{Z}_{j,\ell})(t,x)-e^{{-}i(\lambda_{j,\ell}-\omega_{\ell})t}e^{i\theta_{j}(t,x)\sigma_{3}}\vec{Z}_{j,\ell}(x-v_{\ell}t-y_{\ell})}_{L^{2}_{x}(\mathbb{R})}\lesssim & e^{{-}\epsilon t-\epsilon \min_{\ell}y_{\ell}-y_{\ell+1}},
  \end{align}and
  \begin{align}\label{G2soll}
\norm{\mathfrak{G}_{\ell}(\vec{Z}^{1}_{j,\lambda_{\ell}})(t,x)-e^{i\theta_{j}(t,x)\sigma_{3}}\vec{Z}^1_{j,\ell}(x-v_{\ell}t-y_{\ell})-te^{i\theta_{j}(t,x)\sigma_{3}}\vec{Y}_{j,\ell}(x-v_{\ell}t-y_{\ell})}_{L^{2}_{x}(\mathbb{R})}\lesssim  e^{{-}\epsilon t-\epsilon \min_{\ell}y_{\ell}-y_{\ell+1}},
   \end{align}
   where $\mathcal{H}_{\ell}Z^1_{j,\ell}=Y_{j,\ell}.$
\end{proposition}

%Different from Theorem  $1.13$ from \cite{dispa}, Theorem \ref{stablecase} does guarantee if there is asymptotic completeness on $L^{2}$ for charge transfer model \eqref{p}, when there exists a unstable mode for one of the operators $\mathcal{H}_{\ell}$ and the separation in the speeds $\min_{\ell}v_{\ell}-v_{\ell+1}>0$ is small. 
%Moreover, 
Different from \cite{dispa}, when $\min_{\ell}v_{\ell}-v_{\ell+1}$ is small, we do not expect to have the existence of a solution of \eqref{p} converging asymptotically in $L^{2}$ to 
\begin{equation*}
    e^{{-}i(\lambda_{j,\ell}-\omega_{\ell})t}e^{i\theta_{j}(t,x)\sigma_{3}}\vec{Z}_{j,\ell}(x-v_{\ell}t-y_{\ell}),
\end{equation*}
for an eigenvalue $\lambda_{j,\ell}$ with $\I\lambda_{j,\ell}>0$.

%and $\vec{Z}_{j,\ell}\in L^{2}_{x}(\mathbb{R},\mathbb{C}^{2})$ an eigenvector of $\mathcal{H}_{\ell}$ associated to $\lambda_{\ell}.$ 
As a direction application of the proposition above, we can rewrite the decomposition \eqref{princ11} in terms of the special solutions constructed above.
\begin{corollary}Under the same assumption as Proposition \ref{princ} one has that  $\Ra \mathcal{T}(\Diamond)(t)$ is a closed subspace of $L^{2}_{x}(\mathbb{R},\mathbb{C}^{2})$ and 
\begin{align}   
f(x)&=\mathcal{T}
\left(\overrightarrow{\phi_{0}}\right)(t,x)+\sum_{\ell=1}^{m}\sum_{j=1}^{K_{\ell,2}}a_{j,\ell}^{1}\mathfrak{G}_{\ell}(\vec{Z}^1_{j,\ell})(t,x)+\sum_{\ell=1}^{m}\sum_{\I\lambda_{\ell,j}\leq 0}a_{j,\ell}\mathfrak{G}_{\ell}(\vec{Z}_{j,\ell})(t,x)\\&{+}\sum_{\ell=1}^{m}\sum_{\I\lambda_{\ell,j}> 0}b_{j,\ell} e^{i\theta_{j}(t,x)\sigma_{3}}\vec{Z}_{j,\ell}(x-v_{\ell}t-y_{\ell})
\end{align}
for some $\vec{\phi_0}$ in the domain of $\mathcal{S}$, and some complex numbers $a_{j,\ell}^{1}$, $a_{j,\ell}$ and $b_{j,\ell}$. 
\end{corollary}

%Moreover, for all $t\geq 0,$ Propositions \ref{princ} and \eqref{remark4} imply thatIn particular, similarly to proof of Proposition \ref{princ}, 
From the decomposition above, we can define that for any $t\geq 0$ the existence of a unique projection $P_{\mathrm{cont}}(t)$ onto $\Ra \mathcal{T}(\Diamond)(t)$ such that
\begin{align*}
   \vec{f}(t,x)- P_{\mathrm{cont}}(t)\vec{f}(t,x)=&\sum_{\ell=1}^{m}\sum_{j=1}^{K_{\ell,2}}a_{j,\ell}^{1}(t)\mathfrak{G}_{\ell}(\vec{Z}^1_{j,\ell})(t,x)+\sum_{\ell=1}^{m}\sum_{\I\lambda_{\ell,j}\leq 0}a_{j,\ell}(t)\mathfrak{G}_{\ell}(\vec{Z}_{j,\ell})(t,x)\\&{+}\sum_{\ell=1}^{m}\sum_{\I\lambda_{\ell,j}> 0}b_{j,\ell}(t) e^{i\theta_{j}(t,x)\sigma_{3}}\vec{Z}_{j,\ell}(x-v_{\ell}t-y_{\ell}),
\end{align*}
for complex numbers $a_{j,\ell}^{1}(t)$, $a_{j,\ell}(t)$ and $b_{j,\ell}(t)$ for all indices $j$ and $\ell.$ Note that this projection does not commute with the linear flow\footnote{It is possible that one can construct projections that commute with the linear flow via exponential dichotomies like in \cite{CJdicho}.}.
%
%\cgc{We need to check if this projection commutes with the linear follow. I don't think it commutes. we need to check if the communitivity is required. I suspect no.}

To further study the scattering theory  of solutions to \ref{p},  due to the unstable nature of linear flows, we introduce the stable subspace.

\begin{definition}[Stable space]\label{def:stable}
A solution $\overrightarrow{\psi}(t,x)$ of \eqref{p} is in the \emph{{\bf stable  space}} if it satisfies for any $\ell\in\{1,2\,...,\,m\}$, and $\I\lambda_{\ell,j}> 0$
%and for any function $$h_{\ell}(x)\in \left(\bigcup_{i\lambda_{\ell,n}\in \sigma_d(\mathcal{H}_{\ell}),\lambda_{\ell,n}> 0} \ker [\mathcal{H}_{\ell}-i\lambda_{\ell,n}\mathrm{Id}]\right ),$$ the following property
\begin{equation}\label{asyorthsub}
    \lim_{t\to{+}\infty}\left\langle \overrightarrow{\psi}(t,x),\sigma_{3}e^{i\theta_{j}(t,x)\sigma_{3}}\vec{Z}_{j,\ell}(x-v_{\ell}t-y_{\ell})\right\rangle =0.
\end{equation}
Clearly, the scattering space defined above forms a subspace.
\end{definition}
Note that Proposition \ref{remark4} in particular shows that the stable space is non-trivial. 
In the stable space, solutions exhibit more precise behavior.
\begin{theorem}\label{stablecase}
There exists $L>1$ such that if $\min_{\ell}y_{\ell}-y_{\ell+1}>L>1,$ and $\vec{\psi}$ is a solution of \eqref{p} satisfying \eqref{asyorthsub}, then $\vec{\psi}$ has a unique representation of the form
\begin{align}\label{princ11stablespace}
   \vec{\psi}(t,x)=& \mathcal{S}
\left(\overrightarrow{\phi}(t)\right)(t,x)+ \sum_{\ell=1}^{m}\sum_{j=1}^{K_{\ell,2}}a_{j,\ell}^{1}(t)e^{i\theta_{\ell}(t,x)\sigma_{3}}\vec{Z}^1_{j,\ell}(x-v_{\ell}t-y_{\ell})\\&{+}\sum_{\ell=1}^{m}\sum_{\I\lambda_{\ell,j}\leq 0}e^{i\theta_{\ell}(t,x)\sigma_{3}}a_{j,\ell}(t)\vec{Z}_{j,\ell}(x-v_{\ell}t-y_{\ell}),  \\&{+}\sum_{\ell=1}^{m}\sum_{\I\lambda_{\ell,j}> 0}e^{i\theta_{j}(t,x)\sigma_{3}}a_{j,\ell}(t) \vec{Z}_{j,\ell}(x-v_{\ell}t-y_{\ell})
\end{align}
such that the following estimates holds uniformly for all $t\geq 0.$
\begin{itemize}
\item [(P1)]
\begin{align*}
\max_{\ell,j,\I \lambda_{j,\ell}> 0} \vert a_{j,\ell}(t)\vert\lesssim & e^{{-}\beta \min_{\ell}(v_{\ell}-v_{\ell+1})t+(y_{\ell}-y_{\ell+1})} \norm{\vec{\psi}(0,x)}_{L^{2}_{x}(\mathbb{R})}.
\end{align*}
\item [(P2)]
\begin{align*}    
\max_{\ell,j,\I \lambda_{j,\ell}<  0} |a_{j,\ell}(t)|\lesssim & e^{{-}\min_{\ell}\vert \I \lambda_{\ell} \vert \frac{t}{2}}e^{{-}\frac{\beta y}{2}}\norm{\vec{\psi}(0,x)}_{L^{2}_{x}(\mathbb{R})}.
\end{align*}
\item[(P3)] There exist constant complex values $a_{j,\ell,\infty}$ such that for any $t\geq 0:$
\begin{equation}\label{oscilatinga}
\max_{\ell,j,\lambda_{\ell,j}\in\mathbb{R}\setminus\{0\}}\vert a_{j,\ell}(t)-e^{it\lambda_{j,\ell}}a_{j,\ell,\infty}\vert\lesssim e^{{-}\beta[\min_{\ell}(v_{\ell}-v_{\ell+1})t+(y_{\ell}-y_{\ell+1})]}\norm{\vec{\psi}(0,x)}_{L^{2}_{x}(\mathbb{R})}.
\end{equation}

\item [(P4)] For any $\ell \in [m],$ there exist constant complex values $a^1_{j,\ell,\infty},\,c_{j,\ell,\infty}$ such that for any $t\geq 0:$
\begin{align*}
\max_{\ell}\vert a_{j,\ell}^1(t)-a^1_{j,\ell,\infty}-c_{j,\ell,\infty}t\vert\lesssim & e^{{-}\beta\min_{h}(y_{h}-y_{h+1}+(v_{h}-v_{h+1})s)}\norm{\vec{\psi}(0,x)}_{L^{2}_{x}(\mathbb{R})},\\
\end{align*}
if $\mathcal{H}_\ell \vec{Z}_{j,\ell}\neq 0$ and $\mathcal{H}_\ell^2 \vec{Z}_{j,\ell}=0$;
\begin{align*}
\max_{\ell,j} \vert a_{j,\ell}^1(t)-c_{j,\ell,\infty}\vert\lesssim & e^{{-}\beta\min_{h}(y_{h}-y_{h+1}+(v_{h}-v_{h+1})s)}\norm{\vec{\psi}(0,x)}_{L^{2}_{x}(\mathbb{R})}
\end{align*}
if $\mathcal{H}_\ell \vec{Z}_{j,\ell}=0$
\item [(P5)] $\norm{\mathcal{S}(\vec{\phi}(t))(t,x)}_{L^{2}_{x}(\mathbb{R})}\lesssim \norm{\vec{\psi}(0,x)}_{L^{2}_{x}(\mathbb{R})}.$ 
\item [(P6)] There exists a function $\phi_{\infty}(k)\in L^{2}_{k}(\mathbb{R})$ belonging to the domain of $\mathcal{T}$ satisfying for all $t\geq 0$
\begin{align*}
    \norm{\mathcal{S}(\vec{\phi}(t))(t,x)-\mathcal{S}(\vec{\phi}_{\infty})(t,x)}_{L^{2}_{x}(\mathbb{R})} & \lesssim e^{{-}\beta\min_{h}(y_{h}-y_{h+1}+(v_{h}-v_{h+1})t)}\norm{\vec{\psi}(0,x)}_{L^{2}_{x}(\mathbb{R})},\\
    \norm{\mathcal{S}(\vec{\phi}(t))(t,x)-\mathcal{S}(\vec{\phi}_{\infty})(t,x)}_{H^{1}_{x}(\mathbb{R})} & \lesssim e^{{-}\beta\min_{h}(y_{h}-y_{h+1}+(v_{h}-v_{h+1})t)}\norm{\vec{\psi}(0,x)}_{L^{2}_{x}(\mathbb{R})}.
\end{align*}
\end{itemize}
%\cgc{We need to be careful with the second statement in P6 above. Given $\phi_\infty\in L^2$, will those $\mathcal{S}$ in $L^\infty$?}
%Let 
%Furthermore, if $(\mathrm{H}4)$ is true, then $\vec{\psi}(t,x)$ satisfies the decay estimates \eqref{linfty}, \eqref{weightedlinfty} and \eqref{weightedderivative} for all $t\geq 0.$
\end{theorem}

\begin{remark}
 If $\epsilon\in (0,1)$ is sufficiently small and $\vec{\psi}(t)$ is a solution of \eqref{p} satisfying for all $t\geq 0$
 \begin{equation}\label{eq:epsilongrowth}
     \norm{\vec{\psi}(t)}_{L^{2}_{x}(\mathbb{R})}\lesssim e^{\epsilon t},
 \end{equation}
it can be verified similarly to the proof of Theorem \ref{stablecase} that $\vec{\psi}$ has a unique decomposition of the form \eqref{princ11}, and it satisfies all properties $(\mathrm{P1})-(\mathrm{P6})$ when $t\geq 0.$  The condition  \eqref{eq:epsilongrowth} appears naturally when one construct the center-stable direction using exponential dichotomies, see \cite{CJdicho}.
\end{remark}

To ensure the decay of the solutions, we should remove all discrete modes so that the scattering space is introduced below\footnote{Technically, we do not need expressions for the asymptotic orthogonality to stable modes since all stable modes will decay exponentially. But we still include stable modes in the definition to make the statement clearer.}.
\begin{definition}[Scattering space]\label{def:scatter}
A solution $\overrightarrow{\psi}(t,x)$ of \eqref{p} is in the {\bf scattering space} if it satisfies for any $\ell\in\{1,2\,...,\,m\}$,
%and for any function $$h_{\ell}(x)\in \ker \mathcal{H}_{\ell}^{2}\bigcup\left(\bigcup_{\lambda_{\ell,n}\in \sigma_d(\mathcal{H}_{\ell}),\lambda_{\ell,n}\neq0} \ker [\mathcal{H}_{\ell}-i\lambda_{\ell,n}\mathrm{Id}]\right ),$$ the following property
\begin{equation}\label{asyorth}
    \lim_{t\to{+}\infty}\left\langle \overrightarrow{\psi}(t,x),\sigma_{3}e^{i\theta_{j}(t,x)\sigma_{3}}\vec{Z}_{j,\ell}(x-v_{\ell}t-y_{\ell})\right\rangle =\lim_{t\to{+}\infty}\left\langle \overrightarrow{\psi}(t,x),\sigma_{3}e^{i\theta_{j}(t,x)\sigma_{3}}\vec{Z}^{1}_{j,\ell}(x-v_{\ell}t-y_{\ell})\right\rangle =0
    %\lim_{t\to{+}\infty}\left\langle \overrightarrow{\psi}(t,x),\sigma_{3}\mathfrak{g}_{\omega_{\ell},v_{\ell},y_{\ell},\gamma_{\ell}}(h_{\ell})(t,x)\right\rangle =0.
\end{equation}
Clearly, the scattering space defined above again forms a subspace. % One can find the map $P_{c}:\mathbb{R}_{\geq 0}\times L^{2}_{x}(\mathbb{R},\mathbb{C}^{2})\to L^{2}_{x}(\mathbb{R},\mathbb{C}^{2})$ to be the unique projection satisfying
%\begin{align*}
%P_{c}(t,f(x))=\overrightarrow{\psi}(t) \text{, for a $\overrightarrow{\psi}(t)$ satisfying \eqref{asyorth},}\\
%\left\langle f(x)-P_{c}(t,f(x)),\sigma_{3} \vec{\phi}(t)  \right\rangle =0 \text{, for any $\vec{\phi}(t)$ satisfying \eqref{asyorth}.}
%\end{align*}

\end{definition}
 For the general charge transfer models, we study the dispersive properties of solutions in the scattering space.
%\cgc{It would be lovely to have the following statement.}
\begin{definition}\label{def:localization}
To measure the local behavior of solutions, we introduce the following characteristic function.
 \begin{align*}
     \chi_{\ell}(\tau,x)=&\chi_{\left[\frac{y_{\ell}+y_{\ell+1}+\tau(v_{\ell}+v_{\ell+1})}{2},\frac{y_{\ell}+y_{\ell-1}+\tau(v_{\ell}+v_{\ell-1})}{2}\right]}(x) \text{, if $\ell\neq 1 $ and $\ell\neq m,$}\\
     \chi_{1}(\tau,x)=&\chi_{\left(\frac{y_{1}+y_{2}+\tau(v_{1}+v_{2})}{2},{+}\infty\right)}(x),\\
     \chi_{m}(\tau,x)=&\chi_{\left({-}\infty,\frac{y_{m}+y_{m-1}+\tau(v_{m}+v_{m-1})}{2}\right)}(x). 
 \end{align*}
\end{definition}
%Furthermore, the scattering maps $\mathcal{T}(\vec{\psi})(t,x)$ satisfies the following theorem.
%\cgc{in the statement below, projections are still there. also it seems that the communitivity of the projection is used.}
\begin{theorem}\label{decaytheoscatter}
If $\min_{\ell}y_{\ell}-y_{\ell+1}>L,$ all solutions $\vec{\psi}(t,x)$ of \eqref{p} satisfying \eqref{asyorth} enjoy the following estimates:
%of the form $\vec{\psi}(t,x)=\mathcal{T}(\vec{\phi}_{0})(t,x)$
%satisfies
for constants $C>c>0$ depending only on $\min_{\ell}v_{\ell}-v_{\ell+1}>0$ such that all $t\geq s\geq 0.$
\begin{align}\label{CCC1}
  c \norm{\vec{\psi}(0,x)}_{H^{j}_{x}(\mathbb{R})} \leq & \norm{\vec{\psi}(t,x)}_{H^{j}_{x}(\mathbb{R})}\leq C \norm{\vec{\psi}(0,x)}_{H^{j}_{x}(\mathbb{R})} \text{, for all $j\in\{0,1,2\},$}\\
\norm{\overrightarrow{\psi}(t,x)}_{L^{\infty}_{x}(\mathbb{R})}\leq & \max_{\ell} \frac{C}{(t-s)^{\frac{1}{2}}}\norm{(1+\vert x-y_{\ell}-v_{\ell}s\vert )\chi_{\ell}(s,x)\overrightarrow{\psi}(s,x)}_{L^{2}_{x}(\mathbb{R})}.
 \end{align}
Moreover, if in addition, we assume
\begin{itemize}
\item [(H4)] $\omega_{\ell}$ and ${-}\omega_{\ell}$ are not resonances\footnote{Recall that one says that $\pm \omega_\ell$ is a  resonance of $\mathcal{H}_{\ell}$ provided that $\mathcal{H}_{\ell}f=\pm \omega_\ell f$ has a solution $f\in L^\infty\, \text{but} \,f\notin L^2$.} of $ \mathcal{H}_{\ell},$
\end{itemize} then, for any $p\in (1,2)$ close enough to $1,$ $\vec{\psi}(t,x)$ satisfies for a $K>1$ depending uniquely on $\min_{\ell}v_{\ell}-v_{\ell+1}>0$ and $p^{*}=\frac{p}{p-1}$ the following decay estimates for all $t>s\geq 0.$
  \begin{align}\label{linfty}
  \norm{\vec{\psi}(t,x)}_{L^{\infty}_{x}(\mathbb{R})}\leq &\frac{K}{(t-s)^{\frac{1}{2}}}\Bigg[\norm{\vec{\psi}(s,x)}_{L^{1}_{x}(\mathbb{R})}\\&{+}e^{{-}\beta\min_{\ell}(y_{\ell}-y_{\ell+1}+(v_{\ell}-v_{\ell+1})s)}\norm{\vec{\psi}(s,x)}_{L^{2}_{x}(\mathbb{R})}\Bigg],\\ \label{weightedlinfty}
\norm{\frac{\vec{\psi}(t,x)}{(1+\vert x-y_{\ell}-v_{\ell}t\vert)}}_{L^{\infty}_{x}(\mathbb{R})}\leq & \frac{K (1+s)}{(t-s)^{\frac{3}{2}}} \norm{\overrightarrow{\psi}(s,x)}_{L^{1}_{x}(\mathbb{R})}\\ \nonumber 
&{+}\frac{K}{(t-s)^{\frac{3}{2}}}\max_{\ell}\norm{(1+\vert x-y_{\ell}-v_{\ell}s\vert )\chi_{\ell}(s,x)\overrightarrow{\psi}(s,x)}_{L^{1}_{x}(\mathbb{R})}\\ \nonumber 
&{+} %(y_{1}-y_{m}+(v_{1}-v_{m})s)^{2}
\frac{Ke^{{-}\beta\left[\min_{\ell} (v_{\ell}-v_{\ell+1})s+(y_{\ell}-y_{\ell+1})\right]}\norm{\overrightarrow{\psi}(s,x)}_{L^{2}_{x}(\mathbb{R})}}{(t-s)^{\frac{3}{2}}},\\ \label{weightedderivative}
\max_{\ell}\norm{\frac{\partial_{x}\vec{\psi}(t)}{\langle x-v_{\ell}t-y_{\ell} \rangle^{1+\frac{p^{*}-2}{2p^{*}}+\alpha}}}_{L^{2}_{x}(\mathbb{R})}
\leq & \frac{K\max_{\ell}\norm{(1+\vert x-y_{\ell}-v_{\ell}s\vert)\chi_{\ell}(s,x)\langle \partial_{x}\vec{\psi}(s)}_{L^{1}_{x}(\mathbb{R})}^{\frac{2-p}{p}}\norm{\vec{\psi}(s)}_{H^{1}_{x}(\mathbb{R})}^{\frac{2(p-1)}{p}}}{(t-s)^{\frac{3}{2}(\frac{1}{p}-\frac{1}{p^{*}})}} \\
 &{+}K\frac{(s+y_{1}-y_{m})}{(t-s)^{\frac{3}{2}(\frac{1}{p}-\frac{1}{p^{*}})}}\norm{\vec{\psi}(s)}_{W^{1,1}_{x}(\mathbb{R})}^{\frac{2-p}{p}}\norm{\vec{\psi}(s)}_{H^{1}_{x}(\mathbb{R})}^{\frac{2(p-1)}{p}}\\
 & %(y_{1}-y_{m}+(v_{1}-v_{m})s)^{2}
 {+}K\frac{e^{{-}\beta\min_{\ell}((v_{\ell}-v_{\ell+1})s+y_{\ell}-y_{\ell+1})}}{(t-s)^{\frac{3}{2}(\frac{1}{p}-\frac{1}{p^{*}})}}\norm{\vec{\psi}(s)}_{L^{2}_{x}(\mathbb{R})}.
\end{align}
%for $K=C(\min_{\ell}v_{\ell}-v_{\ell+1}).$ 
%All the estimates above are also true for $\vec{\psi}(t,x)=\mathcal{S}(\vec{\phi}_{0})(t,x)$ in the place of $\mathcal{T}(\vec{\phi}_{0})(t,x).$ 
\end{theorem}

%We will study the asymptotic completeness in the stable space and dispersive decays in the scattering space. More precisely, we will prove the following theorem. Moreover, for the general charge transfer model, we study the dispersive properties of solutions in the scattering space.

\subsection{Scalar models}\label{subsec:scalarmodel}
All our results can be easily extended to scalar charge transfer models in $1$d:
%. 
%Now we consider the scalar charge transfer model:
\begin{equation}
i\partial_{t}\psi(t,x)+\partial^{2}_{x}\psi(t,x)+\sum_{\ell=1}^mV_{\ell}(x-y_{\ell}-v_{\ell}t)\psi(t,x)=0
\end{equation}
with initial data $\psi(0,x)\in L^{2}_{x}(\mathbb{R})$. We list the velocities and initial positions as
\begin{equation}
v_{1}>v_{2}>...>v_{m}\,\,\text{and}\,\,\, y_{1}>y_{2}>...>y_{m}.
\end{equation}
We consider a set  $(V_{\ell}(x))_{\ell\in [m]}$ real potentials that satisfying the following conditions.
\begin{itemize}
    \item [(C1)] There exist constants $C\in\mathbb{R}_{+}$ and $\epsilon\in (0,1)$ satisfying $\max_{\ell}\vert V_{\ell}(x) \vert\leq \frac{C}{(1+x^{2})^{{-}1-\epsilon}}$ for all $x\in\mathbb{R}.$
    \item [(C2)] Each operator $H_{\ell}\coloneqq{-}\partial^{2}_{x}+V_{\ell}(x)$ has $M_\ell$ negative eigenvalues
    \begin{equation}
        \sigma_d(H_\ell)=\{\lambda_{j,\ell}\}_{j=1}^{M_\ell}
    \end{equation}with corresponding eigenfucntions $Z_{j,\ell}$
    \begin{equation}
      H_\ell Z_{j,\ell}= \lambda_{j,\ell} Z_{j,\ell}
    \end{equation}which form an orthonormal basis. We assume that zero is not an eigenvalue.
   % does not have eigenvalue in $\sigma_{c}(H_{\ell}).$
\end{itemize}
The conditions $(\mathrm{C}1)$ and $(\mathrm{C}2)$ above imply that the continuous spectrum of $H_\ell$ is given by $\sigma_c(H_\ell)=[0,\infty)$ and  each operator $H_{\ell}$ has a distorted Fourier basis of generalized eigenfunctions $e_{\ell}(x,k)\in L^{\infty}_{x}(\mathbb{R})$ satisfying
\begin{equation}\label{e(xk)scalar}
    H_{\ell} e_{\ell}(x,k)=k^{2} e_{\ell}(x,k) \text{, for any $k\in \mathbb{R}.$}
\end{equation}
Moreover, it is well-known that the generalized eigenfunctions $e_{\ell}(x,k)$ generate the projection of $L^{2}_{x}(\mathbb{R})$ onto the continuous spectrum of $H_{\ell},$ see the book \cite{Yafaev} for more details. The function $e_\ell(x,k)$ satisfies the following asymptotic behavior
\begin{align}\label{asyee}
    e_{\ell}(x,k)=&\begin{cases}
       e^{ik x}\left[ \frac{s_{\ell}(k)}{\sqrt{2\pi}}+O\left(\langle x \rangle^{{-}2}\right)\right]+e^{{-}ik x}O\left(\langle x \rangle^{{-}2}\right)\text{, if $x>0,$}\\
        \frac{e^{ik x}}{\sqrt{2\pi}}+\frac{r_{\ell}(k)e^{{-}ik x}}{\sqrt{2\pi}}+\sum_{\pm}e^{\pm ik x}O\left(\langle x \rangle^{{-}2}\right) \text{, if $x<0$}
    \end{cases} \text{, when $k>0,$}\\ \nonumber
    e_{\ell}(x,k)=&\begin{cases}
        \frac{s_{\ell}({-}k)e^{ik x}}{\sqrt{2\pi}}+\sum_{\pm}e^{\pm ik x}O\left(\langle x \rangle^{{-}2}\right)\text{, if $x<0,$}\\
        \frac{e^{ik x}}{\sqrt{2\pi}}+\frac{r_{\ell}({-}k)e^{{-}ik x}}{\sqrt{2\pi}}+\sum_{\pm}e^{\pm ik x}O\left(\langle x \rangle^{{-}2}\right) \text{, if $x>0$}
    \end{cases} \text{, when $k<0,$}
\end{align}
see, for example, \cite{DeiTru,WSSurvey,Wederwave,Yafaev}   for references.
\par Denote the projection onto the continuous spectrum of $H_\ell$ as $P_{c,\ell}$.

More precisely,
any function $f(x)\in L^{2}_{x}(\mathbb{R})$ of $\Ra P_{c,\ell}$ of $H_{\ell}$ has a unique eigenfunction expansion of the form
\begin{equation*}
    f(x)=\int_{\mathbb{R}} e(x,k)g(k)\,dk,
\end{equation*}
for $g(k)\in L^{2}_{k}(\mathbb{R}).$ In particular, the function $g$ satisfies
\begin{equation*}
    g(k)=\langle f(x),e_{\ell}(x,k) \rangle \text{, for all $k\in \mathbb{R}.$}
\end{equation*}
As a consequence, if $f\in \Ra {c,\ell}$ of $H_{\ell},$ we can verify that $e^{itH_\ell}f\in L^{2}_{x}(\mathbb{R})$ satisfies for all $t\in\mathbb{R}$
\begin{align}\label{evoldistFourier}
    e^{itH_{\ell}}f(x)=&\int_{\mathbb{R}} e^{itk^{2}} e_{\ell}(x,k)\langle f(\Diamond),e_{\ell}(\Diamond,k) \rangle\,dk.
\end{align}
From now on, considering the functions $\Tilde{G}_{\ell}$ and $\Tilde{F}_{\ell}$ below
\begin{align*}
    \Tilde{G}_{\ell}(x,k)\coloneqq & \frac{e_{\ell}(x,k)1_{({-}\infty,0]}(k)}{s({-}k)}+\left[e_{\ell}(x,k)-\frac{r_{\ell}(k)e_{\ell}(x,{-}k)}{s_{\ell}(k)}\right]1_{[0,{+}\infty)}(k) ,\\
    \tilde{F}_{\ell}(x,k)\coloneqq & \frac{e_{\ell}(x,k)1_{[0,{+}\infty)}(k)}{s_{\ell}(k)}+\left[e_{\ell}(x,k)-\frac{r_{\ell}({-}k)e_{\ell}(x,{-}k)}{s_{\ell}({-}k)}\right]1_{({-}\infty,0]}(k),
\end{align*}
we define the linear operators $\hat{\Tilde{G}}_{\ell}:D_{1}\subset L^{2}_{k}(\mathbb{R})\to L^{2}_{x}(\mathbb{R})$ and $\hat{\tilde{F}}_{\ell}:D_{2}\subset L^{2}_{k}(\mathbb{R})\to L^{2}_{x}(\mathbb{R})$ by
\begin{align*}
    \hat{\Tilde{G}}_{\ell}(\phi)(x)=&\int_{\mathbb{R}}\Tilde{G}_{\ell}(x,k)\phi(k)\,dk,\\
    \hat{\Tilde{F}}_{\ell}(\phi)=&\int_{\mathbb{R}}\Tilde{F}_{\ell}(x,k)\phi(k)\,dk.
\end{align*}
In particular, the asymptotic behavior \eqref{asyee} of $e_{\ell}(x,k)$ and the  Plancherel identity imply that there exist dense domains $D_{1}$ and $D_{2}$ of $L^{2}_{x}(\mathbb{R})$ for $\hat{\Tilde{G}}_{\ell}$ and $\hat{\Tilde{F}}_{\ell}.$ Moreover,  the Plancherel identity also implies that
\begin{align}\label{-inftygell}
    \hat{\Tilde{G}}_{\ell}(\phi)(x)=& \frac{1}{\sqrt{2\pi}}\int_{\mathbb{R}}e^{ikx}\phi(k)\,dk +O\left(\langle x \rangle^{{-}2}\norm{\phi}_{L^{2}_{x}(\mathbb{R})}\right) \text{, when $x<0,$}\\ \label{+inftyfell}
    \hat{\Tilde{F}}_{\ell}(\phi)(x)=& \frac{1}{\sqrt{2\pi}}\int_{\mathbb{R}}e^{ikx}\phi(k)\,dk +O\left(\langle x \rangle^{{-}2}\norm{\phi}_{L^{2}_{x}(\mathbb{R})}\right) \text{, when $x>0.$}
\end{align}
Similarly to Definition \ref{s0def}, there exists a subspace $D$ of $L^{2}_{k}(\mathbb{R})$ and a linear operator $\Tilde{S}(t):D\subset L^{2}_{k}(\mathbb{R})\to L^{2}_{k}(\mathbb{R})$ defined by
\begin{align}\label{Sscalar}
   \Tilde{S}(\vec{\phi})(t,x):= & \sum_{\ell=1}^{m}e^{i\left(\frac{v_{\ell}x}{2}-\frac{v_{\ell}^{2}t}{4}\right)}\hat{\Tilde{G}}_{\ell}\left(
   e^{{-}itk^{2}} 
e^{iy_{\ell}k}\phi_{\ell}\left(k+\frac{v_{\ell}}{2}\right)\right)(x-y_{\ell}-v_{\ell}t)\\
    & {-}\frac{1}{\sqrt{2 \pi}}\int_{\mathbb{R}}e^{{-}it k^{2}}
\varphi(k)e^{ikx}\,dk,\nonumber
\end{align}
such that the sequence of functions $\{\phi_{\ell}\}_{\ell\in [m]}$ satisfies
\begin{align}\label{ini1}
    \phi_{1}(k)=&\phi(k),\\ \label{iniell}
    \phi_{\ell+1}(k)=&\frac{\phi_{\ell}(k)-r_{\ell}\left(k-\frac{v_{\ell}}{2}\right)e^{{-}i2y_{\ell}(k-\frac{v_{\ell+1}}{2})+iy_{\ell}(v_{\ell}-v_{\ell+1})}\phi_{\ell}({-}k+v_{\ell})}{s_{\ell}\left(k-\frac{v_{\ell}}{2}\right)} \text{, if $1\leq  \ell\leq m-1$}\\ \label{varphi}
    \varphi(k)=& \sum_{\ell=1}^{m-1}\phi_{\ell}(k).
\end{align}
%\par Next, for each operator $H_{\ell},$ let 
%\begin{equation}
 %   M_{\ell}=\dim \Ra P_{d} \text{ of $H_{\ell},$}
%\end{equation}
%and $\{Z_{j,\ell}\}_{\ell\in [m_{\ell}]}$ an orthonormal basis of $\Ra P_{d} $ of $H_{\ell}.$
Consequently, similarly to the proof of Theorem \ref{stablecase}, we can deduce the following theorem.
\begin{theorem}
 Assume that all the operators $H_{\ell}$ satisfy $(\mathrm{C}_{1})$ and $(\mathrm{C}_{2}).$ Let $(v_{\ell})_{\ell\in [m]}$ be a sequence of real values satisfying $\min_{\ell}v_{\ell}-v_{\ell+1}>0.$ There exists $L>1$ depending on $\min_{\ell\in [m-1]}v_{\ell}-v_{\ell+1}$ and on the potentials $V_{\ell}$ such that if the real set $(y_{\ell})_{\ell\in [m]}$ satisfies
 \begin{equation*}
     \min_{\ell\in [m-1]} y_{\ell}-y_{\ell+1}>L,
 \end{equation*}
then any strong solution $\psi(t,x)\in L^{2}_{x}(\mathbb{R})$ of the linear Schrödinger equation 
\begin{align*}
i\partial_{t}\psi(t,x)+\partial^{2}_{x}\psi(t,x)+\sum_{\ell}V_{\ell}(x-y_{\ell}-v_{\ell}t)\psi(t,x)=&0,
\psi(0,x)\in L^{2}_{x}(\mathbb{R})
\end{align*}
has a unique representation of the form 
\begin{align*}
 \psi(t,x)&=\Tilde{T}
\left(\phi_{0}\right)(t,x) +\sum_{\ell=1}^{m}\sum_{j=1}^{M_{\ell}}a_{j,\ell}\mathfrak{G}_{\ell}(Z_{j,\ell})(t,x),
\end{align*}
for some $\phi_{0}(x)$ in the domain of $\mathcal{\Tilde{S}}$, and some complex numbers $a_{j,\ell}$ and $a_{j,\ell}^1$. In the decomposition above,
\begin{itemize}
    \item $\Tilde{T}
\left(\phi_{0}\right)(t,x)$ is the unique solution of \eqref{p} satisfying
\begin{equation}\label{tphidecayyscalar}
    \lim_{t\to{+}\infty}\norm{\Tilde{T}
(\phi_{0})(t,x)-\Tilde{S}
(\phi_{0})(t,x)}_{L^{2}_{x}(\mathbb{R})}=0;
\end{equation}
\item For $H_{\ell}Z_{j,\ell}=\lambda_{j,\ell}Z_{j,\ell},$ the function $\mathfrak{G}_{\ell}(Z_{j,\ell})(t,x)$ is the unique solution of \eqref{p} satisfying
\begin{equation}\label{tg1scalar}
    \lim_{t\to{+}\infty}\norm{\mathfrak{G}_{\ell}(Z_{j,\ell})(t,x)-e^{{-}it\lambda_{j,\ell}}e^{i\left(\frac{v_{\ell}x}{2}-\frac{v_{\ell}^{2}t}{4}\right)}Z_{j,\ell}(x-v_{\ell}t-y_{\ell})}_{L^{2}_{x}(\mathbb{R})}=0.
\end{equation}
\end{itemize}
Moreover, the scattering part $\Tilde{T}(\phi_0)$ satisfies the same estimates as in \eqref{CCC1}. Under additional assumption that
\begin{itemize}
\item [(C4)] $0$ is not a resonance of $H_\ell$, \footnote{Recall that one says that $0$ is a  resonance of $H_\ell$ provided that $H_{\ell}f=0$ has a solution $f\in L^\infty\, \text{but} \,f\notin L^2$.}
\end{itemize}
then the scattering part $\Tilde{T}(\phi_0)$  also satisfies same estimates as \eqref{linfty}, \eqref{weightedlinfty}, \eqref{weightedderivative}.
\end{theorem}
\begin{proof}
    The proof is analogous to the proofs of Theorem \ref{stablecase} and Theorem \ref{decaytheoscatter}.
\end{proof}

\section{Structures of linear maps}\label{sec:linearmaps}
In this section, we study linear maps of two special forms which are crucial for us to establish the invertibility of dispersive map $\mathcal{S}$ and dispersive estimates later on.

Given the collection of velocities \eqref{eq:velocity}, reflection coefficients and transmission coefficients as in Lemma \ref{lem:reftranscoeff}, we consider the following two classes of linear operators $$T=(T_{1},\,...,T_{2m-2}):L^{2}_{k}(\mathbb{R},\mathbb{C}^{2m-2})\to L^{2}_{k}(\mathbb{R},\mathbb{C}^{2m-2}).$$ 

The first one is given by the following form
\begin{align}\label{form1}\tag{Form 1}
     T_{1}(\vec{g})=&g_{1}(k)-r_{2}\left({-}k-\frac{v_{2}}{2}\right)g_{2}\left({-}k-v_{2}\right)-s_{2}\left({-}k-\frac{v_{2}}{2}\right)g_{3}(k),\\ \nonumber
     T_{2}(\vec{g})=&g_{2}(k)-r_{1}\left(k+\frac{v_{1}}{2}\right)g_{1}({-}k-v_{1}),\\ \nonumber
      T_{3}(\vec{g})=& g_{3}(k)-r_{3}\left({-}k-\frac{v_{3}}{2}\right)g_{4}({-}k-v_{3})-s_{3}\left(k-\frac{v_{3}}{2}\right)g_{5}(k),\\ \nonumber
      T_{4}(\vec{g})=& g_{4}(k){-}r_{2}\left(k+\frac{v_{2}}{2}\right)g_{3}\left({-}k-v_{2}\right)-s_{2}\left(k+\frac{v_{2}}{2}\right)g_{2}\left(k\right),\, ...,\\ \nonumber
      T_{2\ell}(\vec{g})=& g_{2\ell}(k)-r_{\ell}\left(k+\frac{v_{\ell}}{2}\right)g_{2\ell-1}\left({-}k-v_{\ell}\right)-s_{\ell}\left(k+\frac{v_{\ell}}{2}\right)g_{2\ell-2}(k),\\ \nonumber
    T_{2\ell+1}(\vec{g})=&g_{2\ell+1}(k)-r_{\ell+2}\left({-}k-\frac{v_{\ell+2}}{2}\right)g_{2\ell+2}\left({-}k-v_{\ell+2}\right)\\&{-}s_{\ell+2}\left({-}k-\frac{v_{\ell+2}}{2}\right)g_{2\ell+3}(k),\, ...,\\ \nonumber
      T_{2m-3}(\vec{g})=& g_{2m-3}(k)-r_{m}\left({-}k-\frac{v_{m}}{2}\right)g_{2m-2}\left({-}k-v_{m}\right),\\ \nonumber
      T_{2m-2}(\vec{g})=& g_{2m-2}\left(k\right){-}r_{m-1}\left(k+\frac{v_{m-1}}{2}\right)g_{2m-3}\left({-}k-v_{m-1}\right)\\ \nonumber
      &{-}s_{m-1}\left(k+\frac{v_{m-1}}{2}\right)g_{2m-4}(k).\end{align}
The second one is given by  the form:
 \begin{align}\label{form2}\tag{Form 2}
T_{1}(\vec{g})=&g_{1}(k)-  r_{1}\left({-}k-\frac{v_{1}}{2}\right)g_{2}\left({-}k+v_{1}\right),\\ \nonumber
T_{2}(\vec{g})=&g_{2}(k)-r_{2}\left(k-\frac{v_{2}}{2}\right)g_{1}\left({-}k+v_{2}\right)-s_{2}\left(k-\frac{v_{2}}{2}\right)g_{4}\left(k\right), ...,\\ \nonumber
T_{2n-1}(\vec{g})=&g_{2n-1}(k)- r_{n}\left({-}k+\frac{v_{n}}{2}\right)g_{2n}\left({-}k+v_{n}\right){-}s_{n}\left({-}k+\frac{v_{n}}{2}\right)g_{2n-3}\left(k\right),\\ \nonumber
T_{2n}(\vec{g})=&g_{2n}(k)-r_{n+1}\left(k-\frac{v_{n+1}}{2}\right)g_{2n-1}\left({-}k+v_{n+1}\right)\\
&{-}s_{n+1}\left(k-\frac{v_{n+1}}{2}\right)g_{2n-1}\left({-}k+v_{n+1}\right)g_{2n+2}\left(k\right),...,\\ \nonumber
T_{2m-2}(\vec{g})=&g_{2m-2}(k)-r_{m}\left(k-\frac{v_{m}}{2}\right)g_{2m-3}\left({-}k+v_{m}\right).
\end{align}
\par Next, for each operator $T$ above, let $R=\mathrm{Id}-T,$ we have the following crucial estimates for $R$.
\begin{lemma}\label{lem:claimR}
Given $R=\mathrm{Id}-T$, it satisfies that for $j\in\mathbb{N}$,
\begin{equation}\label{claimR}
    %\text{Lemma \ref{lem:claimR}:} 
    \norm{R^{j(m-1)}}_{L^{2}\to L^{2}}+\norm{R^{j(m-1)}}_{\mathcal{F}L^{1}\to \mathcal{F}L^{1}}\leq \min\left(\frac{jmC(m)^{j}}{j!},\frac{jmC(m)^{j}}{(\lfloor\frac{j-3}{2}\rfloor!)^{2}}\right).
\end{equation}
In particular, $T:L^{2}_{k}(\mathbb{R},\mathbb{C}^{2m-2})\to (\mathbb{R},\mathbb{C}^{2m-2})$ and $T:(\mathcal{F}L^{1})^{2m-2}\to(\mathcal{F}L^{1})^{2m-2}$ are linear homeomorphisms.
\end{lemma}
%The Lemma \ref{lem:claimR} will be proved in Section \ref{sec:linearmaps}. As a consequence of Claim $(R),$ 
%\begin{proof}[Proof of Proposition \ref{princ} assuming Lemma \ref{lem:claimR}]
%The steps $5$ and $6$ of Lemma $5.1$ from \cite{dispa} imply that Theorem \ref{princ} is true if the map $T:L^{2}_{k}(\mathbb{R},\mathbb{C}^{2})\to L^{2}_{k}(\mathbb{R},\mathbb{C}^{2})$ is invertible. Therefore, since Lemma \ref{lem:claimR} implies the existence of a natural number $n$ satisfying 
%\begin{equation}\label{smallRR}
 %  \norm{(\mathrm{Id}-T)^{n}}_{\mathcal{F}L^{1}\to \mathcal{F} L^{1}} +\norm{(\mathrm{Id}-T)^{n}}_{L^{2}\to L^{2}}<\frac{1}{2},
%\end{equation}
%we can verify for $R=\mathrm{Id}-T$ that the operator
%\begin{equation*}
%Q_{n}\coloneqq\left(Id+R+R^{2}+...+R^{n-1}\right) \circ T=Id-R^{n}
%\end{equation*}
%is invertible since \eqref{smallRR} implies $\norm{Q_{n}(\vec{g})}_{L^{2}}\geq \frac{\norm{g}_{L^{2}}}{2}.$
%\par In conclusion, Proposition \ref{princ} is true.
%\end{proof}

%\newpage

The proof of Lemma \ref{lem:claimR} will require the estimate proved in the following Lemma.

\begin{lemma}\label{factUpp}
Let $m\in \mathbb{N}_{\geq 2},$ and $\{(r_{\ell},\,s_{\ell})\}_{\ell\in\mathbb{N}}$ be a set such that each element $(r_{\ell},s_{\ell})\in C^{2}(\mathbb{R},\mathbb{C})$ satisfies for a constant $C>1$
\begin{align}\label{decrs}
     \max_{\ell\in [m],j\in\{0,1\}}\left\vert \frac{d^{j}r_{\ell}(k)}{dk^{j}} \right\vert+ \left\vert \frac{d^{j}}{dk^{j}}\left(1-s_{\ell}(k)\right) \right\vert\leq & \frac{C}{(1+\vert k\vert)^{j+1}} \text{, for all $k\in\mathbb{R},$}\\ \label{rs1}
     \max\left(\vert r_{\ell}(k) \vert, \vert s_{\ell}(k)\vert\right) \leq 1 \text{, for all $\ell$ and $k\in\mathbb{R}.$}
\end{align}

Let $\{q_{j}\}_{j\in\{1,\,...,\,M\}}$ for a $M\in\mathbb{N}_{\geq 2}$ be a set of real numbers satisfying $\min q_{j}-q_{j+1}>0.$ There exists $K>1$ such that if $\{w_{1},\,...,\,w_{M(m-1)}\}$ is any subset of $\mathbb{R}$ of size $M(m-1),$ then, the following inequalities holds for all $y\in\mathbb{R},$
 for any $f(k)\in L^{2}_{k}(\mathbb{R}),$
\begin{align}\label{ineq1}
    \norm{e^{iky}[\prod_{n=1}^{M(m-1)}s_{n}(k+w_{n})][\prod_{j=1}^{M}r_{j}(k+q_{j})]f(k)}_{L^{2}_{k}(\mathbb{R})}\leq & \max_{h\in\{1,2\}}\frac{C^{M}\norm{f(k)}_{L^{2}_{k}(\mathbb{R})}}{(\lfloor\frac{M-2}{2}\rfloor!)^{2}[\min_{\ell}(q_{\ell}-q_{\ell+1})^{M-h}]},
 \end{align}
and  for any $f(k)\in \mathcal{F}L^{1},$
 \begin{align}
    \label{ineq2}
    \norm{e^{iky}[\prod_{n=1}^{M(m-1)}s_{n}(k+w_{n})][\prod_{j=1}^{M}r_{j}(k+q_{j})]f(k)}_{\mathcal{F}L^{1}}\leq & \max_{h\in\{1,2\}}\frac{MmC^{M+1}\norm{f}_{\mathcal{F}L^{1}}}{\left(\lfloor\frac{M-3}{2}\rfloor!\right)^{2}(\min_{\ell}q_{\ell}-q_{\ell+1})^{M-2-h}}.
\end{align}
\end{lemma}
the proof of Lemma \ref{factUpp} will follow from the following claim.

\begin{lemma}\label{claimfactorial}
 If any $M\in\mathbb{N}_{\geq 2},\,\min_{\ell}q_{\ell}-q_{\ell+1}>0$ for any $\ell\in [M-1],$ then the following inequality holds.
\begin{equation}\label{factorialupper}
    \min_{k\in\mathbb{R}}\prod_{j=1}^{M}\frac{1}{1+\vert k+q_{j} \vert }\leq \max\left(\frac{1}{\left(\lfloor\frac{M-2}{2}\rfloor !\right)^{2}(\min_{\ell}q_{\ell}-q_{\ell+1})^{M-2}},\frac{1}{(M-1)![\min_{\ell}q_{\ell}-q_{\ell+1}]^{M-1}}\right).
\end{equation}
\end{lemma}
\begin{proof}[Proof of Lemma \ref{claimfactorial}]
   Without loss of generality, we can restrict the prove of the claim to the case where $\min_{\ell}q_{\ell}>0$ and $\min_{\ell}q_{\ell}-q_{\ell+1}>0$ since one can uniformly add a fixed  big positive constant $C$ to the set $\{q_\ell\}$ so that $q_{M-1}+C>0$, and this procedure preserves  $\min_{\ell}\left((q_{\ell}+C)-(q_{\ell+1}+C)\right)=\min_{\ell}q_{\ell}-q_{\ell+1}.$
   
Firstly, we note that
\begin{equation}\label{factorialkk}
    k+q_{j}\geq k+q_{j+1}+\min_{\ell}(q_{\ell}-q_{\ell+1}),
\end{equation}
Then for  $k>0$ or $k<{-}q_{1},$ iterating the inequality above, we have that
\begin{equation*}
    \vert k+q_{n}\vert \geq  \begin{cases}
        (M-n)[\min_{\ell}q_{\ell}-q_{\ell+1}] \text{, if $k\geq 0,$}\\
     (n-1) [\min_{\ell}q_{\ell}-q_{\ell+1}] \text{, if $k< {-}q_{1}.$}     
    \end{cases}
\end{equation*}
Consequently, if $k>0$ or $k<{-}q_{1},$ from estimates above,
\begin{equation*}
    \prod_{j=1}^{M}\frac{1}{1+\vert k+q_{j} \vert }\leq \frac{1}{(M-1)!}\frac{1}{[\min_{\ell}(q_{\ell}-q_{\ell+1})]^{M-1}}.
\end{equation*}
\par Otherwise, if ${-}q_{j}\leq k< {-}q_{j+1}$ for some $j\in [M-1],$ we can verify from \eqref{factorialkk} the following estimate.
\begin{equation*}
    \vert k+q_{n}\vert \geq  \begin{cases}
        (j-n)[\min_{\ell}q_{\ell}-q_{\ell+1}] \text{, if $n \leq j,$}\\
     (n-j-1) [\min_{\ell}q_{\ell}-q_{\ell+1}] \text{, if $n \geq j+1.$}     
    \end{cases}
\end{equation*} 
Therefore, we deduce from the estimate above that
\begin{equation*}
    \prod_{j=1}^{M}\frac{1}{1+\vert k+q_{j} \vert }
    \leq \left[\frac{1}{(j-1)![\min_{\ell}q_{\ell}-q_{\ell+1}]^{j-1}}\right]\left[\frac{1}{(M-j-1)![\min_{\ell}q_{\ell}-q_{\ell+1}]^{M-j-1}}\right].
\end{equation*}
\par In conclusion, \eqref{factorialupper} holds for any $m\in\mathbb{N}.$
\end{proof}
\begin{proof}[Proof of Lemma \ref{factUpp}]
 First, using \eqref{decrs} and \eqref{rs1}, we can verify the following estimate.
 \begin{equation*}
     \left\vert e^{iky}[\prod_{n=1}^{M(m-1)}s_{n}(k+w_{n})][\prod_{j=1}^{M}r_{j}(k+q_{j})]f(k)\right\vert
         \leq C^{M}\vert f(k) \vert \prod_{j=1}^{M}\frac{1}{1+\vert k+q_{j} \vert }.
 \end{equation*}
Therefore, using the estimate \eqref{factorialupper}, we conclude from the inequality above that
\begin{equation*}
    \norm{e^{iky}[\prod_{n=1}^{M(m-1)}s_{n}(k+w_{n})][\prod_{j=1}^{M}r_{j}(k+q_{j})]f(k)}_{L^{2}_{k}(\mathbb{R})}\leq \max_{h\in\{1,2\}}\frac{C^{M}\norm{f}_{L^{2}_{k}(\mathbb{R})}}{\left(\lfloor\frac{M-2}{2}\rfloor !\right)^{2}(\min_{\ell}q_{\ell}-q_{\ell+1})^{M-h}}.
\end{equation*}
\par Next to show \eqref{ineq2}, we first note that by the Hausdorff-Young inequality and the Cauchy-Schwarz inequality, one has
\begin{equation*}
    \norm{e^{iky}u_{1}(k)u_{2}(k)}_{\mathcal{F}L^{1}}\leq \norm{u_{1}(k)}_{\mathcal{F}L^{1}}\norm{u_{2}(k)}_{L^\infty_k}\leq C\norm{u_{1}(k)}_{H^{1}_{k}(\mathbb{R})}\norm{u_{2}(k)}_{\mathcal{F}L^{1}}.
\end{equation*}
By estimates \eqref{decrs}, \eqref{rs1}, we can deduce using the product rule of the derivative that
   \begin{align}
 &   \norm{e^{iky}[\prod_{n=1}^{M(m-1)}s_{n}(k+w_{n})][\prod_{j=1}^{M}r_{j}(k+q_{j})]f(k)}_{\mathcal{F}L_{1}}\\
    &\quad\quad\quad\quad\quad\quad \leq  C\norm{\prod_{n=1}^{M(m-1)}s_{n}(k+w_{n})][\prod_{j=1}^{M}r_{j}(k+q_{j})]}_{H^{1}_{k}(\mathbb{R})}\norm{f}_{\mathcal{F}L^{1}}\\
  &  \quad\quad\quad\quad\quad\quad\quad\quad \leq MmC^{M+1}\norm{f}_{\mathcal{F}L^{1}}\inf_{\ell,k\in\mathbb{R}}\prod_{j=1,\,j\neq \ell}^{M}\frac{1}{1+\vert k+q_{j}  \vert}.
    \end{align}
Therefore, Lemma \ref{claimfactorial} implies that
\begin{equation*}
    \norm{e^{iky}[\prod_{n=1}^{M(m-1)}s_{n}(k+w_{n})][\prod_{j=1}^{M}r_{j}(k+q_{j})]f(k)}_{\mathcal{F}L_{1}}\leq \max_{h\in\{1,2\}}\frac{MmC^{M+1}\norm{f}_{\mathcal{F}L^{1}}}{\left(\lfloor \frac{M-3}{2}\rfloor !\right)^{2}(\min_{\ell}q_{\ell}-q_{\ell+1})^{M-2-h}},
\end{equation*}
which corresponds to inequality \eqref{ineq2}.
\end{proof}
%In particular, proving that $T$ is invertible on $L^{2}$ is equivalent to prove the Asymptotic Completeness Theorem $1.13$ from \cite{Dispesti} for all the cases where $v_{\ell}>v_{\ell+1}.$

\begin{proposition}\label{mainprop}
Let $n\in\mathbb{N},$ and assume that the linear bounded map $T: L^{2}_{k}(\mathbb{R}_{k},\mathbb{C}^{2m-2})\to  L^{2}_{k}(\mathbb{R}_{k},\mathbb{C}^{2m-2})$ is of the from \eqref{form1} or \eqref{form2}. For any $\vec{g}\in \mathbb{C}^{2m-2}$, the vector $(\mathrm{Id}-T)^{n}(\vec{g})\in L^{2}_{k}(\mathbb{R}_{k},\mathbb{C}^{2m-2})$ has the following representation:
 \begin{equation*}
     (\mathrm{Id}-T)^{n}(\vec{g})(k)=R_{n}(\vec{g})+S_{n}(\vec{g}),
 \end{equation*}
where $R_n(\vec{g})$ consists of terms in which each summand involves \emph{at least one} reflection coefficient, while $ S_n(\vec{g})$ consists of terms involving \emph{only} transmission coefficients.

More precisely, each coordinate $[R_{n}(\vec{g})]_{i}$ is a finite sum of at most $2^{n}$ elements  $\vec{z}_{i}\in L^{2}_{k}(\mathbb{R}_{k},\mathbb{C})$ satisfying one of the two following expressions: for some $N_i\in\mathbb{N}$, some shifts $\{f_{i,n}\}\subset \mathbb{R}$ and $\{v_{\ell,i}\}\subset \mathbb{R}$  in terms of $v_1>v_2>\ldots>v_m$,  a function $j_\ell\in\{1,\ldots,m\}$, and some $h_{i,n}\in \{1,\ldots,2m-1\}$
%\cgc{What is $j$ below?}
%\textcolor{red}{$j_{\ell}$ is one of the numbers of $\{1,\,...,\,m\},$ we don't need to know exactly who they are, the most important thing is the existence of a order for the $v_{\ell,i}$ satisfying $(P1).$}
\begin{align*}
    \vec{z}_{i}(k)=&S_{i}(k)[\prod_{\ell=1}^{N_{i}}r_{j_{\ell}}(k+v_{\ell,i})]\tau_{f_{i,n}(v)}g_{h_{i,n}}(\pm k),\\
    \vec{z}_{i}(k)=&S_{i}(k)[\prod_{\ell=1}^{N_{i}}r_{j_{\ell}}({-}k-v_{\ell,i})]\tau_{f_{i,n}(v)}g_{h_{i,n}}(\pm k),
\end{align*}
such that 
\begin{itemize}
\item [(P1)] If $N_{i}>1$, then $v_{\ell+1,i}-v_{\ell,i}\geq \min_{n}\frac{v_{n}-v_{n+1}}{2}$ for $1\leq\ell\leq N_i-1 $,
If $N_{i}=1$,  this property is vacuously true;
\item [(P2)] $S_{i}(k)=\prod_{\ell=1}^{N_{i,
S}}s_{h_\ell}(\mp k\pm v_{i,n,\ell})  $ for some $N_{i,1}\in\mathbb{N}$, some  function $h_\ell\in\{1,\ldots,m\}$ and some shifts $\{v_{i,n,i}\}\subset \mathbb{R}$  in terms of $v_1>v_2>\ldots>v_m$; 
\item [(P3)] $N_{i,1}+N_{i}=n;$\\
\item [(P4)] $S_n(\vec{g})$ is a finite sums of terms of the form $\prod_{\ell=1}^{n}s_{\iota_\ell}(\mp k\pm v_{i,n,\ell}^S) \tau_{M_{i,n}} g_{q_{i,n}}(\pm k)$ for some shifts $\{v_{i,n,i}^S\},\,\{M_{i.n}\}\subset \mathbb{R}$  in terms of $v_1>v_2>\ldots>v_m$, some  function $\iota_\ell\in\{1,\ldots,m\}$, and $q_{i,n}\in \{1,\ldots,2m-1\}$.
\end{itemize}
\end{proposition}
\begin{proof}[Proof of Proposition \ref{mainprop}.]
It is enough to prove Proposition \ref{mainprop} for the linear operator $T$ to be of \eqref{form1}, since the proof for operators $T$ of \eqref{form2} is completely analogous.
\par We will prove the desired results using induction.

 Note that the properties $(\mathrm{P2}),\,(\mathrm{P3}),\,(\mathrm{P4})$ are direct consequences of the definition of $(\mathrm{Id}-T)$ and the fact that  $(\mathrm{Id}-T)$ are applied $n$ times to the element $\vec{g}=(g_{1},\,...,\,g_{2m-2})\in L^{2}_{k}(\mathbb{R},\mathbb
{C}^{2m-2}).$
\par The proof of property $(\mathrm{P1})$ will be more involved using induction. 
%follow from an argument of the finite induction principle. 
Actually, we will prove that the following stronger condition is always true for any natural number $n\geq 1.$

\noindent{\bf Claim:}
    Given $n\in\mathbb{N},$ then the $(2i+1)-$th ($2i-$th) coordinates of of $R_{n}(\vec{g})$ from $(\mathrm{Id}-T)^{n}$ are consist of elements of the following forms:
    %lement $\vec{z}_{j}$  on  has for a $S(k)$ satisfying property $(P2)$ the following form for $h\in\mathbb{N}$
    \begin{align}
    \label{odd}
    &\mathfrak{S}_i(k)r_{i+2+2h}\left({-}k-\frac{v_{i+2h}}{2}\right)[\prod_{\ell=1}^{N_{i}}r_{j(\ell)}({-}k-v_{\ell,i})]g_{N(i,n)}(k+f_{i,n}(v)) \text{, on the $(2i+1)-$th coordinate,}\\ \label {even}
   & \mathfrak{S}_i(s)r_{i-2h}\left(k+\frac{v_{i-2h}}{2}\right)[\prod_{\ell=1}^{N_{i}}r_{j(\ell)}(k+v_{\ell,i})]g_{N(i,n)}(\pm k\pm f_{i,n}(v)) \text{, on the $(2i)-$th coordinate,}
    \end{align}
for some $h,\,N_i\in\mathbb{N}$, some $\mathfrak{S}_i$ satisfying  $(\mathrm{P2})$, some shifts $\{f_{i,n}\}\subset \mathbb{R}$ and $\{v_{\ell,i}\}\subset \mathbb{R}$  in terms of $v_1>v_2>\ldots>v_m$, and  functions $j(\ell),\,N(i,n)\in\{1,\ldots,m\}$.

\noindent{\bf Proof of Claim:}
Indeed, when $n=1,$ each coordinate of $(\mathrm{Id}-T)(\vec{g})$ is given by the following
\begin{align}
  &\left[(\mathrm{Id}-T)(\vec{g})\right]_{2i}=  r_{i}\left(k+\frac{v_{i}}{2}\right)g_{2i-1}({-}k-v_{i})+(1-\delta^{i}_{1})s_{i}\left(k+\frac{v_{i}}{2}\right)g_{2i-2}(k),\\
 & \left[(\mathrm{Id}-T)(\vec{g})\right]_{2i+1}=r_{i+2}\left({-}k-\frac{v_{i+2}}{2}\right)g_{2i+2}({-}k-v_{i+2})+(1-\delta^{m-2}_{i})s_{i+2}\left({-}k-\frac{v_{i+2}}{2}\right)g_{2i+3}(k),
\end{align}
which clearly satisfies the conditions \eqref{odd}, \eqref{even} and Property $(\mathrm{P1}).$ Therefore, we can assume that the conditions \eqref{odd}  and\eqref{even},   Proposition \ref{mainprop} are still true when $n\leq N_{0},$ for some $N_{0}\in\mathbb{N}$.

\par Since Proposition \ref{mainprop} is true when $n=N_{0}$, one has  $(\mathrm{Id}-T)^{N_{0}}(\vec{g})=R_{N_{0}}(\vec{g})+S_{N_{0}}(\vec{g}),$ and  it is not difficult to verify from the definition of $S_{N_0}(\vec{g})$ and the basis case $n=1$ that each coordinate $[R_{1}(S_{N_{0}}(\vec{g}))]_{i}$ satisfies one of the conditions \eqref{odd} or \eqref{even}, depending on the parity of $i.$

By the induction assumption, typical elements in $(2i+1)$-th and $(2i)$-th coordinates of $R_{N_{0}}(\vec{g})$ are given by the following functions:
\begin{align*}
    W_{2i+1}(k)&:=\mathfrak{S}_i(k)r_{i+2+2h}\left({-}k-\frac{v_{i+2h}}{2}\right)[\prod_{\ell=1}^{n_{i}}r_{j(\ell)}({-}k-v_{\ell,i})]g_{N(i,n)}(k+f_{i,n}(v)),\\
    W_{2i}(k)&:=\mathfrak{S}_i(k)r_{i-2h}\left(k+\frac{v_{i-2h}}{2}\right)[\prod_{\ell=1}^{n_{i}}r_{j(\ell)}(k+v_{\ell,i})]g_{N(i,n)}(\pm k\pm f_{i,n}(v))
\end{align*}respectively.

 Using the explicit formula of $\mathrm{Id}-T,$ 
%and the functions
%\begin{align*}
 %   W_{2i+1}(k):=\mathfrak{S}_i(k)r_{i+2+2h}\left({-}k-\frac{v_{i+2h}}{2}\right)[\prod_{\ell=1}^{n_{i}}r_{j(\ell)}({-}k-v_{\ell,i})]g_{N(i,n)}(k+f_{i,n}(v)),\\
%    W_{2i}(k):=\mathfrak{S}_i(k)r_{i-2h}\left(k+\frac{v_{i-2h}}{2}\right)[\prod_{\ell=1}^{n_{i}}r_{j(\ell)}(k+v_{\ell,i})]g_{N(i,n)}(\pm k\pm f_{i,n}(v)),
%\end{align*}
we can verify that
\begin{align}
    & (\mathrm{Id}-T)[W_{2i+1}(k)e_{2i+1}]\\
     &=r_{i+1}\left(k+\frac{v_{i+1}}{2}\right)W_{2i+1}({-}k-v_{i+1})e_{2i+2}+\delta^{2i-1}_{2m-3}s_{i+1}\left({-}k-\frac{v_{i+1}}{2}\right)W_{2i+1}(k)e_{2i-1}
\end{align}
and
\begin{align}
 & (\mathrm{Id}-T)W_{2i}e_{2i}\\
 &=r_{i+1}\left({-}k-\frac{v_{i+1}}{2}\right)W_{2i}({-}k-v_{i+1})e_{2i-1}+\delta^{2i+2}_{2} s_{i+1}\left(k+\frac{v_{i+1}}{2}\right)W_{2i}(k)e_{2i+2}   
\end{align}
where as usual $e_j$ means the unit vector with $1$ in the $j$th component.

These two result in typical elements  after applying $(\mathrm{Id}-T)$ to the $2i+1$ and $2i$ coordinates of   $(\mathrm{Id}-T)^{N_0}g$. Now it suffices to check all these elements satisfies \eqref{odd} and \eqref{even} depending on the parity.
%\begin{gather*}
 %   (\mathrm{Id}-T)[W_{2i+1}(k)e_{2i+1}]\\=r_{i+1}\left(k+\frac{v_{i+1}}{2}\right)W_{2i+1}({-}k-v_{i+1})e_{2i+2}+\delta^{2i-1}_{2m-3}s_{i+1}\left({-}k-\frac{v_{i+1}}{2}\right)W_{2i+1}(k)e_{2i-1},\\
   
%\end{gather*}
\par We start with $(\mathrm{Id}-T)W_{2i}e_{2i}$. First of all, from the definition of $W_{2i}$ and the fact that \eqref{odd} and \eqref{even} are true when $n\leq N_{0},$ we can verify that
\begin{align*}
    \delta^{2i+2}_{2} s_{i+1}\left(k+\frac{v_{i+1}}{2}\right)W_{2i}(k)e_{2i+2}
\end{align*}
satisfies \eqref{even}. To check $r_{i+1}\left({-}k-\frac{v_{i+1}}{2}\right)W_{2i}({-}k-v_{i+1})e_{2i-1}$, we note that from the definition of  $W_{2i}$, one has
\begin{multline*}
    W_{2i}\left({-}k-v_{i+1}\right)\\=\mathfrak{S}_i({-}k-v_{i+1})r_{i-2h}\left({-}k-v_{i+1}+\frac{v_{i-2h}}{2}\right)[\prod_{\ell=1}^{n_{i}}r_{j(\ell)}({-}k-v_{i+1}+v_{\ell,i})]g_{N(i,n)}\left(\mp k\mp v_{i+1}\pm f_{i,n}(v)\right).
\end{multline*}
We now can verify $(\mathrm{P}1)$. Taking the shift of $r_{i+1}\left({-}k-\frac{v_{i+1}}{2}\right)$ and the shift from $r_{i-2h}\left({-}k-v_{i+1}+\frac{v_{i-2h}}{2}\right)$, one check that
\begin{equation*}
\frac{v_{i-2h}}{2}-v_{i+1}+\frac{v_{i+1}}{2}=\frac{v_{i-2h}-v_{i+1}}{2}\geq \min_{j}\frac{v_{j}-v_{j+1}}{2}.
\end{equation*}
For the remaining shifts, we then get the following inequalities by our   induction hypotheses
\begin{align}
(v_{\ell,i}-v_{i+1})-(\frac{v_{i-2h}}{2}-v_{i+1})=
    v_{\ell,i}-\frac{v_{i-2h}}{2}\geq \min_{j}\frac{v_{j}-v_{j+1}}{2}
\end{align}
and
\begin{equation*}
 (v_{\ell+1,i}-v_{i+1})-(v_{\ell,i}-v_{i+1})= v_{\ell+1,i}-v_{\ell,i}\geq \min_{j}\frac{v_{j}-v_{j+1}}{2}.
\end{equation*}
Therefore, we can deduce from the formula of $W_{2i}\left({-}k-v_{i+1}\right)$ and computations above
\begin{align*}
    r_{i+1}\left({-}k-\frac{v_{i+1}}{2}\right)W_{2i}\left({-}k-v_{i+1}\right)e_{2i+2}
\end{align*}
satisfies the property $(\mathrm{P}1)$ and the condition \eqref{odd}. 
\par Similarly, we can verify that each coordinate of the function $(\mathrm{Id}-T)[W_{2i}(k)e_{2i}]$ has a representation of the form
\begin{equation*}
    R_{N_{0}+1}(e_{2i})+S_{N_{0}+1}(e_{2i})
\end{equation*}
satisfying propositions $(\mathrm{P}1),\,(\mathrm{P}2),\,(\mathrm{P}3),\,(\mathrm{P}4)$ and condition \eqref{even}. Therefore, the Property $(\mathrm{P}4),$ and conditions \eqref{odd} and \eqref{even} are true for $n=N_{0}+1,$ if they are true for $n=N_{0}.$
\par In conclusion, property $(\mathrm{P}4)$ is true for any $n\in\mathbb{N}.$ By induction, the desired decompositions and related properties  hold for any $n\in\mathbb{N}.$
\end{proof}

\begin{definition}
    Given a vector $ \vec{g}\in L^{2}(\mathbb{R}, \mathbb{C}^{2m-2})$, we denote by $ \mathfrak{R}(\vec{g}) \subset L^{2}(\mathbb{R}, \mathbb{C}^{2m-2})$ the collection of vectors whose components are linear combinations of the components of $g$, where the coefficients are products of transmission and reflection coefficients, and each product includes at least one reflection coefficient as a factor.
\end{definition}

\begin{definition}
   Given a vector $ \vec{g}\in L^{2}(\mathbb{R}, \mathbb{C}^{2m-2})$, $R_{q}(\vec{g})(k)\in  L^{2}(\mathbb{R},\mathbb{C}^{2m-2})$ a vector-value function whose $i-th$ coordinate of function is a finite sum of at most $2^{q-1}$ elements, each of them has only one of the the following forms: for some $N_i\in\mathbb{N}$, some shifts $\{f_{i,n}\}\subset \mathbb{R}$ and $\{v_{\ell,i}\}\subset \mathbb{R}$  in terms of $v_1>v_2>\ldots>v_m$,  a function $j_\ell\in\{1,\ldots,m\}$, and some $h_{i,n}\in \{1,\ldots,2m-1\}$
\begin{align*}
    \vec{z}_{i}(k)=&S_{i}(k)[\prod_{\ell=1}^{N_{i}}r_{j_{\ell}}(k+v_{\ell,i})]\tau_{f_{i,n}(v)}g_{h_{i,n}}(\pm k),\\
    \vec{z}_{i}(k)=&S_{i}(k)[\prod_{\ell=1}^{N_{i}}r_{j_{\ell}}({-}k-v_{\ell,i})]\tau_{f_{i,n}(v)}g_{h_{i,n}}(\pm k),
\end{align*}
such that
\begin{equation}\label{eq:Si}
  S_{i}(k)=\prod_{\ell=1}^{N_{i,1}}s_{h_{i}}(\mp k\pm v_{i,n,\ell}),
  \end{equation}
$N_{i}\geq 1$ and $N_{i,1}+N_{i}=q.$

In the following, to simplify notation, we use $R_q(\vec{g})(k)$ to denote any function satisfying the properties described above. It does not represent a specific function. Similarly, $ S_q$ will denote a function given by a product of transmission coefficients of the form~\eqref{eq:Si}. Again,
%since its precise structure is not important, 
it does not refer to a specific function.

\end{definition}

\begin{proposition}\label{RRR}
Let $T:L^{2}_{k}(\mathbb{R},\mathbb{C}^{2m-2})\to L^{2}_{k}(\mathbb{R},\mathbb{C}^{2m-2})$ be a linear bounded operator of \eqref{form1} or \eqref{form2}.
Let $\vec{g}=(g_{1},\,...,\,g_{2m-2})\in L^{2}(\mathbb{R},\mathbb{C}^{2m-2}),$ then %the element
\begin{equation}\label{eq:m-1}
    (\mathrm{Id}-T)^{m-1}(\vec{g})=R_{m}(\vec{g})(k)\in L^{2}_{k}(\mathbb{R},\mathbb{C}^{2}).
\end{equation}
%
%such that each $i-th$ coordinate of function $R_{m}(\vec{g})(k)\in  L^{2}(\mathbb{R},\mathbb{C}^{2m-2})$ is a finite sum of at most $2^{m-2}$ elements, each of them has only one of the the following forms: for some $N_i\in\mathbb{N}$, some shifts $\{f_{i,n}\}\subset \mathbb{R}$ and $\{v_{\ell,i}\}\subset \mathbb{R}$  in terms of $v_1>v_2>\ldots>v_m$,  a function $j_\ell\in\{1,\ldots,m\}$, and some $h_{i,n}\in \{1,\ldots,2m-1\}$
%\begin{align*}
%    \vec{z}_{i}(k)=&S_{i}(k)[\prod_{\ell=1}^{N_{i}}r_{j_{\ell}}(k+v_{\ell,i})]\tau_{f_{i,n}(v)}g_{h_{i,n}}(\pm k),\\
%    \vec{z}_{i}(k)=&S_{i}(k)[\prod_{\ell=1}^{N_{i}}r_{j_{\ell}}({-}k-v_{\ell,i})]\tau_{f_{i,n}(v)}g_{h_{i,n}}(\pm k),
%\end{align*}
%such that
%\begin{equation}\label{eq:Si}
 % S_{i}(k)=\prod_{\ell=1}^{N_{i,1}}s_{h_{i}}(\mp k\pm v_{i,n,\ell}),
  %\end{equation}
%$N_{i}\geq 1$ and $N_{i,1}+N_{i}=m-1.$
\end{proposition}
\begin{proof}
It is enough to prove Proposition \ref{RRR} for operators $T$ of \eqref{form1}, since the proof that the operators of \eqref{form2} satisfy Proposition \ref{RRR} is completely analogous.
\par First, let $\vec{g}=(g_{1},\,...,\,g_{2m-2}).$ From the definition of $(\mathrm{Id}-T),$ firstly, we can verify that
\begin{equation}\label{RR}
    (\mathrm{Id}-T)[ \mathfrak{R}(\vec{g})]\subset  \mathfrak{R}(\vec{g}),
\end{equation}
and
\begin{align}\label{r1}
    [(\mathrm{Id}-T)\vec{g}]_{2i}=& R_{1}(\vec{g})(k)+(1-\delta^{i}_{1})s_{i}\left(k+\frac{v_{i}}{2}\right)g_{2i-2}(k),\\ \label{r2}[(\mathrm{Id}-T)\vec{g}]_{2i+1}= & R_{1}(\vec{g})(k)+(1-\delta^{m-2}_{i})s_{i+2}\left({-}k-\frac{v_{i+1}}{2}\right)g_{2i+3}(k).
\end{align}
%for some $R_{2i}(\vec{g}),\, R_{2i+1}(\vec{g})\in R(\vec{g}).$ In the following, to simply notations, $R_{2i}(\vec{g}),\, R_{2i+1}(\vec{g})$ will be used to denote certain elements from $R(\vec{g})$

%Consequently, we can deduce from \eqref{RR} for any $g(k)e_{2i}$ that\\
\noindent{\bf Claim:}
For any $n\in\mathbb{N},$ given $g\in L^2(\mathbb{R},\mathbb{C})$, the following identities are true.
\begin{align}\label{RRR1}
    (\mathrm{Id}-T)^{n}[g(k)e_{2i}]=&
    R_{n}(g(k)e_{2i})+[S_{n}(k)g(k)]e_{2i+2n}\text{, if $2i+2n\leq 2m-2,$}\\ \label{RRR2}
    (\mathrm{Id}-T)^{n}[g(k)e_{2i+1}]=&
    R_{n}(g(k)e_{2i+1})+[S_{n}(k)g(k)]e_{2i+1-2n} \text{, if $2i+1-2n\geq  1$}
\end{align}
where the expressions $S_{n}(k)$ above are products of transmission coefficients of the form \eqref{eq:Si}.
%,  and  some $R_{2i}(\vec{g}),\, R_{2i+1}(\vec{g})\in R(\vec{g}).$

\noindent{\bf Proof of the Claim:} The proof follows from induction. The case $n=1$ is a direct consequence of the identities \eqref{r1} and \eqref{r2}. Therefore, we can assume the existence of $n_{0}\in\mathbb{N}$ such that \eqref{RRR1} and \eqref{RRR2} are true for any $n\in\{1,\,...,\,n_{0}\}.$ 
\par Next, using \eqref{r1} and \eqref{r2},  we can verify that 
\begin{align*}
    (\mathrm{Id}-T)[S_{n}(k)g(k)e_{2i+2n}]=&R_{n+1}(g(k)e_{2i+2n})+(1-\delta^{i+n+2}_{1})S_{1}(k)S_{n}(k)g(k)e_{2i+2n+2},\\
    (\mathrm{Id}-T)[S_{n}(k)g(k)e_{2j+1-2n}]=&R_{n+1}(g(k)e_{2i+2n})+(1-\delta^{m-2}_{j-n-1})S_{1}(k)S_{n}(k)g(k)e_{2j-2n-1},
\end{align*}
%\textcolor{blue}{ $(\mathrm{Id}-T)[g(k) e_{2\ell}]=s_{\ell+1}\left(k+\frac{v_{\ell+1}}{2}\right)g(k)e_{2\ell+2}+r_{\ell+1}({-}k-\frac{v_{\ell+1}}{2})g({-}k-v_{\ell+1})e_{2\ell-1},$}
%and
%\textcolor{blue}{$(\mathrm{Id}-T)[g(k) e_{2\ell+1}]=s_{\ell+1}\left({-}k-\frac{v_{\ell+1}}{2}\right)g(k)e_{2\ell-1}+r_{\ell+1}(k+\frac{v_{\ell+1}}{2})g({-}k-v_{\ell+1})e_{2\ell+2},$}
%
if $2i+2n+2\leq 2m-2$ and $2j+1-2n-2\geq 1$ respectively. In particular, since $i+n+2>1$ and $i-n-1<m-2$ for any $i\in\{1,\,...,\,m-2\},$ we have that
\begin{align*}
    (\mathrm{Id}-T)[S_{n}(k)g(k)e_{2i+2n}]=&R_{n+1}(g(k)e_{2i+2n})+S_{1}(k)S_{n}(k)g(k)e_{2i+2n+2},\\
    (\mathrm{Id}-T)[S_{n}(k)g(k)e_{2j+1-2n}]=&R_{n+1}(g(k)e_{2i+2n})+S_{1}(k)S_{n}(k)g(k)e_{2j+1-2n-2},
\end{align*}
if $2i+2n+2\leq 2m-2$ and $2j+1-2n-2\geq 1$ respectively. Consequently, using \eqref{RR}, we can conclude that \eqref{RRR1} and \eqref{RRR2} holds for $n+1.$ In conclusion, \eqref{RRR1} and \eqref{RRR2} hold for all $n\in\mathbb{N}.$
\par Next, the identities \eqref{RRR1} and \eqref{RRR2} imply that
\begin{align*}
    (\mathrm{Id}-T)^{m-i-1}[g(k)e_{2i}]=&R_{m-i-1}(g(k)e_{2i})+S_{m-2-i}(k)g(k)e_{2m-2},\,\\
    (\mathrm{Id}-T)^{i}[g(k)e_{2i+1}]=&R_i(g(k)e_{2i+1})+S_{i}(k)g(k)e_{1}.
\end{align*}
Therefore, since \eqref{r1} and \eqref{r2} imply for any $g\in L^{2}_{k}(\mathbb{R},\mathbb{C})$ that
\begin{equation*}
    [(\mathrm{Id}-T)g(k)e_{1}]\in R(\vec{g}),\,[(\mathrm{Id}-T)g(k)e_{2m-2}]\in R(\vec{g}),
\end{equation*}
we conclude that
\begin{equation*}
    (\mathrm{Id}-T)^{i+1}[g(k)e_{2i+1}]\in R(\vec{g}),\,(\mathrm{Id}-T)^{m-i}[g(k)e_{2i}]\in R(\vec{g}). 
\end{equation*}
In conclusion, since $2i$ and $2i+1$ can only take values on the set $\{1,2\,...,\,2m-2\}$ and $(\mathrm{Id}-T)(R(\vec{g}))\subset R(\vec{g}),$ we have
\begin{equation*}
    (\mathrm{Id}-T)^{m-1}(\vec{g})\in R(\vec{g}),
\end{equation*}
for any $\vec{g}\in L^{2}_{k}(\mathbb{R},\mathbb{C}^{2m-2})$, whence, \eqref{eq:m-1}  holds.
\end{proof}

\begin{proof}[Proof of Lemma \ref{lem:claimR}]
First, Propositions \ref{RRR} and \ref{mainprop} imply that for any number $n\in \mathbb{N}$
\begin{equation*}
    \left(\mathrm{Id}-T\right)^{n(m-1)}(\vec{g})
\end{equation*}
is a sum of at the most $2^{n(m-1)}$ terms of the form
\begin{align*}
    \vec{z}_{i}(k)=&S_{i}(k)[\prod_{\ell=1}^{N_{i}}r_{j_{\ell}}(k+v_{\ell,i})]\tau_{f_{i,n}(v)}g_{h_{i,n}}(\pm k)\text{ or }\\
    \vec{z}_{i}(k)=&S_{i}(k)[\prod_{\ell=1}^{N_{i}}r_{j_{\ell}}({-}k-v_{\ell,i})]\tau_{f_{i,n}(v)}g_{h_{i,n}}(\pm k),
\end{align*}
for a natural number $N_{i}\in [n,n(m-1)].$ 
\par Moreover, $(\mathrm{P}1)$ of Proposition \ref{mainprop} implies that for each function $\vec{z}_{i}$ that
\begin{equation*}
    v_{\ell+1,i}-v_{\ell,i}\geq \min_{n}\frac{v_{n}-v_{n+1}}{2}>c>0.
\end{equation*}
Consequently, we can deduce using Lemma \ref{factUpp} and Proposition \ref{mainprop} that each of one of the terms $\vec{z}_{i}$ of $(\mathrm{Id}-T)^{n(m-1)}(\vec{g})$ satisfies 
\begin{align*}
    \norm{\vec{z}_{i}}_{L^{2}_{k}(\mathbb{R})}\leq & \max_{h\in\{1,2\}}\frac{C^{n} \norm{\vec{g}(k)}_{L^{2}_{k}(\mathbb{R})}}{\left(\lfloor \frac{n-2}{2}\rfloor!\right)^{2} c^{n-h}},\\
    \norm{\vec{z}_{i}}_{\mathcal{F}L^{1}}\leq & \max_{h\in\{1,2\}}\frac{m nC^{n} \norm{\vec{g}(k)}_{\mathcal{F}L^{1}}}{\left(\lfloor \frac{n-3}{2}\rfloor!\right)^{2} c^{n-2-h}}.
\end{align*}
In conclusion,
\begin{align*}
    \norm{\left(\mathrm{Id}-T\right)^{n(m-1)}(\vec{g})}_{L^{2}_{k}(\mathbb{R})}\leq & \max_{h\in\{1,2\}}\frac{2^{n(m-1)}C^{n} \norm{\vec{g}(k)}_{L^{2}_{k}(\mathbb{R})}}{\left(\lfloor \frac{n-2}{2}\rfloor!\right)^{2} c^{n-h}},\\
     \norm{\left(\mathrm{Id}-T\right)^{n(m-1)}(\vec{g})}_{\mathcal{F}L^{1}}\leq & \max_{h\in\{1,2\}}\frac{m n 2^{n(m-1)}C^{n} \norm{\vec{g}(k)}_{\mathcal{F}L^{1}}}{\left(\lfloor \frac{n-3}{2}\rfloor!\right)^{2} c^{n-2-h}},
\end{align*}
since $(\mathrm{Id}-T)^{n(m-1)}$ is a sum of $2^{n(m-1)}$ functions $\vec{z}_{i},$ from which we deduce that Lemma \ref{lem:claimR} holds for
\begin{equation*}
    C(m)=\frac{2^{m-1}C}{(\min{c,1})}
\end{equation*}as desired.
\end{proof}

\section{Analysis of dispersive map $\mathcal{S}(t)(\vec{\phi})$}\label{sec:dispersivemap}
First, we introduce the definition below which will be useful in the computations of the estimates in this section.
\begin{definition}
If $\vec{f},\,\vec{g}\in L^{2}_{k}(\mathbb{R},\mathbb{C}^{2}),$ we say that
$
    \vec{f}(k)\cong_{2} \vec{g}(k),
$
if $\min_{\ell}y_{\ell}-y_{\ell+1}>1$ is large enough and there exists a positive $\beta$ such that 
\begin{align*}
    \norm{\vec{f}(k)-\vec{g}(k)}_{L^{2}_{k}(\mathbb{R})}\lesssim & e^{{-}\beta \min_{\ell}(y_{\ell}-y_{\ell+1})}\left[\max_{1\leq \ell\leq m-1 }\norm{\begin{bmatrix}
        \phi_{1,\ell}(k)\\
        \phi_{2,\ell}(k)
    \end{bmatrix}}_{L^{2}_{k}(\mathbb{R})} \right],\\
    \norm{(1+\vert x\vert)\left[\widehat{f}(x)-\widehat{g}(x)\right]}_{L^{1}_{x}(\mathbb{R})}\lesssim & e^{{-}\beta \min_{\ell}(y_{\ell}-y_{\ell+1})}\left[\max_{1\leq \ell\leq m-1 }\norm{\begin{bmatrix}
        \phi_{1,\ell}(k)\\
        \phi_{2,\ell}(k)
    \end{bmatrix}}_{L^{2}_{k}(\mathbb{R})} \right],
\end{align*}
where $\widehat{f}$ and $\widehat{g}$ is the free Fourier transform of $\vec{f}$ and $\vec{g}$ respectively.
\end{definition}

\subsection{Invertibility of the dispersive map }

Next, from \cite{dispa}, we consider the map $(B_{1},\,...,\,B_{2m-2}):L^{2}_{x}(\mathbb{R},\mathbb{C}^{2})\to L^{2}_{x}(\mathbb{R},\mathbb{C}^{2m})$ defined by
\begin{align*}
    B_{1}(h)\coloneqq  & \sigma_{3} F^{*}_{\omega_{1}}\left[\sigma_{3} 1_{\{x\geq \frac{y_{2}-y_{1}}{2}\}}(x)\tau_{y_{1}}\left(e^{{-}i\frac{v_{1}x}{2}\sigma_{3}}h(x)\right)\right] (k),\\
B_{2n}(h) \coloneqq  & \sigma_{3} G^{*}_{\omega_{n+1}}\left[\sigma_{3} 1_{\{\frac{y_{n+2}-y_{n+1}}{2}\leq x\leq \frac{y_{n}-y_{n+1}}{2}\}}(x)\tau_{y_{n+1}}\left(e^{{-}i\frac{v_{n+1}x}{2}\sigma_{3}}h(x)\right)\right](k)\text{ if $n+1<m,$}
    \\
    B_{2n+1}(h) \coloneqq  & \sigma_{3} F^{*}_{\omega_{n+1}}\left[\sigma_{3} 1_{\{\frac{y_{n+2}-y_{n+1}}{2}\leq x\leq \frac{y_{n}-y_{n+1}}{2}\}}(x)\tau_{y_{n+1}}\left(e^{{-}i\frac{v_{n+1}x}{2}\sigma_{3}}h(x)\right)\right](k) \text{ if $n+1<m,$}\\
    B_{2m-2}(h)\coloneqq  & \sigma_{3} G^{*}_{\omega_{m}}\left[\sigma_{3} 1_{\{x\leq \frac{y_{m-1}-y_{m}}{2}\}}(x)\tau_{y_{m}}\left(e^{{-}i\frac{v_{m}x}{2}\sigma_{3}}h(x)\right)\right](k),
\end{align*}
for any function $h(x)\in L^{2}_{x}(\mathbb{R},\mathbb{C}^{2}).$ 
\par Next, consider the following orthogonal projections $P_{\pm}:L^{2}_{x}(\mathbb
b{R},\mathbb{C}^{2})\to L^{2}_{x}(\mathbb
b{R},\mathbb{C}^{2})$
\begin{align}\label{P+}
P_{+}(\vec{f})(x)\coloneqq & \frac{1}{2\pi}\int_{{-}\infty}^{0}e^{{-}ikx}\left[\int_{\mathbb{R}}e^{iky}\vec{f}(y)\,dy\right]dk,\\ \label{P-}
P_{-}(\vec{f})(x)\coloneqq & \frac{1}{2\pi}\int_{0}^{{+}\infty}e^{{-}ikx}\left[\int_{\mathbb{R}}e^{iky}\vec{f}(y)\,dy\right]dk.
\end{align}
In particular, the projections $P_{+}$ and $P_{-}$ are exactly the orthogonal projections of $L^{2}_{x}(\mathbb{R})$ onto the Hardy spaces $H^{2}(\mathbb{C}_{+})$ and $H^{2}(\mathbb{C}_{-})$ respectively, see Definition $3.1$ and Theorem $3.2$ of \cite{dispa} for more details. We will also use the following Proposition.
\begin{lemma}\label{+-interact}
Let $f_{\pm}$ be in $H^{2}(\mathbb{C}_{\pm}),$
and let $s(k),\,r(k)$ be two $C^{1}$ functions on the strip $$\vert\I k \vert \leq \delta$$ satisfying 
\begin{equation*}
 \left\vert 
\frac{d^{j}}{dk^{j}}r(k)\right\vert +\left\vert \frac{d^{j}}{dk^{j}}\left[1-s(k)\right]\right\vert\lesssim \frac{1}{\left(1+\vert k\vert\right)^{1+j}} \text{, for any $j\in\{0,1\}.$}    
\end{equation*}
 If $y_{0}>1,\,h_{0}\in\mathbb{R},$ then
\begin{align}\label{p-1e}
\norm{ P_{\mp}\left(e^{\pm iy_{0}k}\left[s(k+h_{0})f_{\pm}(k)\right]\right)(x)}_{L^{2}(\mathbb{R},dx)}\leq &Ce^{{-}\frac{99}{100}\delta y_{0}}\norm{f_\pm}_{H^{2}(\mathbb{C}_\pm)},\\ \label{p-2e}
     \norm{ P_{\mp}\left(e^{\pm iy_{0}k}\left[r(k+h_{0})f_{\pm}(k)\right]\right)(x)}_{L^{2}(\mathbb{R},dx)}\leq &Ce^{{-}\frac{99}{100}\delta y_{0}}\norm{f_\pm}_{H^{2}(\mathbb{C}_\pm)}.
     %\\ \label{p-3e}
     %\norm{ P_{\mp}\left(e^{\pm i y_{0}k}\left[s(k)f_{\pm}(k)\right]\right)(x)}_{L^{2}(\mathbb{R},dx)}\leq &Ce^{{-}\frac{99}{100}\delta y_{0}}\norm{f_\pm}_{H^{2}(\mathbb{C}_\pm)}.
\end{align}
\end{lemma}
\begin{proof}
See the proof of Lemma $3.4$ from \cite{dispa}. In particular, using Remark $3.5$ of \cite{dispa}, we can restrict to the case where $r$ and $s$ are analytics and bounded on the strip having $\vert \I k\vert\leq \delta.$    
\end{proof}
\par Furthermore, using Remark $4.2$ from \cite{dispa} and Lemma \ref{leper}, we can verify when $\min_{\ell}v_{\ell}-v_{\ell+1}>0$ and $\min_{\ell}y_{\ell}-y_{\ell+1}>1$ is large enough that $B\left(\mathcal{S}(\vec{\phi})(0,\Diamond)\right)(k)$ can be estimate by the following functions below
\begin{align}\label{systemB}
B_{1}\left(\mathcal{S}(\vec{\phi})(0,\Diamond)\right)(k)\cong_{2}  & e^{i\frac{k(y_{1}-y_{2})}{2}}P_{-}\left(e^{i\frac{\Diamond(y_{1}+y_{2})}{2}}
 \begin{bmatrix}
    \phi_{1}\left(\Diamond+\frac{v_{1}}{2}\right)\\
    \phi_{2}\left(\Diamond-\frac{v_{1}}{2}\right)
 \end{bmatrix}\right) (k)\nonumber  
\\ & {-}r_{\omega_{1}}({-}k)e^{{-}i\frac{k(y_{1}-y_{2})}{2}}\left[P_{+}\left(e^{i\frac{\Diamond (y_{1}+y_{2})}{2}}\begin{bmatrix}
    \phi_{1}\left(\Diamond+\frac{v_{1}}{2}\right)\\
    \phi_{2}\left(\Diamond-\frac{v_{1}}{2}\right)
 \end{bmatrix}\right)\right]({-}k),\\ \label{b2n}
  B_{2n}\left(\mathcal{S}(\vec{\phi})(0,\Diamond)\right)(k)\cong_{2} & e^{{-}i\frac{k(y_{n}-y_{n+1})}{2}}P_{+}\left(\begin{bmatrix}
    e^{i\frac{(y_{n}+y_{n+1})\Diamond}{2}}\phi_{1,n}(\Diamond+\frac{v_{n+1}}{2})\\ 
    e^{i\frac{(y_{n}+y_{n+1})\Diamond}{2}}\phi_{2,n}(\Diamond-\frac{v_{n+1}}{2})
    \end{bmatrix}\right)(k)\\ \nonumber
    &{-}r_{n+1}(k)e^{i\frac{k(y_{n}-y_{n+1})}{2}}P_{-}\left(\begin{bmatrix}
    e^{i\frac{(y_{n}+y_{n+1})\Diamond}{2}}\phi_{1,n}(\Diamond+\frac{v_{n+1}}{2})\\
    e^{i\frac{(y_{n}+y_{n+1})\Diamond}{2}}\phi_{2,n}(\Diamond-\frac{v_{n+1}}{2})
    \end{bmatrix}\right)({-}k)\\ \nonumber
    &{-}s_{n+1}(k)e^{{-}i\frac{k(y_{n+2}-y_{n+1})}{2}}P_{+}\left(\begin{bmatrix}
e^{i\frac{(y_{n+1}+y_{n+2})\Diamond}{2}}\phi_{1,n+1}(\Diamond+\frac{v_{n+1}}{2})\\
e^{i\frac{(y_{n+1}+y_{n+2})\Diamond}{2}}\phi_{2,n+1}(\Diamond-\frac{v_{n+1}}{2})
\end{bmatrix}\right)(k),\\ \label{b2n+1}
 B_{2n+1}\left(\mathcal{S}(\vec{\phi})(0,\Diamond)\right)(k)\cong_{2} & e^{{-}i\frac{k(y_{n+2}-y_{n+1})}{2}}P_{-}\left(\begin{bmatrix}
e^{i\frac{(y_{n+1}+y_{n+2})\Diamond}{2}}\phi_{1,n+1}(\Diamond+\frac{v_{n+1}}{2})\\
e^{i\frac{(y_{n+1}+y_{n+2})\Diamond}{2}}\phi_{2,n+1}(\Diamond-\frac{v_{n+1}}{2})
\end{bmatrix}
\right)(k)\\ \nonumber
    &{-}r_{n+1}({-}k)e^{i\frac{k(y_{n+2}-y_{n+1})}{2}}P_{+}\left(\begin{bmatrix}
e^{i\frac{(y_{n+1}+y_{n+2})\Diamond}{2}}\phi_{1,n+1}(\Diamond+\frac{v_{n+1}}{2})\\
e^{i\frac{(y_{n+1}+y_{n+2})\Diamond}{2}}\phi_{2,n+1}(\Diamond-\frac{v_{n+1}}{2})
\end{bmatrix}\right)({-}k)\\ \nonumber
    &{-}s_{n+1}({-}k)e^{{-}i\frac{k(y_{n}-y_{n+1})}{2}}P_{-}\left(
    \begin{bmatrix}
e^{i\frac{(y_{n}+y_{n+1})k}{2}}\phi_{1,n}(k+\frac{v_{n+1}}{2})\\
e^{i\frac{(y_{n}+y_{n+1})k}{2}}\phi_{2,n}(k-\frac{v_{n+1}}{2})
\end{bmatrix}\right),\\ \label{b2m-2}
 B_{2m-2}\left(\mathcal{S}(\vec{\phi})(0,\Diamond)\right)(k)\cong_{2} & e^{{-}ik\frac{(y_{m-1}-y_{m})}{2}}P_{+}\left(\begin{bmatrix}
      e^{i\frac{(y_{m-1}+y_{m})\Diamond}{2}}\phi_{1,m-1}(\Diamond+\frac{v_{m}}{2})\\
        e^{i\frac{(y_{m-1}+y_{m})\Diamond}{2}}\phi_{2,m-1}(\Diamond-\frac{v_{m}}{2})
  \end{bmatrix}\right)(k)\\ \nonumber
  &{-}r_{m}(k)e^{ik\frac{(y_{m-1}-y_{m})}{2}}\left[P_{-}\left(\begin{bmatrix}
      e^{i\frac{(y_{m-1}+y_{m})\Diamond}{2}}\phi_{1,m-1}(\Diamond+\frac{v_{m}}{2})\\
        e^{i\frac{(y_{m-1}+y_{m})\Diamond}{2}}\phi_{2,m-1}(\Diamond-\frac{v_{m}}{2})
  \end{bmatrix}\right)\right]({-}k),
\end{align}
for more details, see Lemma $4.9$ from \cite{dispa}.
%Moreover, the functions $r_{\ell}(k),\,$ satisfies the following decay estimate
%\begin{equation}\label{decayrss}
%    \max_{\ell}\left\vert r_{\ell}(k)\right\vert+ \left\vert \left(1-s_{\ell}(k)\right) \right\vert=O\left(\frac{1}{(1+\vert k\vert)}\right) \text{, for all $k\in\mathbb{R}.$}
%\end{equation}
\par Next, aiming an application of Lemma \ref{lem:claimR}, on the estimates \eqref{systemB}-\eqref{b2m-2}, we consider from now on the following notation 
\begin{align}\label{BSnotation}
    \begin{bmatrix}
        B_{\ell,1}(k)\\
        B_{\ell,2}(k)
    \end{bmatrix}\coloneqq & B_{\ell}\left(\mathcal{S}(\vec{\phi})(0,\Diamond)\right)(k), \\
    \begin{bmatrix}
        g_{2\ell-1,1}(k)\\
        g_{2\ell-1,2}(k)
    \end{bmatrix}\coloneqq & P_{-}\left(e^{i\frac{k(y_{\ell}+y_{\ell+1})}{2}}\begin{bmatrix}
        \phi_{1,\ell}(k)\\
        \phi_{2,\ell}(k)
    \end{bmatrix}\right) \text{, for any $\ell\in [m-1]$}\\
    \begin{bmatrix}
        g_{2\ell,1}(k)\\
        g_{2\ell,2}(k)
    \end{bmatrix}\coloneqq & P_{+}\left(e^{i\frac{k(y_{\ell}+y_{\ell+1})}{2}}\begin{bmatrix}
        \phi_{1,\ell}(k)\\
        \phi_{2,\ell}(k)
    \end{bmatrix}\right) \text{, for any $\ell\in [m-1].$}
\end{align}
In particular, the estimates \eqref{systemB}, \eqref{b2n}, \eqref{b2n+1} and \eqref{b2m-2} imply the existence of real constants $\theta_{c_{\ell},j}$ and $\theta_{d_{\ell},j}$ satisfying the following equivalence relations for any $j\in\{1,2\}$ and $\ell\in [2m-2].$
\begin{align}\label{systemBB}
    e^{\frac{{-}ik(y_{1}-y_{2})}{2}}e^{\frac{({-}1)^{j}i v_{1}y_{1}}{2}}B_{1,j}\left(k+\frac{({-}1)^{j}v_{1}}{2}\right)\cong_{2} & g_{1,j}(k) \\& -r_{1}\left({-}k{+} \frac{({-}1)^{j} v_{1}}{2}\right)e^{{-}ik(y_{1}-y_{2})}e^{i\theta_{c_{1},j}}g_{2,j}({-}k-  ({-}1)^{j}v_{1}),
\end{align}
\begin{multline} \nonumber
    e^{\frac{ik(y_{n}-y_{n+1})}{2}}e^{\frac{({-}1)^{j+1}iv_{n+1}y_{n+1}}{2}}B_{2n,j}\left(k+\frac{({-}1)^{j}v_{n+1}}{2}\right)\\
    \begin{aligned}
    \cong_{2} & g_{2n,j}(k)\\&{-}r_{n+1}\left(k+ \frac{({-}1)^{j}v_{n+1}}{2}\right)e^{ik(y_{n}-y_{n+1})}e^{i\theta_{c_{2n},j}}g_{2n-1,j}({-}k- ({-}1)^{j} v_{n+1})\\&{-}s_{n+1}\left(k+ \frac{({-}1)^{j}v_{n+1}}{2}\right)e^{\frac{ik(y_{n}-y_{n+2})}{2}}e^{i\theta_{d_{2n},j}}g_{2n+2,j}(k),
    \end{aligned}
  \end{multline}  
  \begin{multline}
    e^{\frac{ik(y_{n+2}-y_{n+1})}{2}}e^{\frac{({-}1)^{j+1}i v_{n+1}y_{n+1}}{2}}B_{2n+1,j}\left(k+\frac{({-}1)^{j}v_{n+1}}{2}\right)\\
    \begin{aligned}
    \cong_{2} & g_{2n+1,j}(k)\\& {-}r_{n+1}\left({-}k+ \frac{({-}1)^{j+1}v_{n+1}}{2}\right)e^{ik(y_{n+2}-y_{n+1})}e^{i\theta_{c_{2n+1},j}}g_{2n+2,j}\left({-}k- ({-}1)^{j}v_{n+1}\right)\\ 
    & {-}s_{n+1}\left({-k}+\frac{({-}1)^{j+1}v_{n+1}}{2}\right)e^{\frac{ik(y_{n+2}-y_{n})}{2}}e^{i\theta_{d_{2n+1},j}}g_{2n-1,j}(k),\,...,\\ 
    \end{aligned}
\end{multline}
   \begin{multline*}
    e^{\frac{ik(y_{m-1}-y_{m})}{2}}e^{\frac{({-}1)^{j+1}v_{m}y_{m}}{2}}B_{2m-2,j}\left(k+\frac{({-}1)^{j}v_{m}}{2}\right)\\
    \begin{aligned}
    \cong_{2}& g_{2m-2,j}(k)\\&{-}r_{m}\left( k+ \frac{({-}1)^{j}v_{m}}{2}\right)e^{ik(y_{m-1}-y_{m})}e^{i\theta_{c_{2m-2},j}}g_{2m-3,j}({-}k+({-}1)^{j+1} v_{m}).
    \end{aligned}
\end{multline*}
Moreover, Lemma $B.4$ of \cite{dispa} implies that
\begin{align}\label{estSS1}
    \max_{\ell\in [2m-2],j\in\{1,2\}}\norm{B_{\ell,j}(k)}_{L^{2}_{k}(\mathbb{R})}\lesssim &\norm{\mathcal{S}(\vec{\phi})(0,x)}_{L^{2}_{x}(\mathbb{R})},\\ \label{estSS2}
    \max_{\ell\in [2m-2],j\in\{1,2\}}\norm{B_{\ell,j}(k)}_{\mathcal{F}L^{1}}\lesssim &\norm{\mathcal{S}(\vec{\phi})(0,x)}_{L^{1}_{x}(\mathbb{R})},\\ \label{estSS3}
    \max_{\ell\in [2m-2],j\in\{1,2\}}\norm{\frac{d}{dk}B_{\ell,j}(k)}_{\mathcal{F}L^{1}}\lesssim &\max_{\ell\in [m]}\norm{\chi_{\ell}(0,x)\langle x \rangle\mathcal{S}(\vec{\phi})(0,x)}_{L^{1}_{x}(\mathbb{R})}.
\end{align}
\par Next, since each pair of functions $r_{\ell}$ and $s_{\ell}$ satisfy \eqref{decayrs1}, we can verify for each $h\in\{1,2\}$ that the right-hand side of the estimates \eqref{systemBB} is an linear bounded map $T_{h}:L^{2}(\mathbb{R},\mathbb{C}^{2m-2})\to L^{2}(\mathbb{R},\mathbb{C}^{2m-2})$ of \eqref{form2} given by
\begin{equation*}
    T_{j}=\mathrm{Id}-R_{j}
\end{equation*}
such that $R_{j}$ satisfies Claim $(\mathrm{R}),$ see \eqref{claimR} for more details. 
\par Consequently, using Claim $(\mathrm{R})$ in the linear equations on the right-hand side of \eqref{systemBB}, we can deduce when $\min_{\ell}y_{\ell}-y_{\ell+1}>1$ is large enough and $m\in\mathbb{N}_{\geq 1}$ that 
\begin{equation}\label{Id-Bestimate}
    \max_{h\in\{1,2\}}\norm{(\mathrm{Id}-T_{h})^{j(m-1)}}_{L^{2}\to L^{2}}+ \max_{h\in\{1,2\}}\norm{(\mathrm{Id}-T_{h})^{j(m-1)}}_{\mathcal{F}L^{1}\to \mathcal{F}L^{1}}\leq \min\left(\frac{jmC(m)^{j}}{j!},\frac{jmC(m)^{j}}{(\frac{j}{2}!)^{2}}\right).
\end{equation}
In particular, \eqref{Id-Bestimate} implies that
\begin{equation*}
    \max_{B\in\{\mathcal{F}L^{1},L^{2}\}}\max_{h\in\{1,2\}}\sum_{\ell=1}^{{+}\infty} \norm{(\mathrm{Id}-T_{h})^{\ell}}_{B\to B}<{+}\infty,
\end{equation*}
from which we deduce that $T_{h}:B\to B$ is a linear homeomorphism for each Banach Spaces $B$ of $\{\mathcal{F}L^{1},L^{2}\}.$
\par Therefore, we deduce the existence of a constant $C(\min_{\ell}v_{\ell}-v_{\ell+1})>0$ depending only on $\min_{\ell}v_{\ell}-v_{\ell+1}$ satisfying the following estimate.
\begin{align}\label{fl1coercive}
    \max_{\ell\in [2m]}\norm{g_{\ell}(k)}_{\mathcal{F}L^{1}}\leq & C(\min_{\ell}v_{\ell}-v_{\ell+1})\min_{h\in\{1,2\}}\norm{T_{h}(\vec{g})(k)}_{\mathcal{F}L^{1}},\\ \label{l2coercccc}
    \max_{\ell\in [2m]}\norm{g_{\ell}(k)}_{L^{2}}\leq & C(\min_{\ell}v_{\ell}-v_{\ell+1})\min_{h\in\{1,2\}}\norm{T_{h}(\vec{g})(k)}_{L^{2}}.
\end{align}
Similarly, applying the partial derivative $\partial_{k}$ on the system of equations \eqref{systemB}, we can obtain using \eqref{fl1coercive} that there exists a constant $C(\min_{\ell}v_{\ell}-v_{\ell+1})>0$ depending on $\min_{\ell}v_{\ell}-v_{\ell+1}$ satisfying
\begin{equation}\label{dfl1coercive}
    \max_{\ell\in [2m]}\norm{\partial_{k}g_{\ell}(k)}_{\mathcal{F}L^{1}}\leq C(\min_{\ell}v_{\ell}-v_{\ell+1})\min_{h\in\{1,2\}}\left[(\max_{j}y_{j}-y_{j+1})\norm{T_{h}(\vec{g})(k)}_{\mathcal{F}L^{1}}+\norm{\partial_{k}T_{h}(\vec{g})(k)}_{\mathcal{F}L^{1}}\right].
\end{equation}
\par In conclusion, using the estimates \eqref{estSS1}-\eqref{estSS2} and the coercive inequalities \eqref{fl1coercive}, \eqref{l2coercccc}, we can deduce when $\min_{\ell}y_{\ell}-y_{\ell+1}$ is large enough that
\begin{align}
    \norm{\mathcal{S}(0)(\vec{\phi})}_{L^{2}_{x}(\mathbb{R})}\geq & c(\min_{\ell}v_{\ell}-v_{\ell+1})\max_{j\in[m]} \norm{\vec{\phi}_{j}}_{L^{2}_{x}(\mathbb{R})},\\ \label{dss000}
    \max_{\ell}\norm{\langle x-y_{\ell} \rangle \chi_{\ell}(x)\mathcal{S}(0)(\vec{\phi})}_{L^{2}_{x}(\mathbb{R})}\geq & c(\min_{\ell}v_{\ell}-v_{\ell+1})\max_{j\in[m]} \norm{\vec{\phi}_{j}}_{\mathcal{F}L^{1}_{x}(\mathbb{R})},
\end{align}
such that $c(\min_{\ell}v_{\ell}-v_{\ell+1})>0$ is a parameter depending only on $\min_{\ell}v_{\ell}-v_{\ell+1}>0.$ Moreover, when the hypothesis $(H4)$ is true, we can use Lemma $B.4$ from \cite{dispa} to improve the estimate \eqref{dss000} by the following one
\begin{equation}\label{dss0}
   \norm{\mathcal{S}(0)(\vec{\phi})}_{L^{1}_{x}(\mathbb{R})}\geq  c(\min_{\ell}v_{\ell}-v_{\ell+1})\max_{j\in[m]} \norm{\vec{\phi}_{j}}_{\mathcal{F}L^{1}_{x}(\mathbb{R})} 
\end{equation}
\par Furthermore, using estimates \eqref{estSS2}, \eqref{estSS3} and \eqref{dfl1coercive}, we can deduce similarly to the proof of Lemma $4.2$ from \cite{dispa} the following estimate.
\begin{multline}\label{dS(0)}
\max_{\ell}\norm{\chi_{\ell,n-1}(0,x)\langle x-y_{\ell} \rangle \mathcal{S}(0)(\vec{\phi})(x)}_{L^{1}_{x}(\mathbb{R})}+\max_{\ell}(y_{\ell}-y_{\ell+1})\norm{\mathcal{S}(0)(\vec{\phi})(x)}_{L^{1}_{x}(\mathbb{R})}\\
\geq  c(\min_{\ell}v_{\ell}-v_{\ell+1})\max_{j\in[m]} \norm{\partial_{k}\vec{\phi}_{j}}_{\mathcal{F}L^{1}_{x}(\mathbb{R})}.
\end{multline}
\par Furthermore, using Definition \ref{s0def} and the asymptotic behavior of $\hat{G}_{\omega}$ and $\hat{F}_{\omega},$ we can verify using Fourier analysis $\mathcal{S}(0)$ satisfies the following inequalities.
\begin{align*}
    \norm{\mathcal{S}(\vec{\phi})(0,x)}_{H^{n}_{x}(\mathbb{R})}\lesssim &\max_{\ell\in [m]}\norm{\langle k \rangle^{n}\begin{bmatrix}
        \phi_{1,\ell}(k)\\
        \phi_{2,\ell}(k)
    \end{bmatrix}} \text{, for all $n\in\{0,1,2\}$}\\
    \norm{\mathcal{S}(\vec{\phi})(0,x)}_{L^{1}_{x}(\mathbb{R})}\lesssim & \Bigg[ \max_{\ell\in [m],j\in\{1,2\}}\norm{\widehat{\phi_{\ell,j}}(x)}_{L^{1}_{x}(\mathbb{R})}\\&{+} \max_{\ell\in\{1,m\}}\norm{\hat{G}_{\omega_{\ell}}\left(e^{iy_{\ell}k}\begin{bmatrix}
      \phi_{1,\ell}(k+\frac{v_{\ell}}{2}))\\
      \phi_{2,\ell}(k-\frac{v_{\ell}}{2})) 
   \end{bmatrix}\right)(x)}_{L^{1}_{x}(\mathbb{R})}\Bigg],\\
   \norm{\chi_{\ell}(0,x)\langle x-y_{\ell}\rangle\mathcal{S}(\vec{\phi})(0,x)}_{L^{1}_{x}(\mathbb{R})}\lesssim  & \langle y_{1}-y_{m}\rangle\max_{\ell\in [m],j\in\{1,2\}}\norm{\widehat{\phi_{\ell,j}}(x)}_{L^{1}_{x}(\mathbb{R})}\\
    & {+} \langle y_{1}-y_{m}\rangle \max_{\ell\in\{1,m\}}\norm{\hat{G}_{\omega_{\ell}}\left(e^{iy_{\ell}k}\begin{bmatrix}
      \phi_{1,\ell}(k+\frac{v_{\ell}}{2}))\\
      \phi_{2,\ell}(k-\frac{v_{\ell}}{2})) 
   \end{bmatrix}\right)(x)}_{L^{1}_{x}(\mathbb{R})}\\
   &{+}  \max_{\ell\in [m],j\in\{1,2\}}\norm{x\widehat{\phi_{\ell,j}}(x)}_{L^{1}_{x}(\mathbb{R})}\\
   &{+}\max_{\ell\in\{1,m\}}\norm{\langle x\rangle \hat{G}_{\omega_{\ell}}\left(e^{iy_{\ell}k}\begin{bmatrix}
      \phi_{1,\ell}(k+\frac{v_{\ell}}{2}))\\
      \phi_{2,\ell}(k-\frac{v_{\ell}}{2})) 
   \end{bmatrix}\right)(x)}_{L^{1}_{x}(\mathbb{R})}.
\end{align*}
As a consequence, for any set of speeds $v_{\ell}$ satisfying $\min_{\ell}v_{\ell}-v_{\ell+1}>0,$  we can verify that the map $\mathcal{S}(0)$ satisfies the following proposition, which extends the results in Theorem $4.2$ from \cite{dispa} for all the cases without large separation in the speeds.
\begin{proposition}\label{Coercivityproperty}
   Let $(\mathrm{H}1)-(\mathrm{H}3)$ be true. There exists $M(\vec{v})>1,\, C>1,\,c>0$ depending uniquely on $\min_{\ell}v_{\ell}-v_{\ell+1}>0$ such that the map $\mathcal{S}$ satisfies for any $t\geq 0$ and any $\vec{\phi}\in L^{2}_{k}(\mathbb{R})$ the following estimates.
\begin{align}\label{est01}
   c\norm{\mathcal{S}(\vec{\phi})(0,x)}_{H^{n}_{x}(\mathbb{R})} \leq \max_{\ell\in [m]}\norm{\langle k \rangle^{n}\begin{bmatrix}
        \phi_{1,\ell}(k)\\
        \phi_{2,\ell}(k)
    \end{bmatrix}}\leq &C\norm{\mathcal{S}(\vec{\phi})(0,x)}_{H^{n}_{x}(\mathbb{R})} \text{, for any $n\in\{0,1,2\}.$}
    \end{align}
    In addition, if $(\mathrm{H}4)$ is true, then
    \begin{align}
    \label{est02}
   c \max_{\ell\in [m],j\in\{1,2\}}\norm{\widehat{\phi_{\ell,j}}(x)}_{L^{1}_{x}(\mathbb{R})}&\leq  \norm{\mathcal{S}(\vec{\phi})(0,x)}_{L^{1}_{x}(\mathbb{R})}\leq  C \Bigg[ \max_{\ell\in [m],j\in\{1,2\}}\norm{\widehat{\phi_{\ell,j}}(x)}_{L^{1}_{x}(\mathbb{R})}\\&{+} \max_{\ell\in\{1,m\}}\norm{\hat{G}_{\omega_{\ell}}\left(e^{iy_{\ell}k}\begin{bmatrix}
      \phi_{1,\ell}(k+\frac{v_{\ell}}{2}))\\
      \phi_{2,\ell}(k-\frac{v_{\ell}}{2})) 
   \end{bmatrix}\right)(x)}_{L^{1}_{x}(\mathbb{R})}\Bigg],\\ \label{est03}
    \max_{\ell\in [m],j\in\{1,2\}}\norm{x\widehat{\phi_{\ell,j}}(x)}_{L^{1}_{x}(\mathbb{R})}\leq & C\max_{j\in\{0,1\}} \langle y_{1}-y_{m}\rangle ^{j}\max_{\ell}\norm{\chi_{\ell}(0,x)\langle x-y_{\ell}\rangle^{1-j}\mathcal{S}(\vec{\phi})(0,x)}_{L^{1}_{x}(\mathbb{R})},\\ \label{est04}
    \norm{\chi_{\ell}(0,x)\langle x-y_{\ell}\rangle\mathcal{S}(\vec{\phi})(0,x)}_{L^{1}_{x}(\mathbb{R})}\leq C & \langle y_{1}-y_{m}\rangle\max_{\ell\in [m],j\in\{1,2\}}\norm{\widehat{\phi_{\ell,j}}(x)}_{L^{1}_{x}(\mathbb{R})}\\
    & {+}C \langle y_{1}-y_{m}\rangle \max_{\ell\in\{1,m\}}\norm{\hat{G}_{\omega_{\ell}}\left(e^{iy_{\ell}k}\begin{bmatrix}
      \phi_{1,\ell}(k+\frac{v_{\ell}}{2}))\\
      \phi_{2,\ell}(k-\frac{v_{\ell}}{2})) 
   \end{bmatrix}\right)(x)}_{L^{1}_{x}(\mathbb{R})}\\
   &{+}C  \max_{\ell\in [m],j\in\{1,2\}}\norm{x\widehat{\phi_{\ell,j}}(x)}_{L^{1}_{x}(\mathbb{R})}\\
   &{+}C\max_{\ell\in\{1,m\}}\norm{\langle x\rangle \hat{G}_{\omega_{\ell}}\left(e^{iy_{\ell}k}\begin{bmatrix}
      \phi_{1,\ell}(k+\frac{v_{\ell}}{2}))\\
      \phi_{2,\ell}(k-\frac{v_{\ell}}{2})) 
   \end{bmatrix}\right)(x)}_{L^{1}_{x}(\mathbb{R})}.
\end{align}
\end{proposition}

\subsection{Dispersive estimates for dispersive maps}
%We set $\vec{u}(t,x)=\mathcal{S}(\vec{\phi}_0)(t,x)$. 
In this subsection, we study the dispersive properties of the function given by dispersive map $\mathcal{S}(\vec{\phi})(t,x)$.

\begin{theorem}\label{decaySphi}
If $\min_{\ell}y_{\ell}-y_{\ell+1}>L,$ then $\vec{u}(t,x)=\mathcal{S}(\vec{\phi})(t,x)$ satisfies for constants $C >c>0$ depending only on $\min_{\ell}v_{\ell}-v_{\ell+1}>0$ the following estimates for all $t\geq s\geq 0.$
\begin{align}\label{CCC1Sphi}
  c \norm{\vec{u}(0,x)}_{H^{j}_{x}(\mathbb{R})} \leq & \norm{\vec{u}(t,x)}_{H^{j}_{x}(\mathbb{R})}\leq C \norm{\vec{u}(0,x)}_{H^{j}_{x}(\mathbb{R})} \text{, for all $j\in\{0,1,2\},$}\\
\norm{\overrightarrow{u}(t,x)}_{L^{\infty}_{x}(\mathbb{R})}\leq & \max_{\ell} \frac{C}{(t-s)^{\frac{1}{2}}}\norm{(1+\vert x-y_{\ell}-v_{\ell}s\vert )\chi_{\ell}(s,x)\overrightarrow{u}(s,x)}_{L^{2}_{x}(\mathbb{R})}.
 \end{align}
Moreover, if  $(\mathrm{H}4)$ holds, then, for any $p\in (1,2)$ close enough to $1,$ $\vec{u}(t,x)$ satisfies for a $K>1$ depending uniquely on $\min_{\ell}v_{\ell}-v_{\ell+1}>0$ and $p^{*}=\frac{p}{p-1}$ the following decay estimates for all $t>s\geq 0.$
  \begin{align}\label{Sphilinfty}
  \norm{\vec{u}(t,x)}_{L^{\infty}_{x}(\mathbb{R})}\leq &\frac{K}{(t-s)^{\frac{1}{2}}}\Bigg[\norm{\vec{u}(s,x)}_{L^{1}_{x}(\mathbb{R})}\\&{+}e^{{-}\beta\min_{\ell}(y_{\ell}-y_{\ell+1}+(v_{\ell}-v_{\ell+1})s)}\norm{\vec{u}(s,x)}_{L^{2}_{x}(\mathbb{R})}\Bigg],\\ \label{Sphiweightedlinfty}
\norm{\frac{\vec{u}(t,x)}{(1+\vert x-y_{\ell}-v_{\ell}t\vert)}}_{L^{\infty}_{x}(\mathbb{R})}\leq & \frac{K (s+y_{1}-y_{m})}{(t-s)^{\frac{3}{2}}} \norm{\overrightarrow{u}(s,x)}_{L^{1}_{x}(\mathbb{R})}\\ \nonumber 
&{+}\frac{K}{(t-s)^{\frac{3}{2}}}\max_{\ell}\norm{(1+\vert x-y_{\ell}-v_{\ell}s\vert )\chi_{\ell}(s,x)\overrightarrow{u}(s,x)}_{L^{1}_{x}(\mathbb{R})}\\ 
%\nonumber 
%&{+}\frac{K(y_{1}-y_{m}+(v_{1}-v_{m})s)^{2}e^{{-}\beta\left[\min_{\ell} (v_{\ell}-v_{\ell+1})s+(y_{\ell}-y_{\ell+1})\right]}\norm{\overrightarrow{u}(s,x)}_{H^{2}}}{(t-s)^{\frac{3}{2}}},
 \label{Sphiweightedderivative}
\max_{\ell}\norm{\frac{\partial_{x}\vec{u}(t)}{\langle x-v_{\ell}t-y_{\ell} \rangle^{1+\frac{p^{*}-2}{2p^{*}}+\alpha}}}_{L^{2}_{x}(\mathbb{R})}
\leq & \frac{K\max_{\ell}\norm{(1+\vert x-y_{\ell}-v_{\ell}s\vert)\chi_{\ell}(s,x)\langle \partial_{x}\rangle\vec{u}(s)}_{L^{1}_{x}(\mathbb{R})}^{\frac{2-p}{p}}\norm{\vec{u}(s)}_{H^{1}_{x}(\mathbb{R})}^{\frac{2(p-1)}{p}}}{(t-s)^{\frac{3}{2}(\frac{1}{p}-\frac{1}{p^{*}})}} \\
 &{+}K\frac{(s+y_{1}-y_{m})}{(t-s)^{\frac{3}{2}(\frac{1}{p}-\frac{1}{p^{*}})}}\norm{\vec{u}(s)}_{W^{1,1}_{x}(\mathbb{R})}^{\frac{2-p}{p}}\norm{\vec{u}(s)}_{H^{1}_{x}(\mathbb{R})}^{\frac{2(p-1)}{p}}
 \\
 &{+}K\frac{e^{{-}\beta\min_{\ell}((v_{\ell}-v_{\ell+1})s+y_{\ell}-y_{\ell+1})}}{(t-s)^{\frac{3}{2}(\frac{1}{p}-\frac{1}{p^{*}})}}\norm{\vec{u}(s)}_{H^{1}_{x}(\mathbb{R})}.
\end{align}
\end{theorem}
%\cgc{Please work on this, correct typos and fill in details}
\begin{proof}
    The proof of \eqref{CCC1Sphi} when $j=2$ is similar to the proof of Corollary $4.10$ from \cite{dispa}.   The unique difference is that we apply the estimates \eqref{Id-Bestimate} to prove the extension of Corollary $4.10$ from \cite{dispa} to all the cases satisfying  $\min_{\ell}v_{\ell}-v_{\ell+1}>0.$ Finally, similarly to the proof of Lemma $6.1$ in \cite{dispa}, we can prove estimate \eqref{CCC1Sphi} for $j=1$ using the estimates \eqref{CCC1Sphi} when $j=0$ and $j=2,$ and interpolation. 

\par Next, Proposition $7.1$ of \cite{KriegerSchlag} and estimate \eqref{dss0} imply that
\begin{multline*}
    \max_{\ell\in [m]}\norm{\hat{G}_{\omega_{\ell}}\left(e^{{-}it(k^{2}+\omega_{\ell})\sigma_{3}}e^{iy_{\ell}k}
    \begin{bmatrix}
        \phi_{1,\ell}\left(k+\frac{v_{\ell}}{2}\right)\\
        \phi_{2,\ell}\left(k-\frac{v_{\ell}}{2}\right)
    \end{bmatrix}
    \right)(x)}_{L^{\infty}_{x}(\mathbb{R})}\\
    \begin{aligned}
    \lesssim &\frac{1}{(t-s)^{\frac{1}{2}}}\left[\max_{\ell}\norm{\vec{\phi}_{\ell}}_{\mathcal{F}L^{1}}+\norm{\mathcal{S}(s)(\vec{\phi})}_{L^{1}_{x}(\mathbb{R})}\right]\\
    \lesssim &\frac{1}{(t-s)^{\frac{1}{2}}}\norm{\mathcal{S}(s)(\vec{\phi})}_{L^{1}_{x}(\mathbb{R})}\\
  %  \lesssim & \frac{1}{(t-s)^{\frac{1}{2}}}\norm{\mathcal{T}(s)(\vec{\phi})}_{L^{1}_{x}(\mathbb{R})} +\frac{1}{(t-s)^{\frac{1}{2}}}\norm{\mathcal{S}(s)(\vec{\phi})-\mathcal{T}(s)(\vec{\phi})}_{L^{1}_{x}(\mathbb{R})}\\
   % \lesssim & \frac{1}{(t-s)^{\frac{1}{2}}}\norm{\mathcal{T}(s)(\vec{\phi})}_{L^{1}_{x}(\mathbb{R})} +\frac{\max_{\ell}(v_{\ell}-v_{\ell+1})e^{{-} \min_{\ell}\beta(y_{\ell}-y_{\ell+1}+(v_{\ell}-v_{\ell+1})s)}}{(t-s)^{\frac{1}{2}}}\norm{\mathcal{T}(s)(\vec{\phi})}_{L^{2}_{x}(\mathbb{R})}
   \end{aligned}
\end{multline*}

Next, using Proposition $8.1$ from \cite{KriegerSchlag} we can verify that
\begin{multline}
\max_{\ell\in\{2,\,...,\,m-1\}}\norm{\frac{\hat{G}_{\omega_{\ell}}\left( e^{{-}it(k^{2}+\omega_{\ell})\sigma_{3}}e^{iy_{\ell}k}\begin{bmatrix}
    \phi_{1,\ell}\left(k+\frac{v_{\ell}}{2}\right)\\
    \phi_{2,\ell}\left(k-\frac{v_{\ell}}{2}\right)
\end{bmatrix}\right)(x)}{\langle x\rangle}}_{L^{\infty}_{x}(\mathbb{R})}\\
\begin{aligned}
\lesssim & \frac{1}{(t-s)^{\frac{3}{2}}}\norm{\langle x\rangle\hat{G}_{\omega_{\ell}}\left( e^{{-}is(k^{2}+\omega_{\ell})\sigma_{3}}e^{iy_{\ell}k}\begin{bmatrix}
    \phi_{1,\ell}\left(k+\frac{v_{\ell}}{2}\right)\\
    \phi_{2,\ell}\left(k-\frac{v_{\ell}}{2}\right)
\end{bmatrix}\right)(x)}_{L^{1}_{x}(\mathbb{R})}.
\end{aligned}
\end{multline}
Consequently, we can verify from estimates \eqref{dS(0)}, \eqref{dss0} and Theorem $1.8$ from \cite{dispa} that
\begin{multline}
\max_{\ell\in\{2,\,...,\,m-1\}}\norm{\frac{\hat{G}_{\omega_{\ell}}\left( e^{{-}it(k^{2}+\omega_{\ell})\sigma_{3}}e^{iy_{\ell}k}\begin{bmatrix}
    \phi_{1,\ell}\left(k+\frac{v_{\ell}}{2}\right)\\
    \phi_{2,\ell}\left(k-\frac{v_{\ell}}{2}\right)
\end{bmatrix}\right)(x)}{\langle x\rangle}}_{L^{\infty}_{x}(\mathbb{R})}\\
\begin{aligned}
\lesssim & \max_{\ell\in [m]}\frac{\norm{e^{{-}is(k^{2}+\omega_{\ell})\sigma_{3}}e^{iy_{\ell}k}\vec{\phi}_{\ell}\left(k+\sigma_{3}\frac{v_{\ell}}{2}\right)}_{\mathcal{F}L^{1}}}{(t-s)^{\frac{3}{2}}}\\&{+}\max_{\ell\in [m]}\frac{\norm{\frac{\partial}{\partial k}\left[e^{{-}is(k^{2}+\omega_{\ell})\sigma_{3}}e^{iy_{\ell}k}\vec{\phi}_{\ell}\left(k+\sigma_{3}\frac{v_{\ell}}{2}\right)\right]}_{\mathcal{F}L^{1}}}{(t-s)^{\frac{3}{2}}}\\
\lesssim & \frac{(s+y_{1}-y_{m})}{(t-s)^{\frac{3}{2}}} \norm{\mathcal{S}(\vec{\phi})(s,x)}_{L^{1}_{x}(\mathbb{R})}\\ \nonumber 
&{+}\frac{1}{(t-s)^{\frac{3}{2}}}\max_{\ell}\norm{(1+\vert x-y_{\ell}-v_{\ell}s\vert )\chi_{\ell}(s,x)\mathcal{S}(\vec{\phi})(s,x)}_{L^{1}_{x}(\mathbb{R})}. 
\end{aligned}
\end{multline}
Next, we can verify for any $\ell\in\{1,m\}$ that
\begin{multline*}
  \norm{\frac{\hat{G}_{\omega_{\ell}}\left( e^{{-}it(k^{2}+\omega_{\ell})\sigma_{3}}e^{iy_{\ell}k}\begin{bmatrix}
    \phi_{1,\ell}\left(k+\frac{v_{\ell}}{2}\right)\\
    \phi_{2,\ell}\left(k-\frac{v_{\ell}}{2}\right)
\end{bmatrix}\right)(x)}{\langle x\rangle}}_{L^{\infty}_{x}(\mathbb{R})}\\
\begin{aligned}
    \lesssim & \max_{j\in\{1,m\}}\frac{\norm{\langle x-y_{j}-v_{j}s \rangle\chi_{j}(s,x) \mathcal{S}(\vec{\phi})(s,x)}_{L^{1}_{x}(\mathbb{R})}}{(t-s)^{\frac{3}{2}}}\\
    &{+} \max_{\ell\in[m],j\in\{0,1\}}\frac{\norm{\frac{\partial^{j}}{\partial k^{j}}\left[e^{{-}is(k^{2}+\omega_{\ell})\sigma_{3}}e^{iy_{\ell}k}\vec{\phi}_{\ell}\left(k+\sigma_{3}\frac{v_{\ell}}{2}\right)\right]}_{\mathcal{F}L^{1}}}{(t-s)^{\frac{3}{2}}}.
\end{aligned}
\end{multline*}
Consequently, we can obtain from the estimates \eqref{dss0} and \eqref{dS(0)} and Theorem $1.8$ from \cite{dispa} that
\begin{multline*}
    \norm{\frac{\hat{G}_{\omega_{\ell}}\left( e^{{-}it(k^{2}+\omega_{\ell})\sigma_{3}}e^{iy_{\ell}k}\begin{bmatrix}
    \phi_{1,\ell}\left(k+\frac{v_{\ell}}{2}\right)\\
    \phi_{2,\ell}\left(k-\frac{v_{\ell}}{2}\right)
\end{bmatrix}\right)(x)}{\langle x\rangle}}_{L^{\infty}_{x}(\mathbb{R})}\\
\begin{aligned}
\lesssim & \frac{(s+y_{1}-y_{m})}{(t-s)^{\frac{3}{2}}} \norm{\mathcal{S}(\vec{\phi})(s,x)}_{L^{1}_{x}(\mathbb{R})}\\ \nonumber 
&{+}\frac{1}{(t-s)^{\frac{3}{2}}}\max_{\ell}\norm{(1+\vert x-y_{\ell}-v_{\ell}s\vert )\chi_{\ell}(s,x)\mathcal{S}(\vec{\phi})(s,x)}_{L^{1}_{x}(\mathbb{R})},
\end{aligned}
\end{multline*}
from which we conclude the proof of estimate \eqref{Sphiweightedlinfty} of Theorem 
\ref{decaySphi}. \par 
The proof of estimate \eqref{Sphiweightedderivative} is analogous to the proof of inequality $(\mathrm{C}.1.3)$ from Appendix $C$ of \cite{dispa}. The difference is that we use Lemma \ref{lem:claimR} and Propositions \ref{princ}, \ref{Coercivityproperty} in place of Lemmas $5.3,\,5.1$ of \cite{dispa} and Theorem $4.2$ of \cite{dispa}, respectively, in the argument of the proof of Theorem $1.21$ of \cite{dispa} to deal with the case where the separation in the speeds $\min_{\ell}v_{\ell}-v_{\ell+1}>0$ is  small. In particular, Lemma \ref{lem:claimR} and Propositions \ref{princ}, \ref{Coercivityproperty} are improved versions of Lemmas $5.3,\,5.1$ and Theorem $4.2$ of \cite{dispa}.
\end{proof}
%In this section, we will prove Theorem \ref{decaySphi}. 

\section{Proof of Proposition \ref{princ}}\label{sec:decomp}
The proof of Proposition \ref{princ} will be similar to the proof of Lemma $5.1$ from \cite{dispa}. The unique difference is that we will use Proposition \ref{lem:claimR} for operators $T$ of \eqref{form1} and Proposition \ref{Coercivityproperty} to deal with the cases where $\min_{\ell} v_{\ell}-v_{\ell+1}>0$ is small enough.
\par First, for each function $\vec{f}\in L^{2}_{x}(\mathbb
{R},\mathbb{C}^{2}),$ we consider the following system of equations on the variables $\vec{f}_{\ell,\pm}$ and $\vec{v}_{\ell}\in \Ra P_{d} $ of $ \mathcal{H}_{\ell}.$
\begin{align}\label{lG}
& e^{i\frac{\sigma_{3}v_{\ell}x}{2}}\hat{G}_{\omega_{\ell}}\left(e^{iy_{\ell}k}\begin{bmatrix}
     \phi_{1,\ell}\left(k+\frac{v_{\ell}}{2}\right)\\
     \phi_{2,\ell}\left(k-\frac{v_{\ell}}{2}\right)
 \end{bmatrix}\right)(x-y_{\ell})  \\
 &= 1_{\left\{\frac{(y_{\ell+1}+y_{\ell})}{2}<x\leq \frac{(y_{\ell-1}+y_{\ell})}{2}\right\}}(x)f(x)+ e^{i\frac{\sigma_{3}v_{\ell}x}{2}}\overrightarrow{v_{d_{\ell}}}(x-y_{\ell})\\ \nonumber
 & {+}1_{\left\{x\leq \frac{(y_{\ell+1}+y_{\ell})}{2}\right\}}(x)f_{\ell,-}(x)+1_{\left\{x> \frac{(y_{\ell}+y_{\ell-1})}{2}\right\}}(x)f_{\ell,+}(x),\\ \label{1G}
 & e^{i\frac{\sigma_{3}v_{1}x}{2}}\hat{G}_{\omega_{1}}\left(e^{iy_{1}k}\begin{bmatrix}
     \phi_{1,1}\left(k+\frac{v_{1}}{2}\right)\\
     \phi_{2,1}\left(k-\frac{v_{1}}{2}\right)
 \end{bmatrix}\right)(x-y_{1})\\    
 &=  1_{\left\{x>\frac{(y_{1}+y_{2})}{2}\right\}}(x)f(x)+1_{\left\{x\leq \frac{(y_{1}+y_{2})}{2}\right\}}(x)f_{1,-}(x)+e^{i\frac{\sigma_{3}v_{1}x}{2}}\overrightarrow{v_{d_{1}}}(x-y_{1}),\\ \label{mG}
& e^{i\frac{\sigma_{3}v_{m}x}{2}}\hat{G}_{\omega_{m}}\left(e^{iy_{m}k}\begin{bmatrix}
     \phi_{1,m}\left(k+\frac{v_{m}}{2}\right)\\
     \phi_{2,m}\left(k-\frac{v_{m}}{2}\right)
 \end{bmatrix}\right)(x-y_{m})\\
 &= 1_{\left\{x\leq  \frac{(y_{m}+y_{m-1})}{2}\right\}}(x)f(x)+1_{\left\{x>\frac{(y_{m}+y_{m-1})}{2}\right\}}(x)f_{m,+}(x){+}e^{i\frac{\sigma_{3}v_{m}x}{2}}\overrightarrow{v_{d_{m}}}(x-y_{m}),
 \end{align}
such that the functions $\phi_{\ell}(k)\coloneqq (\phi_{1,\ell}(k),\phi_{2,\ell}(k))\in L^{2}_{k}(\mathbb{R},\mathbb{C}^{2})$ are defined in the items $\mathrm{a)}$ and $\mathrm{b})$ of Definition \ref{s0def}. Similarly to the argument used in the proof of Lemma $5.1$ from \cite{dispa}, the existence of functions $\phi_{\ell}, f_{\ell,\pm} $ and $\vec{v}_{d_{\ell}}$ satisfying equations \eqref{1G} and \eqref{mG} for each given function $\vec{f}\in L^{2}_{x}(\mathbb{R},\mathbb{C}^{2})$ will imply that Proposition \ref{princ} is true.
\par Next, using Lemma \ref{leper}, we can verify applying $F^{*}_{\ell}$ and $G^{*}_{\ell}$ to each equation of \eqref{1G} and obtain the following identities
\begin{align}\label{l11}
 %\tag{Eq. $F_{\omega_{\ell}}^{*}.$}
  \begin{bmatrix}
     e^{iy_{\ell}k} \phi_{1,\ell}(k+\frac{v_{\ell}}{2})\\
      {-}e^{iy_{\ell}k}\phi_{2,\ell}(k-\frac{v_{\ell}}{2})
\end{bmatrix}=&F_{\omega_{\ell}}^{*}\left(\sigma_{3}e^{{-}i\sigma_{3}\frac{v_{\ell}(x+y_{\ell})}{2}}\chi_{\{\frac{y_{\ell+1}-y_{\ell}}{2}< x\leq \frac{y_{\ell-1}-y_{\ell}}{2}\}}(x)f(x+y_{\ell})\right)(k)\\ \nonumber
&{+}F_{\omega_{\ell}}^{*}\left(\sigma_{3}e^{{-}i\sigma_{3}\frac{v_{\ell}(x+y_{\ell})}{2}}\chi_{\{x\leq \frac{y_{\ell+1}-y_{\ell}}{2}\}}(x)f_{\ell,-}(x+y_{\ell})\right)(k)\\ \nonumber
&{+}F_{\omega_{\ell}}^{*}\left(\sigma_{3}e^{{-}i\sigma_{3}\frac{v_{\ell}(x+y_{\ell})}{2}}\chi_{\{x> \frac{y_{\ell-1}-y_{\ell}}{2}\}}(x)f_{\ell,+}(x+y_{\ell})\right)(k),
 \end{align}
and, using the item  $\mathrm{b)}$ of Definition \ref{s0def}, we can verify that
\begin{align}\label{l22}
%\tag{Eq. $G_{\omega_{\ell}}^{*}.$}
  \begin{bmatrix}
     e^{iy_{\ell}k} \phi_{1,\ell-1}(k+\frac{v_{\ell}}{2})\\
      {-}e^{iy_{\ell}k}\phi_{2,\ell-1}(k-\frac{v_{\ell}}{2})
\end{bmatrix}=&G_{\omega_{\ell}}^{*}\left(\sigma_{3}e^{{-}i\sigma_{3}\frac{v_{\ell}(x+y_{\ell})}{2}}\chi_{\{\frac{y_{\ell+1}-y_{\ell}}{2}< x\leq \frac{y_{\ell-1}-y_{\ell}}{2}\}}(x)f(x+y_{\ell})\right)(k)\\ \nonumber
&{+}G_{\omega_{\ell}}^{*}\left(\sigma_{3}e^{{-}i\sigma_{3}\frac{v_{\ell}(x+y_{\ell})}{2}}\chi_{\{x\leq \frac{y_{\ell+1}-y_{\ell}}{2}\}}(x)f_{\ell,-}(x+y_{\ell})\right)(k)\\ \nonumber
&{+}G_{\omega_{\ell}}^{*}\left(\sigma_{3}e^{{-}i\sigma_{3}\frac{v_{\ell}(x+y_{\ell})}{2}}\chi_{\{x> \frac{y_{\ell-1}-y_{\ell}}{2}\}}(x)f_{\ell,+}(x+y_{\ell})\right)(k).
 \end{align}
Furthermore, we can verify that
\begin{align}\label{l33}
%\tag{Eq. $G_{\omega_{m}}^{*}.$}
\begin{bmatrix}
e^{iy_{m}k}\phi_{1,m-1}\left(k+\frac{v_{m}}{2}\right)\\
{-}e^{iy_{m}k}\phi_{2,m-1}\left(k-\frac{v_{m}}{2}\right)
\end{bmatrix}=&G_{\omega_{m}}^{*}\left(\sigma_{3}e^{{-}i\sigma_{3}\frac{v_{m}(x+y_{m})}{2}}\chi_{\{ x\leq  \frac{y_{m-1}-y_{m}}{2}\}}(x)f(x+y_{m})\right)(k)\\ \nonumber
&{+}G_{\omega_{m}}^{*}\left(\sigma_{3}e^{{-}i\sigma_{3}\frac{v_{m}(x+y_{m})}{2}}\chi_{\{x> \frac{y_{m-1}-y_{m}}{2}\}}(x)f_{m,+}(x+y_{m})\right)(k),   
\end{align}
and
\begin{align}\label{l44}
%\tag{Eq. $F_{\omega_{1}}^{*}.$}
\begin{bmatrix}
e^{iy_{1}k}\phi_{1,1}\left(k+\frac{v_{1}}{2}\right)\\
{-}e^{iy_{1}k}\phi_{2,1}\left(k-\frac{v_{1}}{2}\right)
\end{bmatrix}=&F_{\omega_{1}}^{*}\left(\sigma_{3}e^{{-}i\sigma_{3}\frac{v_{1}(x+y_{1})}{2}}\chi_{\{ x>  \frac{y_{2}-y_{1}}{2}\}}(x)f(x+y_{1})\right)(k)\\ \nonumber
&{+}F_{\omega_{1}}^{*}\left(\sigma_{3}e^{{-}i\sigma_{3}\frac{v_{1}(x+y_{1})}{2}}\chi_{\{x\leq  \frac{y_{2}-y_{1}}{2}\}}(x)f_{1,-}(x+y_{1})\right)(k).   \end{align}
\par Next, using the asymptotic behavior of $\mathcal{F}$ and $\mathcal{G}$ from \eqref{asy1}, \eqref{asy2}, \eqref{asy3} and \eqref{asy4},
we can verify the following estimates below for each $\ell \in \{2,\,...,m-1\}.$
\begin{multline}\label{estFl}
\begin{aligned}
    \begin{bmatrix}
     e^{iy_{\ell}k} \phi_{1,\ell}(k+\frac{v_{\ell}}{2})\\
      {-}e^{iy_{\ell}k}\phi_{2,\ell}(k-\frac{v_{\ell}}{2})
  \end{bmatrix}=& F_{1,\ell}(f)(k){+}\overline{s_{\omega_{\ell}}(k)}\sigma_{3}\left[\int_{\frac{y_{\ell-1}-y_{\ell}}{2}}^{{+}\infty}e^{{-}ikx}e^{{-}i\sigma_{3}\frac{v_{\ell}(x+y_{\ell})}{2}}f_{\ell,+}(x+y_{\ell})\,dx\right]
\\ &{+}\sigma_{3}\int_{{-}\infty}^{\frac{y_{\ell+1}-y_{\ell}}{2}}\left(e^{{-}ikx}+r_{\omega_{\ell}}({-}k)e^{ikx}\right)e^{{-}i\sigma_{3}\frac{v_{\ell}(x+y_{\ell})}{2}}f_{\ell,-}(x+y_{\ell})\,dx\\&{+}O\left(\frac{1}{(1+\vert k\vert )}\norm{f_{\ell,-}(x)1_{\{x\leq\frac{y_{\ell}+y_{\ell+1}}{2}\}}}_{L^{2}_{x}(\mathbb{R})}e^{\beta\frac{(y_{\ell+1}-y_{\ell})}{2}}\right)\\&{+}O\left(\frac{1}{(1+\vert k\vert )}\norm{f_{\ell,+}(x)1_{\{x>\frac{y_{\ell}+y_{\ell-1}}{2}\}}}_{L^{2}_{x}(\mathbb{R})}e^{{-}\beta\frac{(y_{\ell-1}-y_{\ell})}{2}}\right),
\end{aligned}
\end{multline}
\begin{multline}\label{estGl}
\begin{aligned}
 \begin{bmatrix}
e^{iy_{\ell}k}\phi_{1,\ell-1}\left(k+\frac{v_{\ell}}{2}\right)\\
    {-} e^{iy_{\ell}k}\phi_{1,\ell-1}\left(k-\frac{v_{\ell}}{2}\right)
 \end{bmatrix}=& G_{1,\ell}(f)(k){+}s_{\omega_{\ell}}(k)\sigma_{3}\int_{{-}\infty}^{\frac{y_{\ell+1}-y_{\ell}}{2}}e^{{-}ikx}e^{{-}i\sigma_{3}\frac{v_{\ell}(x+y_{\ell})}{2}}f_{\ell,-}(x+y_{\ell})\,dx\\
 &{+}\sigma_{3}\int_{\frac{y_{\ell-1}-y_{\ell}}{2}}^{{+}\infty}e^{{-}i\sigma_{3}\frac{v_{\ell}(x+y_{\ell})}{2}}\left[e^{{-}ikx}+r_{\omega_{\ell}}(k)e^{ikx}\right]f_{\ell,+}(x+y_{\ell})\,dx\\
 &{+}O\left(\frac{1}{(1+\vert k\vert )}\norm{f_{\ell,-}(x)1_{\{x\leq \frac{y_{\ell}+y_{\ell+1}}{2}\}}}_{L^{2}_{x}(\mathbb{R})}e^{\beta\frac{(y_{\ell+1}-y_{\ell})}{2}}\right)\\
 &{+}O\left(\frac{1}{(1+\vert k\vert )}\norm{f_{\ell,+}(x)1_{\{x>\frac{y_{\ell}+y_{\ell-1}}{2}\}}}_{L^{2}_{x}(\mathbb{R})}e^{{-}\beta\frac{(y_{\ell-1}-y_{\ell})}{2}}\right),
\end{aligned}
\end{multline}
where
\begin{align*}
   F_{1,\ell}(f)(k)=\begin{bmatrix}
       f_{1,\ell}(k)\\
       f_{2,\ell}(k)
   \end{bmatrix}\coloneqq &F_{\omega_{\ell}}^{*}\left(\sigma_{3}e^{{-}i\sigma_{3}\frac{v_{\ell}(x+y_{\ell})}{2}}1_{\{\frac{y_{\ell+1}-y_{\ell}}{2}< x\leq \frac{y_{\ell-1}-y_{\ell}}{2}\}}(x)f(x+y_{\ell})\right)(k),\\
    G_{1,\ell}(f)(k)=\begin{bmatrix}
       g_{1,\ell}(k)\\
       g_{2,\ell}(k)
   \end{bmatrix} \coloneqq &F_{\omega_{\ell}}^{*}\left(\sigma_{3}e^{{-}i\sigma_{3}\frac{v_{\ell}(x+y_{\ell})}{2}}1_{\{\frac{y_{\ell+1}-y_{\ell}}{2}< x\leq \frac{y_{\ell-1}-y_{\ell}}{2}\}}(x)f(x+y_{\ell})\right)(k).
\end{align*}
Similarly, we can verify when $\ell=1$ and $\ell=2,$ the following estimates below.
\begin{multline}\label{estGm}
\begin{aligned}
\begin{bmatrix}
    e^{iy_{m}k}\phi_{1,m-1}\left(k+\frac{v_{m}}{2}\right)\\
{-}e^{iy_{m}k}\phi_{2,m-1}\left(k-\frac{v_{m}}{2}\right)
\end{bmatrix}=&G_{1,m}(f)(k)\\ 
&{+}r_{\omega_{m}}(k)\sigma_{3}\left[\int_{\frac{y_{m-1}-y_{m}}{2}}^{{+}\infty}e^{{-}i\sigma_{3}\frac{v_{m}(x+y_{m})}{2}}f_{m,{+}}(x+y_{m})e^{ikx}\,dx\right]\\&{+}\sigma_{3}\int_{\frac{y_{m-1}-y_{m}}{2}}^{{+}\infty}e^{{-}i\sigma_{3}\frac{v_{m}(x+y_{m})}{2}}f_{m,{+}}(x+y_{m})e^{{-}ikx}\,dx\\&{+}O\left(\frac{1}{(1+\vert k\vert)}\norm{f_{m,+}(x)1_{\{x> \frac{y_{m-1}+y_{m}}{2}\}}}_{L^{2}_{x}(\mathbb{R})}e^{{-}\beta \frac{y_{m-1}-y_{m}}{2} }\right),
\end{aligned}
\end{multline}
and 
\begin{align}\label{F1est}
\begin{bmatrix}  
    e^{iy_{1}k}\phi_{1,1}\left(k+\frac{v_{1}}{2}\right)\\
{-}e^{iy_{1}k}\phi_{2,1}\left(k-\frac{v_{1}}{2}\right)
\end{bmatrix}=&F_{1,1}(f)(k){+}r_{\omega_{1}}({-}k)\sigma_{3}\int_{{-}\infty}^{\frac{y_{2}-y_{1}}{2}}e^{ikx}e^{{-}i\sigma_{3}\frac{v_{1}(x+y_{1})}{2}}f_{1,-}(x+y_{1})\,dx
\\ \nonumber &{+} \sigma_{3}\int_{{-}\infty}^{\frac{y_{2}-y_{1}}{2}} e^{{-}ikx}e^{{-}i\sigma_{3}\frac{v_{1}(x+y_{1})}{2}}f_{1,-}(x+y_{1})\,dx\\
&{+}O\left(\frac{1}{(1+\vert k\vert)}\norm{f_{m,+}(x)1_{\{x> \frac{y_{m-1}+y_{m}}{2}\}}}_{L^{2}_{x}(\mathbb{R})}e^{{-}\beta \frac{y_{m-1}-y_{m}}{2} }\right),
\end{align}
such that
\begin{align*}
   F_{1,1}(f)(k)=&F_{\omega_{1}}^{*}\left(\sigma_{3}e^{{-}i\sigma_{3}\frac{v_{1}(x+y_{1})}{2}}1_{\{\frac{y_{2}-y_{1}}{2}< x \}}(x)f(x+y_{1})\right)(k),\\
    G_{1,m}(f)(k)=&F_{\omega_{m}}^{*}\left(\sigma_{3}e^{{-}i\sigma_{3}\frac{v_{m}(x+y_{m})}{2}}1_{\{ x\leq \frac{y_{m-1}-y_{m}}{2}\}}(x)f(x+y_{m})\right)(k).
\end{align*}
From now on, we consider the following functions 
\begin{align}\label{unknown}
\begin{bmatrix}
    g_{1,-,1}(k)\\
    g_{1,-,2}(k)
\end{bmatrix}=&\int_{{-}\infty}^{0}f_{1,-}\left(x+\frac{y_{1}+y_{2}}{2}\right)e^{ikx}\,dx,\\ \nonumber  \begin{bmatrix}
    g_{\ell,-,1}(k)\\
    g_{\ell,-,2}(k)
\end{bmatrix}=&\int_{{-}\infty}^{0}f_{\ell,-}\left(x+\frac{y_{\ell}+y_{\ell+1}}{2}\right)e^{ikx}\,dx \text{, when $2\leq \ell\leq m-1,$}\\ \nonumber
\begin{bmatrix}
    g_{\ell,+,1}(k)\\
    g_{\ell,+,2}(k)
\end{bmatrix}=&\int_{0}^{{+}\infty}f_{\ell,+}\left(x+\frac{y_{\ell}+y_{\ell-1}}{2}\right)e^{ikx}\,dx \text{, when $2\leq \ell\leq m-1,$}\\ \nonumber \begin{bmatrix}
    g_{m,+,1}(k)\\
    g_{m,+,2}(k)
\end{bmatrix}=&\int_{0}^{{+}\infty}f_{m,+}\left(x+\frac{y_{m}+y_{m-1}}{2}\right)e^{ikx}\,dx.
\end{align}
Moreover, we consider from now on functions of the form.
\begin{align*}
    A_{1,-}(f)(k)=& P_{+}\left(
e^{i\frac{(y_{1}-y_{2})k}{2}}e^{i\theta_{1,-,1}}g_{1,2}(k)-e^{\frac{i(y_{2}-y_{1})k}{2}}e^{i\theta_{1,-,2}}f_{1,1}\left(k+\frac{v_{2}}{2}-\frac{v_{1}}{2}\right)\right),\\
    A_{m,+}(f)(k)= &P_{+}\left(e^{i\theta_{1,+,m}}e^{i\frac{(y_{m}-y_{m-1})k}{2}}f_{1,m-1}\left(k+\frac{v_{m}}{2}-\frac{v_{m-1}}{2}\right)-e^{i\theta_{2,+,m}}e^{{-}i\frac{(y_{m}-y_{m-1})k}{2}}g_{1,m}\left(k-\frac{v_{m}}{2}\right)\right),\\
    A_{\ell,-}(f)(k)=& P_{-}\left(e^{i\frac{k(y_{\ell}-y_{\ell+1})}{2}}e^{i\theta_{1,-,\ell}}g_{1,\ell+1}(k)-e^{i\frac{k(y_{\ell+1}-y_{\ell})}{2}}e^{i\theta_{2,-,\ell}}f_{1,\ell}\left(k+\frac{v_{\ell+1}}{2}-\frac{v_{\ell}}{2}\right)\right),\\
    A_{\ell,+}(f)(k)= & P_{+}\left(e^{i\frac{k(y_{\ell}-y_{\ell+1})}{2}}e^{i\theta_{1,+,\ell}}g_{1,\ell+1}(k)-e^{i\frac{k(y_{\ell+1}-y_{\ell})}{2}}e^{i\theta_{2,+,\ell}}f_{1,\ell}\left(k+\frac{v_{\ell+1}}{2}-\frac{v_{\ell}}{2}\right)\right).
\end{align*}
Consequently, using the estimates \eqref{estFl}, \eqref{estGl}, \eqref{estGm}, \eqref{F1est} and Lemma \ref{+-interact}, we can verify that there exist functions $A_{\ell,\pm}(f)$ of the form  such that  $g_{\ell,\pm}$ satisfy the following linear system of equations, see Steps $3$ and $4$ of the proof of Lemma $5.1$ from \cite{dispa} for more details,
\begin{align}\label{giant linear system 1}
    A_{1,-}(f)\left({-}k-\frac{v_{2}}{2}\right)=& g_{1,{-},1}\left(k\right)-r_{2}\left({-}k-\frac{v_{2}}{2}\right)g_{2,+,1}\left({-}k-v_{2}\right)\\&{-}s_{2}\left({-}k-\frac{v_{2}}{2}\right)g_{2,{-},1}\left(k\right)+O_{L^{2}}\left(\max_{\ell,\pm}\norm{g_{\ell,\pm}}_{L^{2}}e^{{-}\beta \min_{h\in [m-1]}y_{h}-y_{h+1}}\right),\\  \nonumber
A_{2,+}(f)\left({-}k-\frac{v_{2}}{2}\right)=&g_{2,{+},1}\left(k\right)-r_{1}\left(k+\frac{v_{1}}{2}\right)g_{1,-,1}\left({-}k-v_{1}\right)\\&{+}O_{L^{2}}\left(\max_{\ell,\pm}\norm{g_{\ell,\pm}}_{L^{2}}e^{{-}\beta \min_{h\in [m-1]}y_{h}-y_{h+1}}\right),\\ \label{firstofall}
A_{2,-}(f)\left({-}k-\frac{v_{3}}{2}\right)=&g_{2,{-},1}\left(k\right)-r_{3}\left({-}k-\frac{v_{3}}{2}\right)g_{3,+,1}\left({-}k-v_{3}\right)\\&{-}s_{3}\left({-}k-\frac{v_{3}}{2}\right)g_{3,{-},1}\left( k\right)+O_{L^{2}}\left(\max_{\ell,\pm}\norm{g_{\ell,\pm}}_{L^{2}}e^{{-}\beta \min_{h\in [m-1]}y_{h}-y_{h+1}}\right)
,\\ \nonumber
A_{3,+}\left(f\right)\left({-}k-\frac{v_{3}}{2}\right)=& g_{3,+,1}\left(k\right){-}r_{2}\left(k+\frac{v_{2}}{2}\right)g_{2,-,1}\left({-}k-v_{2}\right)\\&{-}s_{2}\left(k+\frac{v_{2}}{2}\right)g_{2,+,1}\left(k\right)+O_{L^{2}}\left(\max_{\ell,\pm}\norm{g_{\ell,\pm}}_{L^{2}}e^{{-}\beta \min_{h\in [m-1]}y_{h}-y_{h+1}}\right)
,\, ...,\\  \nonumber
A_{\ell,+}\left(f\right)\left({-}k-\frac{v_{\ell}}{2}\right)=&g_{\ell,+,1}\left(k\right){-}r_{\ell-1}\left(k+\frac{v_{\ell-1}}{2}\right)g_{\ell-1,-,1}\left({-}k-v_{\ell-1}\right)\\&{-}s_{\ell-1}\left(k+\frac{v_{\ell-1}}{2}\right)g_{\ell-1,+,1}\left(k\right)+O_{L^{2}}\left(\max_{\ell,\pm}\norm{g_{\ell,\pm}}_{L^{2}}e^{{-}\beta \min_{h\in [m-1]}y_{h}-y_{h+1}}\right),\\ \nonumber
A_{\ell,-}(f)\left({-}k-\frac{v_{\ell+1}}{2}\right)=& g_{\ell,-,1}\left(k\right)-r_{\ell+1}\left({-}k-\frac{v_{\ell+1}}{2}\right)g_{\ell+1,+,1}\left({-}k-v_{\ell+1}\right)\\&{-}s_{\ell+1}\left({-}k-v_{\ell+1}\right)g_{\ell+1,-,1}\left(k\right)+O_{L^{2}}\left(\max_{\ell,\pm}\norm{g_{\ell,\pm}}_{L^{2}}e^{{-}\beta \min_{h\in [m-1]}y_{h}-y_{h+1}}\right), \, ..., \\  \nonumber
A_{m-1,+}(f)\left({-}k-\frac{v_{m-1}}{2}\right)=&g_{m-1,+,1}\left(k\right){-}r_{m-2}\left(k+\frac{v_{m-2}}{2}\right)g_{m-2,-,1}\left({-}k-v_{m-2}\right)\\&{-}s_{m-2}\left(k+\frac{v_{m-2}}{2}\right)g_{m-2,+,1}\left(k\right)+O_{L^{2}}\left(\max_{\ell,\pm}\norm{g_{\ell,\pm}}_{L^{2}}e^{{-}\beta \min_{h\in [m-1]}y_{h}-y_{h+1}}\right),\\ \nonumber
A_{m-1,-}(f)\left({-}k-\frac{v_{m}}{2}\right)=&g_{m-1,-,1}\left(k\right)-r_{m}\left({-}k-\frac{v_{m}}{2}\right)g_{m,+,1}\left({-}k-v_{m}\right)\\&{+}O_{L^{2}}\left(\max_{\ell,\pm}\norm{g_{\ell,\pm}}_{L^{2}}e^{{-}\beta \min_{h\in [m-1]}y_{h}-y_{h+1}}\right),\\ \nonumber
A_{m,+}(f)\left({-}k-\frac{v_{m}}{2}\right)=& g_{m,+,1}\left(k\right){-}r_{m-1}\left(k+\frac{v_{m-1}}{2}\right)g_{m-1,-,1}\left({-}k-v_{m-1}\right)\\&{-}s_{m-1}\left(k+\frac{v_{m-1}}{2}\right)g_{m-1,+,1}\left(k\right)+O_{L^{2}}\left(\max_{\ell,\pm}\norm{g_{\ell,\pm}}_{L^{2}}e^{{-}\beta \min_{h\in [m-1]}y_{h}-y_{h+1}}\right).
\end{align}
Moreover, it is not difficult to verify that the right hand-side of the linear system \eqref{giant linear system 1} is a small perturbation of an application of a operator $T$ of \eqref{form1} on the vector
\begin{equation*}
    \vec{g}(k)=\begin{bmatrix}
        g_{1,-,1}(k)\\
        g_{2,+,1}(k)\\
        g_{2,-,1}(k)\\
        g_{3,+,1}(k)\\
        ...\\
        g_{\ell,+,1}(k)\\
        g_{\ell,-,1}(k)\\
        ...\\
        g_{m,+,1}(k)
    \end{bmatrix}\in L^{2}(\mathbb{R},\mathbb{C}^{2m-2}).
\end{equation*}
More precisely,
\begin{equation*}
    \begin{bmatrix}
        A_{1,-}(f)\left({-}k-\frac{v_{2}}{2}\right)\\
        A_{2,+}(f)\left({-}k-\frac{v_{2}}{2}\right)\\
        A_{2,-}(f)\left({-}k-\frac{v_{3}}{2}\right)\\
        A_{3,+}(f)\left({-}k-\frac{v_{3}}{2}\right)\\
        ...\\
        A_{\ell,+}(f)\left({-}k-\frac{v_{\ell}}{2}\right)\\
        A_{\ell,-}(f)\left({-}k-\frac{v_{\ell+1}}{2}\right)\\
        ...\\
        A_{m,+}(f)\left({-}k-\frac{v_{m}}{2}\right) 
    \end{bmatrix}=T(\vec{g})(k)+O_{L^{2}}\left(\max_{\ell,\pm}\norm{g_{\ell,\pm}}_{L^{2}}e^{{-}\beta \min_{h\in [m-1]}y_{h}-y_{h+1}}\right).
\end{equation*}
As a consequence, since there exists a uniform constant $K>1$ satisfying
\begin{equation*}
   \max_{\ell}\norm{A_{\ell,\pm}(f)}_{L^{2}_{k}(\mathbb{R})}\leq K \norm{f}_{L^{2}_{x}(\mathbb{R})}, 
\end{equation*}
we can deduce from Proposition \ref{lem:claimR} when $\min_{\ell}y_{\ell}-y_{\ell+1}>1$ is large enough that there exists a uniform constant $C>1$ satisfying
\begin{equation*}
    \max_{\ell}\norm{g_{\ell,\pm,1}(k)}_{L^{2}_{k}(\mathbb{R})}\leq C\left[\norm{f}_{L^{2}_{x}(\mathbb{R})}+\max_{\ell,\pm}\norm{g_{\ell,\pm}}_{L^{2}}e^{{-}\beta \min_{h\in [m-1]}y_{h}-y_{h+1}}\right].
\end{equation*}
Similarly, we can verify that
\begin{equation*}
    \max_{\ell}\norm{g_{\ell,\pm,2}(k)}_{L^{2}_{k}(\mathbb{R})}\leq C\left[\norm{f}_{L^{2}_{x}(\mathbb{R})}+\max_{\ell,\pm}\norm{g_{\ell,\pm}}_{L^{2}}e^{{-}\beta \min_{h\in [m-1]}y_{h}-y_{h+1}}\right].
\end{equation*}
Consequently, we can deduce that
\begin{equation*}
    \max_{\ell\in [m],j\in\{1,2\}}\norm{g_{\ell,\pm,j}(k)}_{L^{2}_{k}(\mathbb{R})}\leq 2C\norm{f}_{L^{2}_{x}(\mathbb{R})}.
\end{equation*}
In particular, there exists a unique solution $\vec{g}$ of the linear system \eqref{giant linear system 1}. Therefore, the existence of solution for the system \eqref{giant linear system 1} implies that there exist functions $f_{\ell,\pm}$ satisfying all the equations \eqref{l11}-\eqref{F1est} for a unique $\vec{\phi}$ in the domain of $\mathcal{S},$ see Definition \ref{s0def}.
\par As a consequence, the existence of functions $f_{\ell,\pm}\in L^{2}_{x}(\mathbb{R},\mathbb{C}^{2})$ satisfying \eqref{l11}-\eqref{F1est} implies that the function $f$ satisfies the following identity
\begin{equation}\label{fff}
f(x)=\sum_{\ell=1}^{m}e^{i\frac{\sigma_{3}v_{\ell}x}{2}}\hat{G}_{\omega_{\ell}}\left(e^{iy_{\ell}k}\begin{bmatrix}
\phi_{1,\ell}\left(k+\frac{v_{\ell}}{2}\right)\\
     \phi_{2,\ell}\left(k-\frac{v_{\ell}}{2}\right)
 \end{bmatrix}\right)(x-y_{\ell})\chi_{\ell}(0,x)-\sum_{\ell=1}^{m}e^{i\frac{\sigma_{3}v_{\ell}x}{2}}\vec{v_{d_{\ell}}}(x-y_{\ell})\chi_{\ell}(0,x).   
\end{equation}
In conclusion, since each function $\vec{v}_{d_{\ell}}(x)\in \Ra P_{d}\mathcal{H}_{\ell}$ is generated by a finite set of Schwartz functions having exponential decay, and
\begin{equation*}
    \norm{\mathcal{S}(\phi)(0,x)-\sum_{\ell=1}^{m}e^{i\frac{\sigma_{3}v_{\ell}x}{2}}\hat{G}_{\omega_{\ell}}\left(e^{iy_{\ell}k}\begin{bmatrix}
\phi_{1,\ell}\left(k+\frac{v_{\ell}}{2}\right)\\
     \phi_{2,\ell}\left(k-\frac{v_{\ell}}{2}\right)
 \end{bmatrix}\right)(x-y_{\ell})\chi_{\ell}(0,x)}_{L^{2}_{x}(\mathbb{R})}\lesssim e^{{-}\beta \min_{\ell}y_{\ell}-y_{\ell+1}}\norm{\mathcal{S}(\phi)(0,x)},
\end{equation*}
we can deduce from the identity \eqref{fff} that any $f\in \Ra \mathcal{S}(0)\bigoplus_{\ell=1}^{m}\Ra P_{d}\mathcal{H}_{\ell}$ when $\min_{\ell}y_{\ell}-y_{\ell+1}>1$ is large enough, which is equivalent to the statement of Proposition \ref{princ}. 

\section{Asymptotics of solutions}\label{sec:stabscat}
In this section, we study the behavior of solutions to \ref{p} in the stable space in the sense of Definition \ref{def:stable}  and then show the dispersive decay of solutions in the scattering space in the sense of \ref{def:scatter}.

\subsection{Stable space}\label{prooftunst}
In this subsection, we prove  Theorem \ref{stablecase} to give precise descriptions on solutions.

\begin{proof}[Proof of Theorem \ref{stablecase}]

Firstly,
%for any $\ell\in [m]$ and $\lambda_{\ell,n}\in \sigma_{d}\left(\mathcal{H}_{\ell}\right),$ let
%\begin{equation*}
%   \mathcal{B}_{\ell,\lambda_{\ell,n}}=\left \{\mathfrak{v}_{1,\omega_{\ell},\lambda_{\ell,n}},\mathfrak{v}_{2,\omega_{\ell},\lambda_{\ell,n}},\,...,\,\mathfrak{v}_{\dim \ker \left(\mathcal{H}_{\ell}-\lambda_{\ell,n}\right)^{2},\omega_{\ell},\lambda_{\ell,n}}\right\}
%\end{equation*}
%be an orthonormal basis of $\ker \left(\mathcal{H}_{\ell}-\lambda_{\ell,n}\right)^{2}\subset L^{2}_{x}(\mathbb{R},\mathbb{C}^{2}).$ 
Proposition \ref{princ} implies that $\vec{\psi}(t,x)$ has a unique representation of the form
\begin{align*}
      \vec{\psi}(t,x)&=\mathcal{S}
\left(\vec{\phi}(t)\right)(t,x)+\sum_{\ell=1}^{m}\sum_{j=1}^{2N_\ell+2M_\ell}a_{j,\ell}(t)e^{i\sigma_{3}\left(\frac{v_{\ell}x}{2}-\frac{v_{\ell}^{2}t}{4}+\omega_{\ell}t+\gamma_{\ell}\right)}\vec{Z}_{j,\ell}(x-v_{\ell}t-y_{\ell})\\
&
%+\sum_{\ell=1}^{m}\sum_{j=1}^{K_{\ell,1}}e^{i\sigma_{3}\left(\frac{v_{\ell}x}{2}-\frac{v_{\ell}^{2}t}{4}+\omega_{\ell}t+\gamma_{\ell}\right)}\vec{Z}^0_{j,\ell}(x-v_{\ell}t-y_{\ell})
+\sum_{\ell=1}^{m}\sum_{j=1}^{K_{\ell,2}}a_{j,\ell}^{1}(t)e^{i\sigma_{3}\left(\frac{v_{\ell}x}{2}-\frac{v_{\ell}^{2}t}{4}+\omega_{\ell}t+\gamma_{\ell}\right)}\vec{Z}^1_{j,\ell}(x-v_{\ell}t-y_{\ell})
\end{align*}
%\begin{equation*}
 %   \mathcal{S}(\vec{\phi}(t,k))(t,x)+\sum_{\ell=1}^{m}\sum_{\lambda_{\ell,n}\in\sigma_{d}(\mathcal{H}_{\ell})}\sum_{j=1}^{\dim \ker \left(\mathcal{H}_{\ell}-\lambda_{\ell,n}\right)^{2}}e^{i\theta_{\ell}(t,x)\sigma_{3}}a_{j,\ell,n}(t)\mathfrak{v}_{j,\omega_{\ell},\lambda_{\ell,n}}(x-v_{\ell}t-y_{\ell}),
%\%end{equation*}
such that there is a uniform constant $C>1$ satisfying
\begin{equation}\label{upperboundprinc}
   \norm{S(\vec{\phi}(t))(t,x)}_{L^{2}_{x}(\mathbb{R})}+\max_{j,\ell}\vert a_{j,\ell}(t) \vert+\max_{j,\ell}\vert a^1_{j,\ell}(t) \vert \leq C\norm{\vec{\psi}(t,x)}_{L^{2}_{x}(\mathbb{R})} \text{, for all $t\geq 0.$}
\end{equation}
Furthermore, estimate \eqref{CCC1Sphi} of Theorem \ref{decaySphi} implies that
\begin{align*}
 &\max_{j,\ell}\left\vert \langle \mathcal{S}(\vec{\phi}(t))(t,x),  e^{i\theta_{j}(t,x)\sigma_{3}}\vec{Z}_{j,\ell}(x-v_{\ell}t-y_{\ell}) \rangle \right\vert+
    \max_{j,\ell}\left\vert \langle \mathcal{S}(\vec{\phi}(t))(t,x),  e^{i\theta_{j}(t,x)\sigma_{3}}\vec{Z}^1_{j,\ell}(x-v_{\ell}t-y_{\ell}) \rangle \right\vert\\
    &\lesssim  e^{{-}\beta (\min_{\ell}(v_{\ell}-v_{\ell+1})t+y_{\ell}-y_{\ell+1})}\norm{\mathcal{S}(\vec{\phi}(t))(t,x)}_{L^{2}_{x}(\mathbb{R})}
    \lesssim  e^{{-}\beta (\min_{\ell}(v_{\ell}-v_{\ell+1})t+y_{\ell}-y_{\ell+1})}\norm{\vec{\psi}(t,x)}_{L^{2}_{x}(\mathbb{R})}.
\end{align*}
As a consequence, we can deduce from the condition \eqref{asyorthsub} that there exists a constant $C>1$ following estimates
\begin{align}\nonumber
\max_{\I\lambda_{j,\ell}>0}\vert a_{j,\ell}(t)\vert\leq & C\left[1+e^{{-}\beta (\min_{\ell}(v_{\ell}-v_{\ell+1})t+y_{\ell}-y_{\ell+1})}\left(\norm{\mathcal{S}(\vec{\phi}(t))(t)}_{L^{2}_{x}(\mathbb{R})}+\max_{\I \lambda_{j,\ell}<0}\vert a_{j,\ell}(t) \vert+\max_{j,\ell}\vert a^1_{j,\ell}(t) \vert\right)\right]\\ \label{globalal}
\leq & C+Ce^{{-}\beta (\min_{\ell}(v_{\ell}-v_{\ell+1})t+y_{\ell}-y_{\ell+1})}\norm{\vec{\psi}(t,x)}_{L^{2}_{x}(\mathbb{R})}.
\end{align}
for all $t\geq 0.$

Here we study the generalized kernel \eqref{eq:generalizedkernel} in more details. Using the notations from \eqref{eq:generalizedkernel} and \eqref{eq:generalizedkernel1}, for $0\leq j\leq K_{\ell,1}$, we have
\begin{equation}
    \vec{Z}_{j,\ell}^0= \vec{Z}_{j,\ell}^1
\end{equation}and for $K_{\ell,1}<j\leq K_{\ell,2}$,
\begin{equation}\label{eq:jtoj'}
    \mathcal{H}_\ell \vec{Z}_{j,\ell}^1= \vec{Z}_{j',\ell}^0
\end{equation}for some $j'\in \{1,\ldots,K_{\ell,1}\}.$
For the convenience of notations, we set for $0\leq j\leq K_{\ell,1}$
\begin{equation}\label{eq:a0}
    a_{j,\ell}^0(t):= a_{j,\ell}^1(t)
\end{equation}
and for $K_{\ell,1}<j\leq K_{\ell,2}$, $a^2_{j',\ell}(t)$ is defined as
\begin{equation}\label{eq:a2}
    a^2_{j',\ell}(t) \vec{Z}_{j',\ell}^0:= \mathcal{H}_\ell(a^1_{j,\ell}\vec{Z}_{j,\ell}^1).
\end{equation}For convenience, for $1\leq j'\leq K_{\ell,1}$, if there is no $K_{1,\ell}<j\leq K_{2,\ell}$ such that \eqref{eq:jtoj'} hold, then $ a^2_{j',\ell}(t) \equiv 0.$

Since $\vec{\psi}(t,x)$ is a  solution  to the  equation \eqref{p}, we can verify the following identity.
\begin{multline}\label{forcingmainepde}
    i\mathcal{S}\left(\partial_{t}\vec{\phi}(t)\right)(t,x)+\sum_{\ell=1}^{m}\sum_{j=1}^{2N_\ell+2M_\ell}\Big(i\dot{a}_{j,\ell}(t)-\lambda_{j,\ell}a_{j,\ell}(t)\Big)e^{i\sigma_{3}\left(\frac{v_{\ell}x}{2}-\frac{v_{\ell}^{2}t}{4}+\omega_{\ell}t+\gamma_{\ell}\right)}\vec{Z}_{j,\ell}(x-v_{\ell}t-y_{\ell})\\
+\sum_{\ell=1}^{m}\sum_{j=1}^{K_{\ell,1}}\Big(i\dot{a}^0_{j,\ell}(t)-\dot{a}^2_{j,\ell}(t)\Big)e^{i\sigma_{3}\left(\frac{v_{\ell}x}{2}-\frac{v_{\ell}^{2}t}{4}+\omega_{\ell}t+\gamma_{\ell}\right)}\vec{Z}^1_{j,\ell}(x-v_{\ell}t-y_{\ell})\\
+\sum_{\ell=1}^{m}\sum_{j>K_{\ell,1}}^{K_{\ell,2}}\Big(i\dot{a}^1_{j,\ell}(t)\Big)e^{i\sigma_{3}\left(\frac{v_{\ell}x}{2}-\frac{v_{\ell}^{2}t}{4}+\omega_{\ell}t+\gamma_{\ell}\right)}\vec{Z}^1_{j,\ell}(x-v_{\ell}t-y_{\ell})\\
    %{+}\sum_{\ell=1}^{m}\sum_{\lambda_{\ell,n}\in\sigma_{d}(\mathcal{H}_{\ell})}\sum_{j=1}^{\dim \ker \left(\mathcal{H}_{\ell}-\lambda_{\ell,n}\right)^{2}}e^{i(\theta_{\ell}(t,x)+\omega t)\sigma_{3}}\left[i\dot a_{j,\ell,n}(t)-\lambda_{\ell,n}a_{j,\ell,n}(t)\right]\mathfrak{v}_{j,\omega_{\ell},\lambda_{\ell,n}}(x-v_{\ell}t-y_{\ell})\\
    \begin{aligned}
    =&{-}\sum_{\ell_{1}\neq \ell_{2}}\sum_{j=1}^{2N_{\ell_1}+2M_{\ell_1}}a_{j,\ell_1}(t)e^{i\sigma_{3}\left(\frac{v_{\ell_1}x}{2}-\frac{v_{\ell_1}^{2}t}{4}+\omega_{\ell_1}t+\gamma_{\ell_1}\right)}\vec{Z}_{j,\ell_1}(x-v_{\ell_1}t-y_{\ell_1})V_{\ell_2}(t,x)\\
    &{-}\sum_{\ell_{1}\neq \ell_{2}}\sum_{j=1}^{K_{\ell_1,2}}a^1_{j,\ell_1}(t)e^{i\sigma_{3}\left(\frac{v_{\ell_1}x}{2}-\frac{v_{\ell_1}^{2}t}{4}+\omega_{\ell_1}t+\gamma_{\ell_1}\right)}\vec{Z}_{j,\ell_1}(x-v_{\ell_1}t-y_{\ell_1})V_{\ell_2}(t,x)\\
%&
%+\sum_{\ell=1}^{m}\sum_{j=1}^{K_{\ell,1}}e^{i\sigma_{3}\left(\frac{v_{\ell}x}{2}-\frac{v_{\ell}^{2}t}{4}+\omega_{\ell}t+\gamma_{\ell}\right)}\vec{Z}^0_{j,\ell}(x-v_{\ell}t-y_{\ell})
%+\sum_{\ell=1}^{m}\sum_{j=1}^{K_{\ell,2}}a_{j,\ell}^{1}(t)e^{i\sigma_{3}\left(\frac{v_{\ell}x}{2}-\frac{v_{\ell}^{2}t}{4}+\omega_{\ell}t+\gamma_{\ell}\right)}\vec{Z}^1_{j,\ell}(x-v_{\ell}t-y_{\ell})\\
 %   =&{-}\sum_{\ell_{1}\neq \ell_{2}}^{m}\sum_{\lambda_{\ell,n}\in\sigma_{d}(\mathcal{H}_{\ell})}\sum_{j=1}^{\dim \ker \left(\mathcal{H}_{\ell}-\lambda_{\ell_{1},k}\right)^{2}}e^{i(\theta_{\ell}(t,x)+\omega t)\sigma_{3}}a_{j,\ell_{1},k}(t)V_{\ell_{2}}(t,x)\mathfrak{v}_{j,\omega_{\ell_{1}},\lambda_{\ell_{1},k}}(x-v_{\ell_{1}}t-y_{\ell_{1}})\\
    &{-}\sum_{\ell=1}^{m}V_{\ell}(t,x)\Bigg[\mathcal{S}\left(\vec{\phi}(t)\right)(t,x)\\&{-}e^{i\left(\frac{v_{\ell}x}{2}-\frac{v_{\ell}^{2}t}{4}+\omega_{\ell}t+\gamma_{\ell}\right)\sigma_{3}}\hat{G}_{\omega_{\ell}}\left(
   e^{{-}it(k^{2}+\omega_{\ell})\sigma_{3}}e^{{-}i\gamma_{\ell}\sigma_{3}} \begin{bmatrix}
e^{iy_{\ell}k}\phi_{1,\ell}\left(t,k+\frac{v_{\ell}}{2}\right)\\
       e^{iy_{\ell}k}\phi_{2,\ell}\left(t,k-\frac{v_{\ell}}{2}\right)
    \end{bmatrix}\right)(x-y_{\ell}-v_{\ell}t)\Bigg]
    \\
    =&Forc(t,x).
    \end{aligned}
\end{multline}

%\begin{multline}\label{forcingmainepde1}
 %   i\mathcal{S}\left(\partial_{t}\vec{\phi}(t)\right)(t,x)\\{+}\sum_{\ell=1}^{m}\sum_{\lambda_{\ell,n}\in\sigma_{d}(\mathcal{H}_{\ell})}\sum_{j=1}^{\dim \ker \left(\mathcal{H}_{\ell}-\lambda_{\ell,n}\right)^{2}}e^{i(\theta_{\ell}(t,x)+\omega t)\sigma_{3}}\left[i\dot a_{j,\ell,n}(t)-\lambda_{\ell,n}a_{j,\ell,n}(t)\right]\mathfrak{v}_{j,\omega_{\ell},\lambda_{\ell,n}}(x-v_{\ell}t-y_{\ell})\\
  %  \begin{aligned}
  %  =&{-}\sum_{\ell_{1}\neq \ell_{2}}^{m}\sum_{\lambda_{\ell,n}\in\sigma_{d}(\mathcal{H}_{\ell})}\sum_{j=1}^{\dim \ker \left(\mathcal{H}_{\ell}-\lambda_{\ell_{1},k}\right)^{2}}e^{i(\theta_{\ell}(t,x)+\omega t)\sigma_{3}}a_{j,\ell_{1},k}(t)V_{\ell_{2}}(t,x)\mathfrak{v}_{j,\omega_{\ell_{1}},\lambda_{\ell_{1},k}}(x-v_{\ell_{1}}t-y_{\ell_{1}})\\
  %  &{-}\sum_{\ell=1}^{m}V_{\ell}(t,x)\Bigg[\mathcal{S}\left(\vec{\phi}(t)\right)(t,x)\\&{-}e^{i\left(\frac{v_{\ell}x}{2}-\frac{v_{\ell}^{2}t}{4}+\omega_{\ell}t+\gamma_{\ell}\right)\sigma_{3}}\hat{G}_{\omega_{\ell}}\left(
   %e^{{-}it(k^{2}+\omega_{\ell})\sigma_{3}}e^{{-}i\gamma_{\ell}\sigma_{3}} \begin{bmatrix}
%e^{iy_{\ell}k}\phi_{1,\ell}\left(k+\frac{v_{\ell}}{2}\right)\\
 %      e^{iy_{\ell}k}\phi_{2,\ell}\left(k-\frac{v_{\ell}}{2}\right)
  %  \end{bmatrix}\right)(x-y_{\ell}-v_{\ell}t)\Bigg]
   % \\
    %=&Forc(t,x).
    %\end{aligned}
%\end{multline}
In particular, Proposition \ref{princ} implies that the function $Forc(t,x)$ defined in the right-hand side of the linear partial differential equation \eqref{forcingmainepde} has a unique decomposition of the form
\begin{multline}\label{forcdecomposition}
    Forc(t,x)\\=\mathcal{S}(t)(\varphi(t))+\sum_{\ell=1}^{m}\sum_{j=1}^{2N_\ell+2M_\ell}b_{j,\ell}(t)e^{i\sigma_{3}\left(\frac{v_{\ell}x}{2}-\frac{v_{\ell}^{2}t}{4}+\omega_{\ell}t+\gamma_{\ell}\right)}\vec{Z}_{j,\ell}(x-v_{\ell}t-y_{\ell})\\
%+\sum_{\ell=1}^{m}\sum_{j=1}^{K_{\ell,1}}e^{i\sigma_{3}\left(\frac{v_{\ell}x}{2}-\frac{v_{\ell}^{2}t}{4}+\omega_{\ell}t+\gamma_{\ell}\right)}\vec{Z}^0_{j,\ell}(x-v_{\ell}t-y_{\ell})
+\sum_{\ell=1}^{m}\sum_{j=1}^{K_{\ell,2}}b_{j,\ell}^{1}(t)e^{i\sigma_{3}\left(\frac{v_{\ell}x}{2}-\frac{v_{\ell}^{2}t}{4}+\omega_{\ell}t+\gamma_{\ell}\right)}\vec{Z}^1_{j,\ell}(x-v_{\ell}t-y_{\ell}),
\end{multline}
for any $t\geq 0.$

From now on, we are going to use the following notation for the functions denote right-hand side of equation \eqref{forcdecomposition}
\begin{align}\label{ScontForc}
    P_{\mathrm{cont}}(t)\left[Forc(t,x)\right]=&\mathcal{S}(t)(\varphi(t)),\\ \label{bjformula}
    P_{\mathrm{disc},j,\ell}(t)\left[Forc(t,x)\right]=&b_{j,\ell}(t)e^{i\sigma_{3}\left(\frac{v_{\ell}x}{2}-\frac{v_{\ell}^{2}t}{4}+\omega_{\ell}t+\gamma_{\ell}\right)}\vec{Z}_{j,\ell}(x-v_{\ell}t-y_{\ell}),\\
    P^1_{\mathrm{disc},j,\ell}(t)\left[Forc(t,x)\right]=&b^1_{j,\ell}(t)e^{i\sigma_{3}\left(\frac{v_{\ell}x}{2}-\frac{v_{\ell}^{2}t}{4}+\omega_{\ell}t+\gamma_{\ell}\right)}\vec{Z}^1_{j,\ell}(x-v_{\ell}t-y_{\ell})
    .
\end{align}
It is not difficult to verify using Lemma $7.2$ from \cite{dispa} and Definition \ref{s0def} that there exist constants $C>1,\,\beta>0$ satisfying
\begin{equation}\label{l2forcupp}
 \norm{Forc(t,x)}_{H^{1}_{x}(\mathbb{R})}\leq C\left[\norm{\mathcal{S}(\vec{\phi}(t))(t)}_{L^{2}_{x}(\mathbb{R})} +\max_{\ell,j}\vert a_{j,\ell}(t)\vert +\max_{\ell,j}\vert a^1_{j,\ell}(t)\vert \right]e^{{-}\beta \min_{h}(v_{h}-v_{h+1})t+(y_{h}-y_{h+1})}.  \end{equation}
 \newpage
In particular,
\begin{multline}\label{upperboundPDD}
 \max_{j,\ell} \vert b^1_{j,\ell}(t)\vert + \max_{j,\ell} \vert b_{j,\ell}(t)\vert +\norm{P_{\mathrm{cont}}(t)Forc(t,x)}_{L^{2}_{x}(\mathbb{R})}\\
  \leq C\left[\norm{\mathcal{S}(\vec{\phi})(t)}_{L^{2}_{x}(\mathbb{R})} +\max_{\ell,j}\vert a_{j,\ell}(t)\vert +\max_{\ell,j}\vert a^1_{j,\ell}(t)\vert \right]e^{{-}\beta \min_{h}(v_{h}-v_{h+1})t+(y_{h}-y_{h+1})}.
\end{multline}
\par As a consequence, we can deduce from the decomposition \ref{princ11} and \eqref{forcingmainepde} that
\begin{align}\label{ode1}
    &\mathcal{S}\left(\partial_{t}\vec{\phi}(t,k)\right)(t,x)={-}iP_{\mathrm{cont}}(t)\left[Forc(t,x)\right],\\ \label{ode2}
    &i\dot a_{j,\ell}(t)-\lambda_{j,\ell}a_{j,\ell}(t)= b_{j,\ell}(t) \\ \label{ode3}
   & i\dot a^0_{j,\ell}(t)=a^2_{j,\ell}(t)+b^1_{j,\ell}(t), \quad\quad\text{if $1\leq j\leq K_{1,\ell},$}\\ \label{ode4}
   & i\dot a^1_{j,\ell}(t)= b^1_{j,\ell}(t), \quad\quad\text{if $K_{\ell,1}<j\leq K_{\ell,2}$}
\end{align}
for any subindices $j,\,\ell,\,k$ and any $t\geq 0.$ In particular, the differential equation \eqref{ode1} implies that 
\begin{equation}\label{integralT}
    \mathcal{S}(\vec{\phi}(t))(t,x)=\mathcal{S}(\vec{\phi}(0))(t,x)-i\int_{0}^{t}\mathcal{S}(t)\circ \mathcal{S}^{{-}1}(s)P_{\mathrm{cont}}(s)\left[Forc(s,x)\right]\,ds.
\end{equation}
Moreover, using the fundamental theorem of calculus, we can verify from \eqref{ode2} that $a_{j,\ell}(t)$ satisfies the following integral equation.
\begin{equation}\label{ODEofa}
    a_{j,\ell}(t)=e^{{-}i\lambda_{j,\ell}t}a_{j,\ell}(0)-i\int_{0}^{t}e^{{-}i\lambda_{j,\ell}(t-s)}b_{j,\ell}(s)\,ds.
\end{equation}
Consequently, we can verify from \eqref{ODEofa} for all $t\geq 0$ the following estimates for some constant $\beta>0.$
\begin{align}\label{stableupperbound}
   \max_{\I \lambda_{j,\ell} \leq  0}\norm{ a_{j,\ell}(s)}_{L^{\infty}_{s}[0,t]}\lesssim & \norm{\vec{\psi}(0,x)}_{L^{2}_{x}(\mathbb{R})} \\ &{+} e^{{-}\beta \min_{h} y_{h}-y_{h+1}} \max_{j,\ell,\I \lambda_{j,\ell}>0}\norm{ a_{j,\ell}(s)}_{L^{\infty}_{s}[0,t]}
   \\&{+} e^{{-}\beta \min_{h} y_{h}-y_{h+1}} \max_{\ell,1\leq j\leq K_{1,\ell}}\norm{\frac{ a^0_{j,\ell}(s)}{\langle s \rangle }}_{L^{\infty}_{s}[0,t]}\\
   & {+} e^{{-}\beta \min_{h} y_{h}-y_{h+1}} \max_{\ell, K_{\ell,1}<j\leq K_{\ell,2}}\norm{ a^1_{j,\ell}(s)}_{L^{\infty}_{s}[0,t]}\\
   &{+} e^{{-}\beta \min_{h} y_{h}-y_{h+1}} \max_{s\in [0,t]} \norm{\mathcal{S}(\vec{\phi}(s))(s,x)}_{L^{2}_{x}(\mathbb{R})}.
\end{align}
Using estimates \eqref{globalal} and inequality \eqref{CCC1Sphi} of Theorem \ref{decaySphi} on the identity \eqref{integralT}, we can deduce the existence of uniform constants $K,\,C>1$ satisfying
\begin{align*}\nonumber
    \norm{\mathcal{S}(\vec{\phi}(t))(t,x)}_{L^{2}_{x}(\mathbb{R})}
    \leq & \norm{\mathcal{S}(\vec{\phi}(0))(0,x)}_{L^{2}_{x}(\mathbb{R})}+C\int_{0}^{t}\norm{Forc(s,x)}_{L^{2}_{x}(\mathbb{R})}\,ds
    \\ 
    \lesssim & \norm{\mathcal{S}(\vec{\phi}(0))(0,x)}_{L^{2}_{x}(\mathbb{R})}\\&{+}\int_{0}^{t}e^{{-}\beta \min (v_{\ell}-v_{\ell+1})s+y_{\ell}-y_{\ell+1}}\left[1+\max_{\ell,j}\vert a^1_{j,\ell}(s) \vert+\norm{\mathcal{S}(\vec{\phi}(s))(s,x)}_{L^{2}_{x}(\mathbb{R})} \right]\,ds.
\end{align*}
Consequently, for all $t\geq 0,$
\begin{align}\label{upperboundS}
     \norm{\mathcal{S}(\vec{\phi}(t))(t,x)}_{L^{2}_{x}(\mathbb{R})}\lesssim &  \norm{\vec{\psi}(0,x)}_{L^{2}_{x}(\mathbb{R})} + e^{{-}\beta \min_{h} y_{h}-y_{h+1}} \max_{\ell, j, \I \lambda_{j,\ell}>0}\norm{ a_{j,\ell}(s)}_{L^{\infty}_{s}[0,t]}\\ \nonumber
   &{+} e^{{-}\beta \min_{h} y_{h}-y_{h+1}}  \max_{\ell, 1\leq j\leq K_{1,\ell}}\norm{\frac{ a^0_{j,\ell}(s)}{\langle s \rangle }}_{L^{\infty}_{s}[0,t]}\\ \nonumber
   & {+} e^{{-}\beta \min_{h} y_{h}-y_{h+1}} \max_{ \ell, K_{\ell,1}<j\leq K_{\ell,2}}\norm{ a^1_{j,\ell}(s)}_{L^{\infty}_{s}[0,t]}\\ \nonumber
   &{+} e^{{-}\beta \min_{h} y_{h}-y_{h+1}} \max_{s\in [0,t]} \norm{\mathcal{S}(\vec{\phi}(s))(s,x)}_{L^{2}_{x}(\mathbb{R})}.
\end{align}

\par Next, we can verify from \eqref{ode2} that
\begin{align}\label{a1formula}
    a^0_{j,\ell}(t)= & a^0_{j,\ell}(0)-ia^2_{j,\ell}(0)t-\int_{0}^{t}b^1_{j,\ell}(s)\,ds-\int_{0}^{t}\int_{0}^{s}\dot{a}^2_{j,\ell}(s_{1})\,ds_{1}\,ds,,\quad 1\leq j\leq K_{\ell,1},\\ \label{a2formula}
    a^1_{j,\ell}(t)= & a^1_{j,\ell}(0)-i\int_{0}^{t}b^1_{j,\ell}(s)\,ds,\quad K_{\ell,1}<j\leq K_{\ell,2}.
\end{align}
Consequently, using the definition of $b^1_{j,\ell}$ in \eqref{bjformula} and estimate \eqref{l2forcupp}, we can verify the following inequality below for all $t\geq 0.$
\begin{align}\label{kernelajupp}
   \max_{K_{\ell,1}<j\leq K_{\ell,2}}\left\vert  a^1_{j,\ell}(t) \right\vert+\max_{1\leq j\leq K_{\ell,1}}\left\vert \frac{a^0_{j,\ell}(t)}{\langle t \rangle} \right\vert\lesssim &  \norm{\vec{\psi}(0,x)}_{L^{2}_{x}(\mathbb{R})} \\
   &+ e^{{-}\beta \min_{h} y_{h}-y_{h+1}} \max_{\ell,j, \I \lambda_{j,\ell}>0}\norm{ a_{j,\ell}(s)}_{L^{\infty}_{s}[0,t]}\\ \nonumber
   &{+} e^{{-}\beta \min_{h} y_{h}-y_{h+1}}  \max_{\ell, 1\leq j\leq K_{1,\ell}}\norm{\frac{ a^0_{j,\ell}(s)}{\langle s \rangle }}_{L^{\infty}_{s}[0,t]}\\ \nonumber
   & {+} e^{{-}\beta \min_{h} y_{h}-y_{h+1}} \max_{\ell, K_{\ell,1}<j\leq K_{\ell,2}}\norm{ a^1_{j,\ell}(s)}_{L^{\infty}_{s}[0,t]}\\ \nonumber
   &{+} e^{{-}\beta \min_{h} y_{h}-y_{h+1}} \max_{s\in [0,t]} \norm{\mathcal{S}(\vec{\phi}(s))(s,x)}_{L^{2}_{x}(\mathbb{R})}. 
\end{align}
\par As a consequence, using the estimates \eqref{globalal}, \eqref{stableupperbound}, \eqref{upperboundS} and \eqref{kernelajupp}, if $\min_{h}y_{h}-y_{h+1}\gg 1,$ we obtain the existence of a constant $C>1$ satisfying
\begin{equation}\label{globalupperbound}
    \norm{\mathcal{S}(\vec{\phi}(t))(t)}_{L^{2}_{x}(\mathbb{R})}+\max_{j, \ell}\vert a_{j,\ell}(t)\vert+  \max_{\ell, K_{\ell,1}<j\leq K_{\ell,2}}\left\vert  a^1_{j,\ell}(t) \right\vert+\max_{\ell, 1\leq j\leq K_{\ell,1}}\left\vert \frac{a^0_{j,\ell}(t)}{\langle t \rangle} \right\vert\leq C \text{, for all $t\geq 0.$}
\end{equation}
In particular, using Lemma $4.6$ of \cite{KriegerSchlag} in the equation \eqref{ODEofa}, and estimate \eqref{globalupperbound}, we deduce for all $t\geq 0$ that
\begin{equation}\label{unstablemodes}
    a_{j,\ell}(t)=i\int_{t}^{{+}\infty}e^{{-}i\lambda_{\ell,n}(t-s)}b_{j,\ell}(s)\,ds \text{, when $\I \lambda_{j,\ell}>0.$}
\end{equation}
Otherwise, $\lim_{t\to {+}\infty} \vert a_{j,\ell}(t) \vert ={+}\infty.$

\par Moreover, using the identity \eqref{unstablemodes} and estimate \eqref{l2forcupp}, we can deduce the following estimate for all $t\geq 0.$
\begin{align}\label{unstableimproved}
    \max_{\I \lambda_{j,\ell}>0} \norm{ a_{j,\ell}(s)}_{L^{\infty}_{s}[t,{+}\infty)}\lesssim &
   e^{{-}\beta (\min_{h} y_{h}-y_{h+1}+t)} \max_{\ell, 1\leq j\leq K_{\ell,1}}\norm{\frac{ a^0_{j,\ell}(s)}{\langle s \rangle }}_{L^{\infty}_{s}[t,{+}\infty)}\\ \nonumber
   & {+} e^{{-}\beta (\min_{h} y_{h}-y_{h+1}+t)} \max_{\ell, K_{\ell,1}<j\leq K_{\ell,2}}\norm{ a^1_{j,\ell}(s)}_{L^{\infty}_{s}[t,{+}\infty)}\\ \nonumber
   &{+} e^{{-}\beta (\min_{h} y_{h}-y_{h+1}+t)} \max_{s\in [t,{+}\infty)} \norm{\mathcal{S}(\vec{\phi}(s))(s,x)}_{L^{2}_{x}(\mathbb{R})}.
\end{align}
Therefore, using \eqref{unstableimproved}, \eqref{stableupperbound}    \eqref{upperboundS} and \eqref{kernelajupp}, we can deduce the existence of a constant $C>1$ satisfying
\begin{equation}\label{phi0upp}
     \norm{\mathcal{S}(\vec{\phi}(t))(t)}_{L^{2}_{x}(\mathbb{R})}+\max_{j, \ell}\vert a_{j,\ell}(t)\vert+  \max_{\ell, K_{\ell,1}<j\leq K_{\ell,2}}\left\vert  a^1_{j,\ell}(t) \right\vert+\max_{\ell, 1\leq j\leq K_{\ell,1}}\left\vert \frac{a^0_{j,\ell}(t)}{\langle t \rangle} \right\vert\leq C\norm{\vec{\psi}(0,x)}_{L^{2}_{x}(\mathbb{R})} \text{, for all $t\geq 0.$}
\end{equation}
Note that we can verify from \eqref{upperboundPDD} and \eqref{phi0upp} that there exists a $\beta>0$ satisfying
\begin{equation*}
    \max_{j,\ell}\int_{t}^{{+}\infty}\vert b^1_{j,\ell}(s)\vert\,ds\lesssim e^{{-}\beta \min_{\ell} (v_{\ell}-v_{\ell+1})t+y_{\ell}-y_{\ell+1} }\norm{\vec{\psi}(0,x)}_{L^{2}_{x}(\mathbb{R})}.
\end{equation*}
As a consequence, \eqref{ode2} and \eqref{ODEofa} implies that if $\lambda_{\ell,n}\in\mathbb{R},$ then there exists a unique complex constant $a_{j,\ell,\infty}$ satisfying
\begin{equation*}
    \vert a_{j,\ell}(t)-e^{i\lambda_{j,\ell}t}a_{j,\ell,\infty} \vert\lesssim e^{{-}\beta \min_{\ell} (v_{\ell}-v_{\ell+1})t+y_{\ell}-y_{\ell+1} }\norm{\vec{\psi}(0,x)}_{L^{2}_{x}(\mathbb{R})} \text{, for all $t\geq 0,$}
\end{equation*}
from which we deduce the Property $\mathrm{(P3)}$ of Theorem \ref{stablecase}.
\par Next, using \eqref{phi0upp}, we can verify when $\quad 1\leq j\leq K_{\ell,1}$ that
\begin{align*}
\left\vert \int_{t}^{{+}\infty}b^1_{j,\ell}(s)\,ds\right\vert +\left\vert \int_{t}^{{+}\infty}\int_{s}^{{+}\infty}\dot{a}^2_{j,\ell}(s_{1})\,ds_{1}\,ds\right\vert\lesssim e^{{-}\beta \min_{\ell} (v_{\ell}-v_{\ell+1})t+y_{\ell}-y_{\ell+1} }\norm{\vec{\psi}(0,x)}_{L^{2}_{x}(\mathbb{R})} \text{, for all $t\geq 0,$} 
\end{align*}
and when $K_{\ell,1}<j\leq K_{\ell,2}$ that
\begin{align*}
\int_{t}^{{+}\infty}\left\vert b^1_{j,\ell}(s)\right\vert\,ds \lesssim e^{{-}\beta \min_{\ell} (v_{\ell}-v_{\ell+1})t+y_{\ell}-y_{\ell+1} }\norm{\vec{\psi}(0,x)}_{L^{2}_{x}(\mathbb{R})}\text{, for all $t\geq 0.$}
\end{align*}
 Consequently, we can deduce using estimates \eqref{globalal} and \eqref{upperboundPDD} on the identities \eqref{a1formula} and \eqref{a2formula} that there exist complex constants $a_{\ell,\infty},\,c_{\ell,\infty}$ and a real constant $C>1$ satisfying for all $t\geq 0$
\begin{align}\label{kernelsize}
   \vert a^1_{j,\ell}(t)-c_{j,\ell,\infty}\vert\leq C e^{{-}\beta\min_{h}(y_{h}-y_{h+1}+(v_{h}-v_{h+1})t)} \norm{\vec{\psi}(0,x)}_{L^{2}_{x}(\mathbb{R})},\quad K_{1,\ell}<j\leq K_{\ell,2}\\
     \vert a^0_{j,\ell}(t)-a_{j,\ell,\infty}-c_{j,\ell,\infty}t \vert\leq C e^{{-}\beta\min_{h}(y_{h}-y_{h+1}+(v_{h}-v_{h+1})t)} \norm{\vec{\psi}(0,x)}_{L^{2}_{x}(\mathbb{R})},\quad \quad 1\leq j\leq K_{\ell,1}.
\end{align}
 
\par Similarly,  we can verify when $\lambda_{\ell,n}\in\mathbb{R}$ that there exists $a_{j,\ell,\infty}\in\mathbb{C}$ satisfying 
\begin{equation*}
    \vert a_{j,\ell}(t)- e^{{-}i\lambda_{j,\ell} t}a_{j,\ell,\infty} \vert \leq C e^{{-}\beta\min_{h}(y_{h}-y_{h+1}+(v_{h}-v_{h+1})t)} \norm{\vec{\psi}(0,x)}_{L^{2}_{x}(\mathbb{R})} \text{, for all $t\geq 0,$}
\end{equation*}
for a uniform constant $C>1.$

Finally, since $P_{\mathrm{cont}}(t)\in \Ra \mathcal{S}(t)$ and $g(t)=P_{\mathrm{cont}}(t)$ is continuous in the operator norm on the set $[0,{+}\infty),$ Theorem \ref{CCC1} implies for any $T>0$ that there exists a unique function $\vec{f}(T,k)\in L^{2}_{k}(\mathbb{R})$ satisfying
\begin{equation*}
    \mathcal{S}(\vec{f}(T,k))(t)=\mathcal{S}(\vec{\phi}(0))(t,x)-i\int_{0}^{t}\mathcal{S}(t)\circ \mathcal{S}^{{-}1}(s)P_{\mathrm{cont}}(s)\left[Forc(s,x)\right]\,ds.
\end{equation*}
Moreover, Theorem \ref{decaySphi} and estimates \eqref{l2forcupp}, \eqref{phi0upp} imply that
\begin{align*}
\lim_{T\to{+}\infty}\norm{\int_{t}^{T}\mathcal{S}(t)\circ \mathcal{S}^{{-}1}(s)P_{\mathrm{cont}}(s)\left[Forc(s,x)\right]\,ds}_{L^{2}_{x}(\mathbb{R})}
\lesssim & e^{{-}\beta\min_{h}(y_{h}-y_{h+1}+(v_{h}-v_{h+1})t)}\norm{\vec{\psi}(0,x)}_{L^{2}_{x}(\mathbb{R})}.
\end{align*}
Therefore, can deduce using estimate \eqref{CCC1Sphi} of Theorem \eqref{decaySphi} $\vec{f}(T,k)$ is a Cauchy sequence in $L^{2}_{k}(\mathbb{K})$ having a limit $\vec{\phi}_{\infty}(k)$ belonging to the domain of $\mathcal{T}.$ Moreover, $\vec{f}_{\infty}(k)$ satisfies for all $t\geq 0$
\begin{equation*}
    \norm{\mathcal{S}(\vec{\phi}_{\infty}(k))(t,x)-\mathcal{S}(\vec{\phi}(t,k))(t,x)}\lesssim e^{{-}\beta\min_{h}(y_{h}-y_{h+1}+(v_{h}-v_{h+1})t)}\norm{\vec{\psi}(0,x)}_{L^{2}_{x}(\mathbb{R})}. 
\end{equation*}
\par Furthermore, using Theorem \ref{decaySphi}, we can verify the following inequalities.
\begin{align}\label{ineqForc1}
    \norm{\mathcal{S}(t)\circ\mathcal{S}^{{-}1}(s)P_{\mathrm{cont}}(s)Forc(s)}_{L^{\infty}_{x}(\mathbb{R})}\lesssim & \norm{\mathcal{S}(0)\circ\mathcal{S}^{{-}1}(s)P_{\mathrm{cont}}(s)\left[Forc(s,x)\right]}_{H^{1}_{x}(\mathbb{R})},\\ \label{ineqForc2}
     \norm{\frac{\mathcal{S}(t)\circ\mathcal{S}^{{-}1}(s)P_{\mathrm{cont}}(s)Forc(s)}{\langle x-v_{\ell}t-y_{\ell}\rangle}}_{L^{\infty}_{x}(\mathbb{R})}\lesssim &\norm{\mathcal{S}(0)\circ\mathcal{S}^{{-}1}(s)P_{\mathrm{cont}}(s)\left[Forc(s,x)\right]}_{H^{1}_{x}(\mathbb{R})}.
\end{align}
Consequently, we can deduce 
from Proposition \ref{Coercivityproperty} and estimate \eqref{l2forcupp} that
\begin{align*}
   \norm{\mathcal{S}(0)\circ \mathcal{S}^{{-}1}(s)P_{\mathrm{cont}}(s)Forc(s)}_{H^{1}_{x}(\mathbb{R})}\lesssim &  \norm{P_{\mathrm{cont}}(s)Forc(s)}_{H^{1}_{x}(\mathbb{R})}
   \\
   \lesssim & e^{{-}\beta\min_{h}(y_{h}-y_{h+1}+(v_{h}-v_{h+1})s)}\norm{\vec{\psi}(0,x)}_{L^{2}_{x}(\mathbb{R})}.
\end{align*}
Consequently, we  conclude the proof of Theorem \ref{stablecase} using the fundamental theorem of calculus. 
\end{proof}

\subsection{Decay estimates}
In this section, we prove the dispersive decay estimates for solutions in the scattering space in the sense of Definition \eqref{def:scatter}.

\begin{proof}[Proof of Theorem \ref{decaytheoscatter}] 
 Clearly, if $\vec{\psi}$ is in the scattering space in the sense of Definition \eqref{def:scatter}, then it is in the stable space in the sense of Definition \eqref{def:stable}. Applying Theorem \ref{stablecase} and using asymptotic conditions \eqref{asyorth},  all $a_{j,\ell}$ and $a^1_{j,\ell}$ from $(\mathrm{P}1),\,(\mathrm{P}2),\,(\mathrm{P}3),\,(\mathrm{P}4)$ will decay exponentially. Since all eigenfunctions are exponentially localized, it follows that the decay estimates for $\vec{\psi}(t)$ are fully determined by $\mathcal{S}(\vec{\phi}(t))(t,x)$. Then from $(\mathrm{P}6)$ it is reduced to the study of dispersive properties of $\mathcal{S}(\vec{\phi}_\infty)(t,x)$ which were established in Theorem \ref{decaySphi}.
\par More precisely, using Theorem \ref{stablecase}, we can verify that any solution $\vec{\psi}(t,x)\in H^{1}_{x}(\mathbb{R})$ in the scattering space is of the form
\begin{equation*}
  \vec{\psi}(t,x)=\mathcal{S}(\vec{\phi}_{\infty})(t,x)+\vec{r}(t,x),  
\end{equation*}
such that for all $t\geq 0$
\begin{equation*}
    \norm{\vec{r}(t,x)}_{H^{1}_{x}(\mathbb
    {R})}\lesssim e^{{-}\beta (c t+\min_{\ell}y_{\ell}-y_{\ell+1})}\norm{\mathcal{S}(\vec{\phi}_{\infty})(t,x)}_{L^{2}_{x}(\mathbb{R})}\sim e^{{-}\beta (c t+\min_{\ell}y_{\ell}-y_{\ell+1})}\norm{\vec{\psi}(0,x)}_{L^{2}_{x}(\mathbb{R})} ,
\end{equation*}
 for a $c>0$ depending only on the sets $\{\I \lambda_{j,\ell}\},\,\{\alpha_{\ell}\}$ and the real value $\min_{\ell}v_{\ell}-v_{\ell+1}>0.$ As a consequence, we can verify the following estimates.
 \begin{align*}
    \norm{\vec{\psi}(t,x)}_{L^{\infty}_{x}(\mathbb{R})}\lesssim &  \norm{\mathcal{S}(\vec{\phi}_{\infty})(t,x)}_{L^{\infty}_{x}(\mathbb{R})}+O\left(  e^{{-}\beta (c t+\min_{\ell}y_{\ell}-y_{\ell+1})}\norm{\vec{\psi}(0,x)}_{L^{2}_{x}(\mathbb{R})}\right),\\
    \norm{\frac{\vec{\psi}(t,x)}{\langle x-y_{\ell}-v_{\ell}t\rangle}}_{L^{\infty}_{x}(\mathbb{R})}\lesssim & \norm{\frac{\mathcal{S}(\vec{\phi}_{\infty})(t,x)}{\langle x-y_{\ell}-v_{\ell}t\rangle}}_{L^{\infty}_{x}(\mathbb{R})}+ O\left(  e^{{-}\beta (c t+\min_{\ell}y_{\ell}-y_{\ell+1})}\norm{\vec{\psi}(0,x)}_{L^{2}_{x}(\mathbb{R})}\right),
    \\
\norm{\frac{\partial_{x}\vec{\psi}(t,x)}{\langle x-y_{\ell}-v_{\ell}t\rangle^{1+\frac{p^{*}-2}{2p^{*}}+\alpha}}}_{L^{2}_{x}(\mathbb{R})}\lesssim & \norm{\frac{\partial_{x}\mathcal{S}\left(\vec{\phi}_{\infty}\right)(t,x)}{\langle x-y_{\ell}-v_{\ell}t\rangle^{1+\frac{p^{*}-2}{2p^{*}}+\alpha}}}_{L^{2}_{x}(\mathbb{R})}\\&{+}O\left(  e^{{-}\beta (c t+\min_{\ell}y_{\ell}-y_{\ell+1})}\norm{\vec{\psi}(0,x)}_{L^{2}_{x}(\mathbb{R})}\right).
 \end{align*}
 \par Therefore, Theorem \ref{decaySphi} implies that there exists a uniform constant $K>1$ for which $\vec{\psi}(t,x)$ satisfies the following decay estimates for all $t\geq s\geq 0.$
  \begin{align}\label{bshit1}
  \norm{\vec{\psi}(t,x)}_{L^{\infty}_{x}(\mathbb{R})}\leq &\frac{K}{(t-s)^{\frac{1}{2}}}\norm{\mathcal{S}(\vec{\phi}_{\infty})(s,x)}_{L^{1}_{x}(\mathbb{R})}\\&{+}\frac{K e^{{-}\beta\min_{\ell}(y_{\ell}-y_{\ell+1}+cs)}}{(t-s)^{\frac{1}{2}}}\norm{\mathcal{S}(\vec{\phi}_{\infty})(s,x)}_{L^{2}_{x}(\mathbb{R})},\\ \label{bshit2}
\norm{\frac{\vec{\psi}(t,x)}{(1+\vert x-y_{\ell}-v_{\ell}t\vert)}}_{L^{\infty}_{x}(\mathbb{R})}\leq & \frac{K (s+y_{1}-y_{m})}{(t-s)^{\frac{3}{2}}} \norm{\mathcal{S}(\vec{\phi}_{\infty})(s,x)}_{L^{1}_{x}(\mathbb{R})}\\  
&{+}\frac{K}{(t-s)^{\frac{3}{2}}}\max_{\ell}\norm{(1+\vert x-y_{\ell}-v_{\ell}s\vert )\chi_{\ell}(s,x)\mathcal{S}(\vec{\phi}_{\infty})(s,x)}_{L^{1}_{x}(\mathbb{R})}\\
&{+}K e^{{-}\beta\min_{\ell}(y_{\ell}-y_{\ell+1}+ct)}\norm{\mathcal{S}(\vec{\phi}_{\infty})(s,x)}_{L^{2}_{x}(\mathbb{R})},\\ \label{bshit3} 
\max_{\ell}\norm{\frac{\partial_{x}\vec{\psi}(t)}{\langle x-v_{\ell}t-y_{\ell} \rangle^{1+\frac{p^{*}-2}{2p^{*}}+\alpha}}}_{L^{2}_{x}(\mathbb{R})}
\leq & \frac{K\max_{\ell}\norm{\langle x-y_{\ell}-v_{\ell}s\rangle\chi_{\ell}(s,x)\langle \partial_{x}\rangle\mathcal{S}(\vec{\phi}_{\infty})(s,x)}_{L^{1}_{x}(\mathbb{R})}^{\frac{2-p}{p}}\norm{\mathcal{S}(\vec{\phi}_{\infty})(s,x)}_{H^{1}_{x}(\mathbb{R})}^{\frac{2(p-1)}{p}}}{(t-s)^{\frac{3}{2}(\frac{1}{p}-\frac{1}{p^{*}})}} \\ \nonumber
 &{+}K\frac{(s+y_{1}-y_{m})}{(t-s)^{\frac{3}{2}(\frac{1}{p}-\frac{1}{p^{*}})}}\norm{\mathcal{S}(\vec{\phi}_{\infty})(s,x)}_{W^{1,1}_{x}(\mathbb{R})}^{\frac{2-p}{p}}\norm{\mathcal{S}(\vec{\phi}_{\infty})(s,x)}_{H^{1}_{x}(\mathbb{R})}^{\frac{2(p-1)}{p}}
 \\ \nonumber
 &{+}Ke^{{-}\beta\min_{\ell}(y_{\ell}-y_{\ell+1}+ct)}\norm{\mathcal{S}(\vec{\phi}_{\infty})(s,x)}_{H^{1}_{x}(\mathbb{R})}.
\end{align}
 Moreover, using properties $\mathrm{(P1)},\,\mathrm{(P2)},\,\mathrm{(P3)}$ and $\mathrm{(P4)}$ of Theorem \ref{stablecase}, we can verify when $\vec{\psi}(t,x)$ is in the scattering space that the remainder $\vec{r}(t,x)=\vec{\psi}(t,x)-\mathcal{S}(\vec{\phi}_{\infty})(t,x)$ satisfies the following decay estimates for all $s\geq 0$
 \begin{align*}
     \max_{\ell}\norm{\langle x-y_{\ell}-v_{\ell}s\rangle \chi_{\ell}(s,x)\langle \partial_{x}\rangle\vec{r}(s,x)}_{L^{1}_{x}(\mathbb{R})}+\norm{\vec{r}(s,x)}_{H^{1}_{x}(\mathbb{R})}\lesssim e^{{-}\beta\min_{\ell}(y_{\ell}-y_{\ell+1}+ct)}\norm{\mathcal{S}(\vec{\phi}_{\infty})(s,x)}_{L^{2}_{x}(\mathbb{R})}.
 \end{align*}
 In conclusion, we can obtain from the estimates \eqref{bshit1}, \eqref{bshit2} and \eqref{bshit3} that $\vec{\psi}(t,x)$ satisfies \eqref{linfty}, \eqref{weightedlinfty} and \eqref{weightedderivative}, when $t\geq s\geq 0.$
 
\end{proof}

\section{Existence of wave operators: proof of Proposition \ref{remark4}}\label{prooftunst2}

 Concerning the proof of Proposition \ref{remark4}, we will construct a solution of \eqref{p} of the form 
\begin{equation*}
\vec{\psi}(t)=\mathcal{S}(\vec{\phi})(t,x)+\vec{r}(t,x),    
\end{equation*}
for a constant function $\vec{\phi}\in L^{2}_{k}(\mathbb{R},\mathbb{C}^{2})$ belonging to the domain of $\mathcal{S}.$ Moreover, using Proposition \ref{princ}, we can denote $\vec{r}(t,x)$ uniquely for any $t\geq 0$ as
\begin{align*}
 \vec{r}(t,x)=& \mathcal{S}(\varphi_{1}(t,k))(t,x)\\&
+\sum_{\ell=1}^{m}\sum_{j=1}^{2N_\ell+2M_\ell}d_{j,\ell}(t)e^{i\sigma_{3}\left(\frac{v_{\ell}x}{2}-\frac{v_{\ell}^{2}t}{4}+\omega_{\ell}t+\gamma_{\ell}\right)}\vec{Z}_{j,\ell}(x-v_{\ell}t-y_{\ell})\\
&
%+\sum_{\ell=1}^{m}\sum_{j=1}^{K_{\ell,1}}e^{i\sigma_{3}\left(\frac{v_{\ell}x}{2}-\frac{v_{\ell}^{2}t}{4}+\omega_{\ell}t+\gamma_{\ell}\right)}\vec{Z}^0_{j,\ell}(x-v_{\ell}t-y_{\ell})
+\sum_{\ell=1}^{m}\sum_{j=1}^{K_{\ell,2}}d_{j,\ell}^{1}(t)e^{i\sigma_{3}\left(\frac{v_{\ell}x}{2}-\frac{v_{\ell}^{2}t}{4}+\omega_{\ell}t+\gamma_{\ell}\right)}\vec{Z}^1_{j,\ell}(x-v_{\ell}t-y_{\ell})
 %{+}\sum_{\ell=1}^{m}\sum_{\lambda_{\ell,n}\in\sigma_{d}(\mathcal{H}_{\ell})}\sum_{j=1}^{\dim \ker \left(\mathcal{H}_{\ell}-\lambda_{\ell_{1},k}\right)^{2}}e^{i(\theta_{\ell}(t,x)+\omega t)\sigma_{3}}a_{j,\ell,n}(t)\mathfrak{v}_{j,\omega_{\ell_{1}},\lambda_{\ell_{1},k}}(x-v_{\ell}t-y_{\ell}),   
\end{align*}
for a function $\varphi_{1}(t,\Diamond)\in L^{2}_{k}(\mathbb{R},\mathbb{C}^{2})$ belonging to the domain of $\mathcal{S},$ and $d_{j,\ell}(t),\,d^1_{j,\ell}(t)\in\mathbb{C}$ for all $t\geq 0.$ Consequently, similarly to the proof of Theorem \ref{stablecase} in the previous section, we can verify using Lemma $7.2$ from \cite{dispa} that
\begin{equation*}
    i \partial_t \vec{r}(t)+\left(\begin{array}{cc}
\partial_x^2 & 0 \\
0 & {-}\partial_x^2
\end{array}\right) \vec{r}(t)+\sum_{j=1}^m V_j\left(t \right) \vec{r}(t)=Forc(t,x),
\end{equation*}
for a complex-valued function $Forc(t,x)$ satisfying
\begin{equation*}
    \norm{Forc(t,x)}_{L^{2}_{x}(\mathbb{R})}\lesssim e^{{-}\beta \min_{h}(v_{h}-v_{h+1})t+(y_{h}-y_{h+1})},
\end{equation*}
for a constant $\beta>0.$ Using Proposition \ref{princ} again, we can denote $Forc(t,x)$ by a finite sum of the form \eqref{forcdecomposition}. 

%\cgc{in decomposition referred here coefficients are given by $b$. but $b$ were used in the decomposition of $r$.}
\par Using the same notations as in \eqref{eq:a0}, \eqref{a2},  the condition 
\begin{equation}
   %\lim_{t\to{+}\infty} \max_{\ell,j} \vert d_{j,\ell}(t)\vert=
   \lim_{t\to{+}\infty} \max_{\ell,j} \vert d^{1}_{j,\ell}(t)\vert=0,
\end{equation}

 implies that $d^1_{j,\ell}$ satisfies the following identities for all $t\geq 0.$
\begin{align}\label{a1}
    d^1_{j,\ell}(t)=&i\int_{t}^{{+}\infty}b_{j,\ell}(s)\,ds+\int_{t}^{{+}\infty}\int_{s}^{{+}\infty} \dot{d}^2_{j,\ell}(t)(s_{1})\,ds_{1}\,ds,\,\,1\leq j\leq K_{\ell,1}\\ \label{a2}
    d^1_{j,\ell}(t)=&i\int_{t}^{{+}\infty}b^1_{j,\ell}(s)\,ds,\,K_{\ell,1}<j\leq K_{\ell,2}.
\end{align}
Since \eqref{globalal}, \eqref{upperboundPDD} imply that $\vert b^1_{j,\ell}(s)\vert$ has exponential decay rate, it is not difficult to check that the unique bounded function $d^1_{j,\ell}(t)$ satisfying \eqref{ode3}, \eqref{ode4} and
\begin{equation*}
    \lim_{t\to{+}\infty} d^1_{j,\ell}(t)=0
\end{equation*}
is the one denoted in \eqref{a1} and \eqref{a2}.

\par Consequently, using the Banach fixed-point theorem, we can verify the existence and uniqueness of a unique function $\vec{r}(t,x)\in L^{2}_{x}(\mathbb{R})$ satisfying
\begin{equation*}
    \sup_{t\geq 0}e^{\epsilon t}\norm{\vec{r}(t)}_{L^{2}_{x}(\mathbb{R})}<{+}\infty
\end{equation*}
for an small $\epsilon\in (0,1)$ solving all the integral equations \eqref{a1}, \eqref{a2},
%\cgc{notations are not consisten here. Also how about $\I \lambda_{j,\ell}=0?$}
and 
\begin{align*}
    d_{j,\ell}(t)=&{-}i\int_{0}^{t}e^{{-}i\lambda_{j,\ell}(t-s)}b_{j,\ell}(s)\,ds \text{, for  $\I \lambda_{\ell,n}<0,$}\\
     \mathcal{S}(\vec{\varphi_1}(t))(t,x)=&i\int_{t}^{{+}\infty}\mathcal{S}(t)\circ \mathcal{S}^{{-}1}(s)P_{\mathrm{cont}}(s)\left[Forc(s,x)\right]\,ds,\\
     d_{j,\ell}(t)=&{-}i\int_{{+}\infty}^{t} e^{{-}i\lambda_{\ell,n}(t-s)}b_{j,\ell}(s)\,ds \text{, if $\I \lambda_{\ell,n}\geq 0.$}
\end{align*}
As a consequence, there is $\epsilon\in (0,1)$ such that $\mathcal{T}(\vec{\phi})(t,x)=\mathcal{S}(\vec{\phi})(t,x)+\vec{r}(t,x)$ is a strong solution of \eqref{p} that satisfies Proposition \ref{remark4}.
\par The proof of the existence of solutions $\mathfrak{G}_{\ell}(\vec{Z}_{j,\lambda_{\ell}})(t,x)$ satisfying \eqref{G1solll} or \eqref{G2soll} in Proposition \ref{remark4} is completely analogous.

\bibliographystyle{plain}
\bibliography{ref}
\appendix

\end{document}